\documentclass{mcom-l}
\usepackage{amssymb}
\usepackage{mathtools}
\usepackage{bbm}
\usepackage{enumerate}
\usepackage{tikz}
\usepackage{xcolor}
\usepackage{enumitem}
\usepackage{stmaryrd}
\usepackage{subcaption}
\usepackage{caption}
\usepackage{booktabs}
\usepackage{siunitx,microtype}
\usepackage{comment}

\definecolor{mutedCyan}{RGB}{20,110,150}
\usepackage{hyperref}
\hypersetup{
    colorlinks=true,
    citecolor=mutedCyan,
    linkcolor=black,
}
\providecommand{\doi}[1]{DOI \href{https://doi.org/#1}{\begingroup\urlstyle{same}\nolinkurl{#1}\endgroup}}

\let\originalleft\left
\let\originalright\right
\renewcommand{\left}{\mathopen{}\mathclose\bgroup\originalleft}
\renewcommand{\right}{\aftergroup\egroup\originalright}

\newcommand{\ve}{\varepsilon}

\newcommand\norm[1]{\left\lVert#1\right\rVert}
\newcommand\absv[1]{\left\lvert#1\right\rvert}
\renewcommand\brack[1]{\left (#1\right)}
\newcommand\abrack[1]{\left \langle #1 \right \rangle}
\newcommand\sbrack[1]{\left [ #1 \right ]}
\newcommand\cbrack[1]{\left \{ #1 \right \}}

\newcommand{\dd}{\mathop{}\!\mathrm{d}}

\newcommand{\bbE}{\mathbb{E}}
\newcommand\E[1]{\bbE\sbrack{#1}}

\newcommand{\btd}{\mathbb{T}^d}

\newcommand{\bx}{\boldsymbol{x}}
\newcommand{\bv}{\boldsymbol{v}}

\newcommand{\bB}{\boldsymbol{B}}
\newcommand{\bj}{\boldsymbol{j}}
\newcommand{\by}{\boldsymbol{y}}
\newcommand{\bxi}{\boldsymbol{\xi}}
\newcommand{\bn}{\boldsymbol{n}}

\newcommand{\boldf}{\boldsymbol{f}}
\newcommand{\bvphi}{\boldsymbol{\varphi}}
\newcommand{\bX}{\boldsymbol{X}}
\newcommand{\bu}{\boldsymbol{u}}

\newcommand{\Th}{\mathcal{T}_h}
\newcommand{\Fk}{\mathcal{F}_K}

\newcommand{\Fh}{\mathcal{F}_h}

\newcommand{\rhomin}{\rho_{\mathrm{min}}}
\newcommand{\rhomax}{\rho_{\mathrm{max}}}

\newcommand{\scalarpolys}{P_p(K)}
\newcommand{\vectorpolys}{\boldsymbol{P}_p^d(K)}
\newcommand{\scalardG}{V_p}
\newcommand{\vectordG}{\boldsymbol{V}_p^d}

\newcommand{\avg}[1]{\left\{\!\left\{#1\right\}\!\right\}}
\newcommand{\jump}[1]{\left\llbracket #1 \right\rrbracket}

\newcommand{\locallift}{\boldsymbol{\mathcal{L}}_F}
\newcommand{\globallift}{\boldsymbol{\mathcal{L}}_h}

\newcommand{\dgrad}{G_h}

\newcommand{\mfl}{\overline{\rho}}
\newcommand{\sdmfl}{\overline{\rho}_h}

\newcommand{\ritz}{R_h}

\theoremstyle{plain}
\newtheorem{theorem}{Theorem}[section]
\newtheorem{proposition}[theorem]{Proposition}
\newtheorem{lemma}[theorem]{Lemma}
\newtheorem{corollary}[theorem]{Corollary}

\theoremstyle{definition}
\newtheorem{definition}[theorem]{Definition}

\newtheorem{assumptionalt}{Assumption}
\newenvironment{assumptionp}[1]{
  \renewcommand\theassumptionalt{#1}
  \assumptionalt
}{\endassumptionalt}

\newtheorem*{conditionalassumption}{Conditional Assumption}

\theoremstyle{remark}
\newtheorem{remark}[theorem]{Remark}

\numberwithin{equation}{section}

\begin{document}

\title[Discontinuous Galerkin for Dean--Kawasaki]{A discontinuous Galerkin approximation of the Dean--Kawasaki equation}
\author{Kamran Arora}
\address{Kamran Arora, \newline Department of Mathematics, University of Bath BA2 7AY, UK}
\email{ka679@bath.ac.uk}
\author{Federico Cornalba}
\address{Federico Cornalba, \newline Department of Mathematics, University of Bath BA2 7AY, UK}
\email{fc402@bath.ac.uk}
\author{Tony Shardlow}
\address{Tony Shardlow, \newline Department of Mathematics, University of Bath BA2 7AY, UK}
\email{t.shardlow@bath.ac.uk}


\date{}


\begin{abstract}
    We introduce and analyse an arbitrary order spatial discontinuous Galerkin (dG) method for the Dean--Kawasaki equation, a highly singular SPDE modelling density fluctuations of $N$ diffusing particles in the regime of large particle number $N \gg 1$. Our starting point is a general procedure for discretising multiplicative, divergence-form noise on finite element spaces whilst preserving its cross-variation structure at the discrete level; the construction is explicit, elementwise, applies to continuous and discontinuous spaces alike, and extends to general mobilities. Using it, we prove weak error estimates of order $O(h^p)$ between fluctuations of the semi-discrete scheme and those of the underlying particle system, together with a correction that is exponentially small in the scaling regime $Nh^d \gg 1$ and arises because the scheme does not preserve positivity. The resulting method is locally and globally conservative and applies on unstructured simplicial meshes. 
    Quantitative numerical experiments support the analysis, and further experiments illustrate the method beyond the scope of the theory: indicator function observables, external and interaction potentials, convection-dominated regimes and reflecting boundary conditions.
\end{abstract}

\maketitle

\section{Introduction}

The Dean--Kawasaki equation is a pivotal SPDE from the theory of fluctuating hydrodynamics proposed by Dean~\cite{dean1996langevin} and Kawasaki~\cite{kawasaki1998microscopic} as a model for density fluctuations in large but finite-sized systems of (interacting) particles undergoing diffusion. Take the simplest such system, that of $N$ standard, independent Brownian motions $\{\bB_i(t)\}_{i=1}^N$ on the torus, and let
\begin{equation}\label{eq: empirical-measure}
    \mu_t^N \coloneqq N^{-1}\sum_{i=1}^N \delta_{\bB_i(t)}
\end{equation}
denote the associated empirical measure. Applying It\^o's formula to $\abrack{\mu_t^N, \varphi}$ for smooth $\varphi$, and representing the resulting martingale as an integral against space-time white noise, leads formally to a closed equation for the density $\rho$ of $\mu^N$,
\begin{equation}\label{eq: Dean--Kawasaki}
    \partial_t \rho = \frac{1}{2} \Delta \rho + \frac{1}{\sqrt{N}}\nabla \cdot \brack{\sqrt{\rho} \bxi},
\end{equation}
where $\bxi$ is a vector-valued space-time white noise and the prefactor $N^{-1/2}$ sets the scale of the fluctuations about the mean-field limit. Mathematically, \eqref{eq: Dean--Kawasaki} is a challenging equation to study due to its highly singular nature: it is neither renormalisable by the theory of regularity structures nor can it be treated using paracontrolled distributions. The derivation just sketched is, moreover, purely formal, since $\mu_t^N$ carries no density. This turns out not to be an artefact of the derivation: in~\cite{konarovskyi2019dean} the authors show that \eqref{eq: Dean--Kawasaki} admits only trivial martingale solutions, in the sense that its solutions are precisely the empirical measures \eqref{eq: empirical-measure} it was derived from.

The practical motivation for resorting to \eqref{eq: Dean--Kawasaki} (specifically, to suitable discretisations with, say, mesh resolution $h$) is purely computational: in the physically relevant high-density regime $Nh^d \gg 1$ simulating the Dean--Kawasaki
equation costs $O(h^{-d})$, which is significantly less than the cost of direct particle simulation (which is no less than $O(N)$).
Our analysis treats the purely diffusive case, which is the setting in which the
structural difficulties of \eqref{eq: Dean--Kawasaki} are already present;
interacting systems are treated numerically in Section~\ref{sec: numerics} and
their future outlook is discussed in Section~\ref{sec: future}.

The key ingredient of the present work is a general procedure for discretising
multiplicative, divergence-form noise of the form $\nabla \cdot (\sqrt{m(\rho)}\bxi)$
on finite element spaces whilst preserving its cross-variation structure at the
discrete level. Here $m$ denotes the mobility, with $m(\rho) = \rho$ in
the Dean--Kawasaki case; other choices of interest include
$m(\rho) = \rho(1-\rho)$, which arises for the simple symmetric exclusion
process~\cite{dirr2026conservative}, and the general class treated in~\cite{fehrman2024well}. The construction is explicit, elementwise, of arbitrary
order, and applies equally to continuous and discontinuous finite element
spaces. It allows us to
\begin{enumerate}[label=(\alph*)]
    \item extend the continuous finite element analysis of~\cite{cornalba2023dean}
    beyond the piecewise linear case treated there to arbitrary polynomial order, and
    \item introduce and analyse an arbitrary order spatial discontinuous Galerkin
    (dG) method for the Dean--Kawasaki equation (see \eqref{eq: dg-dk}).
\end{enumerate}

The motivation for a dG method stems from the stochastic conservation law
structure of the Dean--Kawasaki equation, for which a dG discretisation with
appropriately chosen fluxes is locally conservative. To the best of our
knowledge, this is the first analysis of a dG method for
\eqref{eq: Dean--Kawasaki}. Two further features of the finite element
framework are exploited in Section~\ref{sec: numerics}: unstructured meshes let
us treat non-trivial geometries and impose zero-flux boundary conditions for
reflecting particle systems, whereas existing analyses of
\eqref{eq: Dean--Kawasaki} are predominantly confined to uniform grids on the torus; and
upwind fluxes provide stability in the convection-dominated regimes arising when
particles are subject to external 
or interaction potentials. 

\subsection{Our contributions}

The main contributions of this work are as follows:
\begin{itemize}
    \item We provide an explicit method for constructing an arbitrary order, continuous or discontinuous Galerkin approximation to the highly irregular noise that preserves the correct cross-variation structure at the discrete level (see Propositions~\ref{prop: construct-diffusion-coefficients} and~\ref{prop: existence-of-noise}). Additionally, our construction is shown to be computationally efficient and applicable to general divergence-form noise (see Remarks~\ref{rem: noise-cost} and~\ref{rem: noise-general-mobility}). 
    \item Building upon the techniques in~\cite{cornalba2023dean, cornalba2026density}, we prove weak error estimates for fluctuations between a purely diffusive particle system and our dG approximation of the associated Dean--Kawasaki equation (see Theorem~\ref{thm: main-theorem}). Our rate is of order $O(h^p)$ for mesh resolution $h$ and polynomial order $p \geq 1$ plus an exponentially small correction associated with the numerical scheme going negative. 
    \item We perform quantitative numerical experiments to verify our weak error estimates and qualitative experiments that apply dG methods to settings not covered by the present theory. This includes indicator-function observables, particle systems with external and interaction potentials, and domains with reflecting boundary conditions (see Section~\ref{sec: numerics}). 
    We use the Python package Firedrake~\cite{FiredrakeUserManual} and the code is available in the open-source repository \url{https://github.com/kamran-arora/Dean-Kawasaki-DG}.
\end{itemize}

Two limitations should be stated at the outset. Our analysis is semi-discrete:
we discretise in space only, since this is where the structural difficulties
reside. Second, the absence of a
discrete maximum principle for the interior penalty method forces an upper
restriction on the mesh size (Assumption~\ref{ass: dG2}); both are discussed in
Section~\ref{sec: main-result} and revisited in Section~\ref{sec: future}.

\subsection{Related literature}

As mentioned previously, the only solutions to the Dean--Kawasaki equation are the empirical measures of the underlying particle system~\cite{konarovskyi2019dean}. Analogous results have since been established for a much wider class of Dean--Kawasaki-type equations in~\cite{konarovskyi2020dean, konarovskyi2024dean, schiavo2024massive, muller2025well}.
In light of this result, research has shifted to studying regularised variants of the Dean--Kawasaki equation that admit non-trivial solutions yet remain a faithful approximation of the underlying particle system. Of particular interest (and most closely related to the present work) are the results obtained in~\cite{cornalba2023dean, cornalba2026density}, where the authors show that arbitrary order finite difference discretisations and first order (i.e., linear) finite element discretisations of the Dean--Kawasaki equation provide accurate descriptions of density fluctuations of particle systems when measured in suitably weak (but highly relevant) metrics. Their estimates include a numerical error and error due to the negative part of the solution, the latter being of exponentially small size in the scaling regime $Nh^d \gg 1$ where $h$ is the mesh resolution. A similar result is shown to hold for piecewise linear, continuous finite elements. This framework is extended in~\cite{cornalba2026density} to cross-diffusing, weakly-interacting particle systems with sufficiently regular interaction potentials while~\cite{cornalba2025multilevel} provides rigorous justification for the use of multilevel Monte Carlo methods.

A complementary approach is taken in~\cite{djurdjevac2024weak,djurdjevac2026weak}, where weak-error estimates are proved between particle systems and analytically regularised Dean--Kawasaki equations. Rather than relying on a spatial discretisation, the square-root coefficient is smoothed, and the noise is truncated through a frequency cutoff. The resulting model preserves positivity and mass, although the convergence rate is limited by an $N$-dependent bound that deteriorates with dimension. The analysis is infinite-dimensional and relies on entirely different techniques so can be seen as dual to the finite-dimensional framework considered here. More recently, related techniques have been used to study the role of positivity preservation in approximations of the Dean--Kawasaki equation~\cite{damnjanovic2026role}.

Several other regularised models have been studied. In~\cite{cornalba2019regularized, cornalba2020weakly, cornalba2021well}, a regularised Dean--Kawasaki model is derived for inertial particle systems by replacing point particles with particles of a small finite width. The resulting SPDE admits high probability existence and positivity results. Later, the authors propose and analyse a discontinuous Galerkin method for their regularised SPDE~\cite{cornalba2023regularised}. Another line of works regularise the Dean--Kawasaki equation by introducing a spatial correlation to the noise. Well-posedness is established in~\cite{fehrman2024well} using stochastic kinetic solutions. Notably, they can treat the full square-root nonlinearity. Later, the framework was extended to bounded domains in~\cite{popat2025well} and to kinetic-type equations in~\cite{hao2025kinetic}.

The Dean--Kawasaki equation is also finding widespread applications in the physical sciences (see, for example, the review~\cite{illien2025dean}). As such, there are a growing number of works investigating the numerical approximation of the Dean--Kawasaki equation and related SPDEs. We have the aforementioned finite-difference results,~\cite{cornalba2023dean, cornalba2025multilevel, cornalba2026density}, but also finite element methods in~\cite{bavnas2024numerical, martinez2024finite}, and the finite-volume method in~\cite{kim2017stochastic, bell2026surface} whilst~\cite{djurdjevac2025hybrid} introduces a hybrid particle-SPDE framework. A discontinuous Galerkin method for stochastic conservation laws with a simpler non-divergence-form noise is studied in~\cite{li2020discontinuous}. Related numerical work uses the Dean--Kawasaki equation as a coarse-grained model to study features of interacting particle systems, for example,~\cite{delfau2016pattern, wehlitz2025approximating, jin2026field}.

\subsection{Structure of the paper}

The remainder of the paper is structured as follows. In Section~\ref{sec: notation}, we fix our notation then, in Section~\ref{sec: numerical-scheme}, we derive our numerical scheme (see \eqref{eq: dg-dk}). In Section~\ref{sec: noise-construc}, we rigorously construct the noise term (Propositions~\ref{prop: construct-diffusion-coefficients} and~\ref{prop: existence-of-noise}). We also state a weak existence result for the semi-discrete model (Theorem~\ref{thm: weak-existence}), whose proof is provided in Appendix~\ref{app: well-posed}, and detail some conservation properties (Propositions~\ref{prop: global-conservation} and~\ref{prop: local-conservation}). Section~\ref{sec: main-result} contains our weak error estimate (Theorem~\ref{thm: main-theorem}) and analysis with technical proofs given in Appendix~\ref{app: proofs}. 
We then present a selection of numerical experiments in Section~\ref{sec: numerics}, and conclude in Section~\ref{sec: future} with a discussion of extensions and directions for future work. Appendix~\ref{sec: recursive} contains auxiliary results used for the analysis and Appendix~\ref{sec: heat-estimates} provides semi-discrete error estimates for an adjoint heat equation. In Appendix~\ref{app: construction-of-initial-data}, we outline a method of constructing initial data satisfying Assumption~\ref{ass: dG1}.
 
\section{Notation}\label{sec: notation}

We work on the torus $\btd$ for $d \in \{1, 2, 3\}$. Our existence result is probabilistically weak (Theorem~\ref{thm: weak-existence}) so any reference to a probability space is to the one upon which the solution being considered is defined. We use the notation $(\Omega, \mathcal{G}, \mathbb{P})$ with filtration $(\mathcal{G}_t)_{t \geq 0}$. Expectation with respect to $\mathbb{P}$ is denoted by $\mathbb{E}$. 

\subsubsection*{Mesh}

We partition the domain using a sequence of shape-regular, quasi-uniform, affine simplicial meshes $\Th$ such that
$$
\btd = \bigcup_{K \in \Th} K.
$$
The index $h$ denotes the refinement level of the mesh. We define $h \coloneqq \max_{K \in \Th} h_K$ where $h_K \coloneqq \mathrm{diam}(K)$. The quasi-uniformity assumption implies that there exists a constant $C > 0$, independent of $h$, such that
\begin{equation}\label{eq: quasi-uniform}
h_K \geq C h, \quad \forall K \in \Th.
\end{equation}
Write $\Fk$ for the collection of facets of element $K$ and $\Fh$ for the collection of all facets. Since we work on the torus, boundary faces are identified periodically and treated as interior faces. For each facet $F \in \Fh$, we define $h_F > 0$ by $h_F \coloneqq \mathrm{diam}(F)$. Moreover, we assume that the facets do not degenerate with respect to $h_K$, that is, there exists a constant $C > 0$, independent of $h$, such that
$$
h_K \leq C h_F, \quad \forall K \in \Th, ~ \forall F \in \Fk.
$$
In particular, it follows that
$$
c h_K \leq h_F \leq C h_K, \quad \forall K \in \Th, ~ \forall F \in \Fk,
$$
for constants $c, C>0$ both independent of $h$.

\subsubsection*{Function spaces}

Let $\scalarpolys$ (respectively $\vectorpolys$) denote the space of polynomials (respectively, $\mathbb{R}^d$-valued polynomials) of degree $p \in \mathbb{N}_0$ over an element $K$. Our piecewise polynomial dG spaces are then defined as 
\begin{align*}
    \scalardG &\coloneqq \{u \in L^2(\btd) : u|_K \in \scalarpolys ~\text{for all}~ K \in \Th\}, \\
    \vectordG &\coloneqq \{\bv \in L^2(\btd) : \bv|_K \in \vectorpolys ~\text{for all}~ K \in \Th\}.
\end{align*}
We define the space $C^\beta \coloneqq C^\beta(\btd)$ and the Sobolev spaces $W^{k, p}(\btd)$ with $H^k(\btd)\coloneqq W^{k, 2}(\btd)$. We drop $\btd$ when clear. We also require the broken Sobolev spaces $\mathcal{W}^{k,p}(\Th)$ defined via
$$
\mathcal{W}^{k,p}(\Th) \coloneqq \{u \in L^2(\btd) : u|_K \in W^{k,p} ~\text{for all}~ K \in \Th\},
$$
with $\mathcal{H}^k \coloneqq \mathcal{W}^{k, 2}$.

\subsubsection*{Nodal interpolation operator}

For each element $K \in \Th$, let $\{\bx_i^K\}_{i=1}^{n_p} \subset \overline{K}$ be a set of nodal points. We define the (global) nodal interpolation operator $\mathcal{I}_h: \mathcal{H}^{p+1}(\Th) \rightarrow V_p$ via
\begin{equation}
    (\mathcal{I}_hu)|_K(\bx_i^K) = u(\bx_i^K), \quad \forall 1 \leq i \leq n_p, \quad \forall K \in \Th,
\end{equation}
for any $u \in \mathcal{H}^{p+1}(\Th)$. That is, $\mathcal{I}_h u$ is the function that when restricted to each element agrees with $u$ at each nodal point.

\subsubsection*{Average and jump operators}

Given an element $K \in \Th$, we define $\bn_K$ as its outward unit normal. Consider a facet $\Fh \ni F \coloneqq \partial K_\ell \cap \partial K_r$ for any two neighbouring elements $K_\ell, K_r$. We then have $\bn_{K_\ell} = - \bn_{K_r}$. We can then define the average of a function $\bv \in \mathcal{W}^{1,1}(\Th; \mathbb{R}^m)$ on $F$ by
$$
\avg{\bv} \coloneqq \frac{1}{2}\brack{\bv|_{K_\ell}+ \bv|_{K_r}},
$$
where $\bv|_{K_\ell}$ (respectively $\bv|_{K_r}$) denotes the trace of $\bv$ taken from the interior of $K_\ell$ (respectively $K_r$). We also define the jump of a scalar-valued function $v \in \mathcal{W}^{1, 1}(\Th; \mathbb{R})$ on $F$ by
$$
\jump{v} \coloneqq v|_{K_\ell}\bn_{K_\ell} + v|_{K_r}\bn_{K_r}.
$$
Similarly, define the jump of a vector-valued function $\bv \in \mathcal{W}^{1, 1}(\Th; \mathbb{R}^m)$ on $F$ by
$$
\jump{\bv} \coloneqq \bv|_{K_\ell} \cdot \bn_{K_\ell} + \bv|_{K_r} \cdot \bn_{K_r}.
$$
Note the jump is vector-valued for scalar inputs and scalar-valued for vector inputs. Moreover, it is independent of the order in which the two elements are labelled.

\subsubsection*{Lifting operators}

For $F \in \Fh$ and $\ell \in \mathbb{N}_0$, we define the local lifting operator $\boldsymbol{\mathcal{L}}_F^\ell: L^2(F; \mathbb{R}^d) \rightarrow \boldsymbol{V}_\ell^d$ as follows: for all $\bu \in L^2(F; \mathbb{R}^d)$, let $\boldsymbol{\mathcal{L}}_F^\ell(\bu)$ be the function defined by
\begin{equation}\label{eq: local-lift}
\int_{\btd} \boldsymbol{\mathcal{L}}_F^\ell(\bu) \cdot \bv_h \dd \bx \coloneqq \int_F \avg{\bv_h} \cdot \bu \dd S, \quad \forall \bv_h \in \boldsymbol{V}_\ell^d.
\end{equation}
For a function $u \in \mathcal{H}^1(\Th)$, we define the global lift of its jumps via
\begin{equation}\label{eq: global-lift}
\globallift^\ell(\jump{u}) \coloneqq \sum_{F \in \Fh} \locallift^\ell(\jump{u}).
\end{equation}

\subsubsection*{Broken and discrete gradients}
Let $u \in \mathcal{H}^1(\Th)$. We define the broken gradient $\nabla_h: \mathcal{H}^1(\Th) \rightarrow L^2(\btd; \mathbb{R}^d)$ by
$$
(\nabla_h u)|_K \coloneqq \nabla (u|_K), \quad \forall K \in \Th.
$$
We also define the discrete gradient $\dgrad^\ell: \mathcal{H}^1(\Th) \rightarrow L^2(\btd; \mathbb{R}^d)$ via
\begin{equation}\label{eq: discrete-gradient}
\dgrad^\ell u \coloneqq \nabla_h u - \globallift^\ell(\jump{u}).
\end{equation}
Throughout this work, unless stated otherwise, we make the choice $\ell = p-1$ for the discrete gradient and will write $\dgrad \coloneqq \dgrad^{p-1}$, $\locallift \coloneqq \locallift^{p-1}$ and $\globallift \coloneqq \globallift^{p-1}$ for the remainder of the paper.

\begin{remark}[Intuition for $\dgrad$]
    When taking the gradient of a broken function in the distributional sense, one must account for facet contributions which are proportional to the jump. Indeed, for a broken function $v_h \in \mathcal{H}^1(\Th)$, we have
    $$
    \nabla v_h = \nabla_h v_h - \sum_{F \in \Fh} \jump{v_h} \delta_F,
    $$
    in the distributional sense where $\delta_F$ is the surface Dirac measure. It is convenient to represent this in a single $L^2$ vector field; hence, we lift the facet contributions into the interior. The resulting object is precisely the discrete gradient $\dgrad$.
\end{remark}

\subsubsection*{Positive and negative part}

For a function $v_h \in V_p$, define the positive part $v_h^+ \coloneqq \max\{v_h, 0\}$ and the negative part $v_h^- \coloneqq -\min\{v_h, 0\}$ where the minima/maxima are taken elementwise. 

\section{The numerical scheme}\label{sec: numerical-scheme}

\subsection{Variational formulation of the noise}\label{subsec: discretisation-of-noise}

Formally testing the Dean--Kawasaki noise with a smooth test function $v$ and integrating by parts gives
$$
\int_{\btd} \nabla \cdot (\sqrt{\rho} \bxi) v \dd \bx = - \int_{\btd} \sqrt{\rho} \bxi \cdot \nabla v \dd \bx = - \sum_{k \geq 1} \dot{\beta}_k \int_{\btd} \sqrt{\rho} \boldsymbol{e}_k \cdot \nabla v \dd \bx.
$$
The final equality follows from the formal representation $\bxi = \sum_{k \geq 1} \boldsymbol{e}_k \dot{\beta}_k$ where $(\boldsymbol{e}_k)_{k \geq 1}$ is an orthonormal basis of $L^2(\btd; \mathbb{R}^d)$ and $(\beta_k)_{k \geq 1}$ are i.i.d.\ standard Brownian motions. For two test functions $v_1, v_2$, one can then use Parseval's identity to compute the cross-variation of the noise:
\begin{multline*}
    \dd \abrack{\sum_{k \geq 1} \dot{\beta}_k \int_{\btd} \sqrt{\rho} \boldsymbol{e}_k \cdot \nabla v_1 \dd \bx, \sum_{\ell \geq 1} \dot{\beta}_\ell \int_{\btd} \sqrt{\rho} \boldsymbol{e}_\ell \cdot \nabla v_2 \dd \bx}, \\
    = \sum_{k \geq 1}(\sqrt{\rho}\boldsymbol{e}_k, \nabla v_1)(\sqrt{\rho}\boldsymbol{e}_k, \nabla v_2) \dd t = (\rho, \nabla v_1 \cdot \nabla v_2) \dd t.
\end{multline*}

Motivated by this observation, we wish to design a numerical scheme that preserves this cross-variation structure at the semi-discrete level. In the dG setting, test functions can be discontinuous; therefore, the appropriate notion of gradient is the discrete (or distributional) gradient defined in \eqref{eq: discrete-gradient}. This combines the classical gradient defined on the interior of each element with a jump contribution on the facets to account for discontinuities. Given this, a natural discretisation of the noise is to seek a finite-dimensional martingale $\mathcal{W}(\rho_h^+(t), v_h)$ with the following cross-variation structure:
\begin{equation}\label{eq: crossvar-target}
\dd \abrack{\mathcal{W}(\rho_h^+(\cdot), v_{1, h}), \mathcal{W}(\rho_h^+(\cdot), v_{2, h})}_t = (\rho_h^+(t), \dgrad v_{1, h} \cdot \dgrad v_{2, h}) \dd t,
\end{equation}
for all $v_{1, h}, v_{2, h} \in V_p$. We impose \eqref{eq: crossvar-target} on the discrete space $V_p$ rather than on all of $\mathcal{H}^1(\Th)$ since the construction in Proposition~\ref{prop: construct-diffusion-coefficients} relies on $\dgrad v_h$ lying in the finite-dimensional space $\boldsymbol{V}_{p-1}^d$; this holds for $v_h \in V_p$, but fails for a general element of $\mathcal{H}^1(\Th)$. Every test function to which we apply \eqref{eq: crossvar-target} belongs to $V_p$, so this costs us nothing.
\begin{remark}
    We have used the positive part of the solution $\rho_h^+$ to account for the fact that our scheme does not necessarily preserve positivity. This regularisation is needed in Proposition~\ref{prop: construct-diffusion-coefficients} when we construct an explicit representation of the noise in our finite element space.
\end{remark}
In fact, we can further justify that this is the correct discretisation by the following formal argument: consider the (formal) pure noise Dean--Kawasaki equation and write it as a conservation law. That is,
\begin{align}\label{purenoise_DK}
\partial_t \rho - \nabla \cdot \boldsymbol{J} = 0, \quad \boldsymbol{J} \coloneqq N^{-1/2}\sqrt{\rho}\,\bxi.
\end{align}
We focus on the term $\nabla \cdot \boldsymbol{J}$. Multiplying by $v \in H^1$, integrating over an element $K \in \Th$, integrating by parts and then summing over all elements gives
$$
-\sum_{K \in \Th}\int_K (\nabla \cdot \boldsymbol{J}) v \dd \bx = \sum_{K \in \Th} \int_K \boldsymbol{J} \cdot \nabla v \dd \bx - \sum_{K \in \Th} \int_{\partial K} (\boldsymbol{J} \cdot \bn_K) v \dd S.
$$
We now project into our finite-dimensional dG spaces. In particular, we seek $\boldsymbol{J}_h \in \vectordG$ such that 
$$
-\sum_{K \in \Th}\int_K (\nabla \cdot \boldsymbol{J}_h) v_h \dd \bx = \sum_{K \in \Th} \int_K \boldsymbol{J}_h \cdot \nabla v_h \dd \bx - \sum_{K \in \Th} \int_{\partial K} (\widehat{\boldsymbol{J}_h} \cdot \bn_K) v_h \dd S,
$$
holds for all $v_h \in \scalardG$. Since $\boldsymbol{J}_h$ is undefined on element boundaries, we have introduced the numerical flux $\widehat{\boldsymbol{J}_h}$. We rewrite the boundary integral over facets and then choose $\widehat{\boldsymbol{J}_h} = \avg{\boldsymbol{J}_h}$ to get
$$
-\sum_{K \in \Th}\int_K (\nabla \cdot \boldsymbol{J}_h) v_h \dd \bx = \int_{\btd} \boldsymbol{J}_h \cdot \nabla_h v_h \dd \bx - \sum_{F \in \Fh} \int_{F} \avg{\boldsymbol{J}_h} \cdot \jump{v_h} \dd S.
$$
We now recognise that the final term can be rewritten using the lift operator \eqref{eq: global-lift}, and therefore the Dean--Kawasaki noise (formally) acts on discontinuous test functions via
$$
- \sum_{K \in \Th}\int_K (\nabla \cdot \boldsymbol{J}_h) v_h \dd \bx = \int_{\btd} \boldsymbol{J}_h \cdot \dgrad v_h \dd \bx.
$$
We have chosen the central flux since it is symmetric and conservative. The former property is natural since there is no preferred direction for particles to jump over element interfaces. In Proposition~\ref{prop: global-conservation}, we show that the latter property ensures global conservation of mass.
\begin{remark}
    We have chosen to perform calculations on \eqref{purenoise_DK} to keep our arguments as succinct as possible: we stress that our arguments are purely formal, and are unrelated to -- and unaffected by -- the recently proved ill-posedness of \eqref{purenoise_DK} \cite{dello2025ill}.
\end{remark}

\subsection{Discretisation of the Laplacian}

So far, we have focused on the noise term since this is the primary challenge when deriving a suitable discretisation of the Dean--Kawasaki equation. To approximate the Laplacian, we use a standard scheme from the dG literature known as the symmetric interior penalty method~\cite{arnold1982interior}; see~\cite{riviere2008discontinuous, ern2021finite2} for a textbook treatment.

\begin{definition}\label{def: SIP-bilinear-form}
    Define the symmetric interior penalty bilinear form $a_h(\cdot, \cdot)$ via
    \begin{multline*}
        a_h(\rho_h, v_h) \coloneqq \frac{1}{2}(\nabla_h \rho_h, \nabla_h v_h) - \frac{1}{2}\sum_{F \in \Fh}\int_F \Big ( \jump{\rho_h} \cdot \avg{\nabla_h v_h} + \avg{\nabla_h \rho_h} \cdot \jump{v_h} \Big ) \dd S \\
        + \frac{\eta}{2} \sum_{F \in \Fh} h_F^{-1} \int_F \jump{\rho_h} \cdot \jump{v_h} \dd S, \quad \text{for any}~ \rho_h, v_h \in V_p,
    \end{multline*}
    where $\eta > 0$ is a penalty parameter.
\end{definition}

We have chosen this discretisation since it is symmetric by definition and coercive provided $\eta > 0$ is chosen sufficiently large. The size of this penalty parameter depends on the dimension, polynomial degree and shape regularity of the mesh. We also note that symmetry is a natural property since the diffusion of Brownian motions has no preferred direction.

\subsection{Definition of the numerical scheme}

Combining our discretisations of the noise and the Laplacian, we can now define our numerical scheme for \eqref{eq: Dean--Kawasaki}.

\begin{definition}[dG discretisation of the Dean--Kawasaki equation]\label{def: dk-DG}
    We consider the following discontinuous Galerkin Dean--Kawasaki model: find $\rho_h(t, \cdot) \in V_p$ for $t \in (0, T]$ such that
    \begin{multline*}\label{eq: dg-dk}
        \dd (\rho_h(t), v_h) = -\frac{1}{2}(\nabla_h \rho_h(t), \nabla_h v_h) \dd t \\
        + \frac{1}{2}\sum_{F \in \Fh}\int_F \Big ( \jump{\rho_h(t)} \cdot \avg{\nabla_h v_h} + \avg{\nabla_h \rho_h(t)} \cdot \jump{v_h}\Big ) \dd S \dd t \\
        - \frac{\eta}{2} \sum_{F \in \Fh} h_F^{-1} \int_F \jump{\rho_h(t)} \cdot \jump{v_h} \dd S \dd t + N^{-1/2}\dd \mathcal{W}(\rho_h^+(t), v_h), \tag{dG-DK}
    \end{multline*}
    holds for all $v_h \in V_p$ with initial condition $\rho_{0, h} \in V_p$, and where $\mathcal{W}(\rho_h^+(t), v_h)$ is a real-valued martingale with the cross-variation \eqref{eq: crossvar-target}.
    To make $\mathcal{W}$ explicit, we can write
    $$
    \dd \mathcal{W}(\rho_h^+(t), v_h) \coloneqq \sum_{j=1}^{n_G} (\boldsymbol{g}_j(\rho_h(t)), \dgrad v_h) \dd B_j(t),
    $$
    where the $B_j(t)$ are standard independent Brownian motions, $n_G \coloneqq \mathrm{dim}(\boldsymbol{V}_{p-1}^d)$, and the functions $\boldsymbol{g}_j$ are those constructed in Proposition~\ref{prop: construct-diffusion-coefficients} below. Intuitively, the functions $\{\boldsymbol{g}_j(\rho_h)\}_j$ mimic $\sqrt{\rho_h}$ appropriately such that we recover the desired cross-variation structure. Section~\ref{sec: noise-construc} is devoted to rigorously justifying this representation of the noise and proving a well-posedness result for \eqref{eq: dg-dk}.
\end{definition}

\begin{remark}
    Although the coefficients $\boldsymbol{g}_j(\rho_h)$ are defined for $\rho_h \in V_p$, they depend on $\rho_h$ only through its positive part $\rho_h^+$. We write $\mathcal{W}(\rho_h^+, \cdot)$ to emphasise this dependence.
\end{remark}

\section{Construction and properties of the noise and weak existence of a solution to the numerical scheme}\label{sec: noise-construc}

We prove one of the main contributions of the paper, namely, that the noise $\mathcal{W}(\rho_h^+(t), v_h)$ introduced in Definition~\ref{def: dk-DG} is a martingale with the desired cross-variation structure. We also establish the existence of probabilistically weak solutions to \eqref{eq: dg-dk} along with moment bounds.

\begin{proposition}\label{prop: construct-diffusion-coefficients}
    For any $u_h \in V_p$, there exists a collection of functions $\{\boldsymbol{g}_j(u_h)\}_j \subset \boldsymbol{V}_{p-1}^d$ such that
    $$
    \sum_{j=1}^{\mathrm{dim}(\boldsymbol{V}_{p-1}^d)} (\boldsymbol{g}_j(u_h), \dgrad v_{1, h})(\boldsymbol{g}_j(u_h), \dgrad v_{2, h}) = (u_h^+, \dgrad v_{1, h} \cdot \dgrad v_{2, h}),
    $$
    for any $v_{1, h}, v_{2, h} \in V_p$. The functions $\boldsymbol{g}_j$ are given by
    \begin{equation}\label{eq: gj-defn}
    \boldsymbol{g}_j(u_h) \coloneqq \sum_i (M^{-1}_{\mathrm{G}}R(u_h))_{ij} \boldf_i,
    \end{equation}
    where $\{\boldf_i\}_i$ is a basis for $\boldsymbol{V}_{p-1}^d$, $M_{\mathrm{G}}$ is the mass matrix on $\boldsymbol{V}_{p-1}^d$ with entries $(M_{\mathrm{G}})_{ij} = (\boldf_i, \boldf_j)$ and $R(u_h)$ is the unique, symmetric, positive semi-definite square root of the symmetric, positive semi-definite weighted mass matrix $P(u_h)$ with entries given by
    \begin{equation}\label{eq: Pt-definition}
    P(u_h)_{ij} \coloneqq \int_{\btd} u_h^+ \boldf_i \cdot \boldf_j \dd \bx.
    \end{equation}
    Moreover, the map $u_h \mapsto \boldsymbol{g}_j(u_h)$ is continuous for each $j$.
\end{proposition}

\begin{proof}
    The proof proceeds by factorising the object we wish to reconstruct. Define $n_G \coloneqq \mathrm{dim}(\boldsymbol{V}_{p-1}^d)$. Since $\dgrad v_{1, h}, \dgrad v_{2, h} \in \boldsymbol{V}_{p-1}^d$, there exist vectors of coefficients $\alpha, \beta$ such that
    $$
    \dgrad v_{1, h} = \sum_{j=1}^{n_G} \alpha_j \boldf_j, \quad \dgrad v_{2, h} = \sum_{j=1}^{n_G} \beta_j \boldf_j.
    $$
    Thus, we can write
    $$
    \int_{\btd} u_h^+ \dgrad v_{1, h} \cdot \dgrad v_{2, h} \dd \bx = \sum_{i, j=1}^{n_G}\alpha_i \beta_j \int_{\btd} u_h^+ \boldf_i \cdot \boldf_j \dd \bx.
    $$
    Using \eqref{eq: Pt-definition}, this can further be rewritten as
    $$
    \int_{\btd} u_h^+ \dgrad v_{1, h} \cdot \dgrad v_{2, h} \dd \bx = \alpha^{\mathrm{T}}P(u_h) \beta.
    $$
    Notice that $P(u_h)$ is symmetric positive semi-definite since $u_h^+ \geq 0$ so by~\cite[Theorem~7.2.6]{horn2012matrix}, there exists a unique, symmetric positive semi-definite square root $R(u_h) = P(u_h)^{1/2}$ such that $R(u_h)^2 = R(u_h)R^{\mathrm{T}}(u_h) = P(u_h)$. In particular, we have
    \begin{equation}\label{eq: alpha-RT-R-beta}
    \int_{\btd} u_h^+ \dgrad v_{1, h} \cdot \dgrad v_{2, h} \dd \bx = \alpha^{\mathrm{T}}R(u_h) R^{\mathrm{T}}(u_h) \beta.
    \end{equation}
    We now reconstruct this via the functions $\boldsymbol{g}_j$. Using \eqref{eq: gj-defn}, it follows that the coefficient vector of $\boldsymbol{g}_j(u_h)$ is the $j$-th column of $M_G^{-1}R(u_h)$. That is, we define $c^{(j)} \coloneqq M_G^{-1}R(u_h)e_j$, where $e_j$ is the $j$-th standard basis vector. Therefore, 
    $$
    (\boldsymbol{g}_j(u_h), \dgrad v_{1, h}) = \sum_{i, k=1}^{n_G}c_i^{(j)}\alpha_k (\boldf_i, \boldf_k) = (c^{(j)})^{\mathrm{T}}M_G \alpha.
    $$
    Inserting the definition of $c^{(j)}$, rearranging the transpose, using $(M_G^{-1})^{\mathrm{T}} = M_G^{-1}$ and noting that $R(u_h)$ is symmetric gives
    $$
    (\boldsymbol{g}_j(u_h), \dgrad v_{1, h}) = (M_G^{-1}R(u_h)e_j)^{\mathrm{T}}M_G \alpha = e_j^{\mathrm{T}}R(u_h)^{\mathrm{T}}\alpha = (R(u_h)\alpha)_j. 
    $$
    Similarly,
    $$
    (\boldsymbol{g}_j(u_h), \dgrad v_{2, h}) = (R(u_h)\beta)_j.
    $$
    Therefore, 
    \begin{align*}
    \sum_{j=1}^{n_G} (\boldsymbol{g}_j(u_h), \dgrad v_{1, h})(\boldsymbol{g}_j(u_h), \dgrad v_{2, h}) &= \sum_{j=1}^{n_G} (R(u_h)\alpha)_j (R(u_h)\beta)_j, \\
    &= (R(u_h)\alpha)^{\mathrm{T}} (R(u_h)\beta), \\
    &= \alpha^{\mathrm{T}}R(u_h)^{\mathrm{T}}R(u_h)\beta,
    \end{align*}
    which concludes the construction upon comparison with \eqref{eq: alpha-RT-R-beta}. To prove continuity, we note that the positive part map $u_h \mapsto u_h^+$ is Lipschitz and hence the map $u_h \mapsto P(u_h)$ is continuous. Moreover, by~\cite[Theorem X.1.1]{bhatia2013matrix}, the square root map is $1/2$-H\"older in the operator norm. By composition, it follows that $u_h \mapsto \boldsymbol{g}_j(u_h)$ is $1/2$-H\"older and hence continuous for each $j$.
\end{proof}

\begin{remark}[Extension to continuous Galerkin]
    Note that this construction can easily be applied in the continuous Galerkin (cG) setting by replacing $\dgrad$ with $\nabla$. The remainder of the construction is analogous since for a continuous degree-$p$ finite element space, the associated gradient space is also $\boldsymbol{V}_{p-1}^d$. In particular, this allows the cG results in~\cite{cornalba2023dean} to be extended beyond the $p=1$ case that the authors were previously restricted to.
\end{remark}

\begin{remark}[Implementation and computational cost]\label{rem: noise-cost}
    In practice, we use a matrix factorisation to compute the square root $R(\rho_h)$. Since $\boldsymbol{V}_{p-1}^d$ is discontinuous, the weighted mass matrix $P(\rho_h)$ is block-diagonal with one dense block associated to each element. In particular, its factorisation boils down to a sequence of independent factorisations of these blocks. If $m_{p, d} \coloneqq \mathrm{dim} \boldsymbol{P}_{p-1}^d(K)$ and $N_h$ denotes the number of elements in the mesh, then (typically) the total cost is $O(N_h m^3_{p, d})$. For fixed $p$ and $d$, the block size is independent of the mesh resolution hence the factorisation runs in linear time with respect to $N_h$ (and therefore the number of degrees of freedom). Moreover, the elementwise factorisations are independent so can be computed in parallel.
\end{remark}

We can now rewrite our numerical scheme \eqref{eq: dg-dk} as the following finite-dimensional SDE:
$$
\dd (\rho_h(t), v_h) + a_h(\rho_h(t), v_h) \dd t = N^{-1/2}\sum_{j=1}^{\mathrm{dim}(\boldsymbol{V}_{p-1}^d)} (\boldsymbol{g}_j(\rho_h), \dgrad v_h) \dd B_j(t),
$$
where $\{B_j(t)\}_j$ is a collection of standard, independent Brownian motions. We then have the following result that asserts the existence of a probabilistically weak solution to the above system.

\begin{theorem}[Weak existence and moment bounds]\label{thm: weak-existence}
    Let $\rho_{0, h} \in V_p$ be deterministic and fix $T>0$. There exists a filtered probability space $(\Omega, \mathcal{G}, (\mathcal{G}_t)_{t \geq 0}, \mathbb{P})$ satisfying the usual conditions, an $n_G$-dimensional standard Brownian motion
    $$
    \boldsymbol{B} = (B_1, \ldots, B_{n_G})^{\mathrm{T}},\quad n_G \coloneqq \mathrm{dim}(\boldsymbol{V}_{p-1}^d),
    $$
    and a continuous, $\mathcal{G}_t$-adapted, $V_p$-valued process $\rho_h(t): [0, T] \times \Omega \rightarrow V_p$ such that $\mathbb{P}$-a.s., for every $v_h \in V_p$ and $t \in [0, T]$,
    $$
    (\rho_h(t), v_h) = (\rho_{0, h}, v_h) - \int_0^t a_h(\rho_h(s), v_h) \dd s + \frac{1}{\sqrt{N}}\sum_{j=1}^{n_G}\int_0^t (\boldsymbol{g}_j(\rho_h(s)), \dgrad v_h) \dd B_j(s).
    $$
    Moreover, for any $q \geq 2$, we have the moment bound
    $$
    \E{\sup_{t \in [0, T]}\norm{\rho_h(t)}^q} < \infty.
    $$
\end{theorem}

\begin{proof}
    See Appendix~\ref{app: well-posed}.
\end{proof}

\begin{remark}
    We have a probabilistically weak solution because the constructed noise degenerates as $\rho_h^+(t) \rightarrow 0$. We also do not prove pathwise uniqueness. Neither of these is an issue for our resulting analysis.
\end{remark}

\begin{proposition}[Existence of the noise]\label{prop: existence-of-noise}
    In the same setting as Theorem~\ref{thm: weak-existence}, define $\mathcal{W}(\rho_h^+(t), v_h)$ for $v_h \in V_p$ via
    $$
    \dd \mathcal{W}(\rho_h^+(t), v_h) \coloneqq \sum_{j=1}^{n_G} (\boldsymbol{g}_j(\rho_h(t)), \dgrad v_h) \dd B_j(t).
    $$
    Then $\mathcal{W}(\rho_h^+(t), v_h)$ is a continuous, square-integrable martingale with quadratic cross-variation
    $$
    \dd \abrack{\mathcal{W}(\rho_h^+(t), v_{1, h}), \mathcal{W}(\rho_h^+(t), v_{2, h})}_t = (\rho_h^+(t), \dgrad v_{1, h} \cdot \dgrad v_{2, h}) \dd t,
    $$
    for any $v_{1, h}, v_{2, h} \in V_p$.
\end{proposition}

\begin{proof}
    This follows immediately from Proposition~\ref{prop: construct-diffusion-coefficients} and Theorem~\ref{thm: weak-existence}. In particular, the former implies that $\mathcal{W}$ has the stated cross-variation structure and the moment bounds in the latter imply that the noise is a true continuous, square integrable martingale rather than a continuous local martingale.
\end{proof}

\begin{remark}[General mobilities]\label{rem: noise-general-mobility}
    A significant advantage of our noise construction is that it readily extends to general divergence-form noise of the following type
    \begin{equation}\label{eq: general-div-form-noise}
    \nabla \cdot \brack{\sqrt{m(\rho)} \bxi},
    \end{equation}
    where $m(\rho)$ is a prescribed mobility (e.g., in the Dean--Kawasaki setting, $m(\rho) = \rho$). Formally, the cross-variation of this noise satisfies
    $$
    \dd \abrack{\int_{\btd} \sqrt{m(\rho)} \bxi \cdot \nabla v_1, \int_{\btd} \sqrt{m(\rho)} \bxi \cdot \nabla v_2} = (m(\rho), \nabla v_1 \cdot \nabla v_2) \dd t,
    $$
    for sufficiently regular test functions $v_1, v_2$. As a result, the noise construction extends directly to this setting by defining the weighted mass matrix via
    $$
    P(\rho_h)_{ij} \coloneqq \int_{\btd} [m(\rho)]_h \boldsymbol{f}_i \cdot \boldsymbol{f}_j \dd \bx,
    $$
    where $[m(\rho)]_h$ is a suitable non-negative numerical approximation to $m(\rho)$. The remainder of the construction is unchanged. In particular, our proposed method should be viewed as a general procedure for discretising multiplicative, divergence-form noise whilst preserving its cross-variation structure at the discrete level, rather than as a construction tied specifically to the Dean--Kawasaki equation \eqref{eq: Dean--Kawasaki}.
    
    This observation considerably widens the range of potential applications. For instance, studying the Simple Symmetric Exclusion Process leads to a mobility of the form $m(\rho) = \rho(1-\rho)$ (see, for example,~\cite{dirr2026conservative}). More broadly, generalised Dean--Kawasaki equations with noise of the form \eqref{eq: general-div-form-noise} have been extensively studied in~\cite{fehrman2024well} (and references therein).
\end{remark}

\begin{remark}[Simple representation when $p=1$]
    When $p=1$, the martingale $\mathcal{W}$ admits the simpler representation
    $$
    \dd \mathcal{W}(\rho_h^+(t), v_h) = \sum_{K \in \Th}\sum_{i=1}^d \sqrt{\frac{\Pi_h^0 \rho_h^+(t)}{\mathrm{meas}(K)}}(\boldsymbol{f}_{K, i}, \dgrad v_h) \dd \boldsymbol{B}_{K, i}(t),
    $$
    where $\Pi_h^0$ denotes the $L^2$ projection onto $V_0$, $\boldsymbol{f}_{K, i} \coloneqq \mathbbm{1}_K \boldsymbol{e}_i$ for $K \in \Th$, $i=1, \ldots, d$ and where $\{\boldsymbol{e}_i\}_{i=1}^d$ is the standard basis for $\mathbb{R}^d$, and $\{\boldsymbol{B}_{K, i}\}$ is a collection of standard, i.i.d.\ Brownian motions.

    This follows from considering the proof of Proposition~\ref{prop: construct-diffusion-coefficients} and noticing that the mass matrices $M_G$ and $P(\rho_h)$ are diagonal. As a result, $R(\rho_h)$ can be obtained by simply square rooting each entry of $P(\rho_h)$. Therefore, exploiting that the basis functions are constant and plugging the definitions of $M_G$ and $R(\rho_h)$ into the general expression for $\boldsymbol{g}_j(\rho_h)$ gives
    $$
    \boldsymbol{g}_{K, i}(\rho_h(t)) = \frac{\sqrt{\int_K \rho_h^+(t, \bx) \dd \bx}}{\mathrm{meas}(K)} \boldsymbol{f}_{K, i} = \sqrt{\frac{\Pi_h^0 \rho_h^+(t)}{\mathrm{meas}(K)}} \boldsymbol{f}_{K, i},
    $$
    and the claim follows.

    This is advantageous from a computational perspective since we no longer need to perform a matrix factorisation. Instead, we only need to project onto $V_0$ which is equivalent to inverting a diagonal matrix.
\end{remark}

Armed with our explicit construction of the noise, we can show that \eqref{eq: dg-dk} conserves global mass and is locally conservative. The latter property means that mass transfer on an element is only driven by fluxes across its boundary.

\begin{proposition}\label{prop: global-conservation}
    The numerical scheme \eqref{eq: dg-dk} conserves global mass. That is, for all $t > 0$ we have
    $$
    \dd \int_{\btd} \rho_h(t, \bx) \dd \bx = 0.
    $$
\end{proposition}

\begin{proof}
    Choose $v_h = 1$ globally in \eqref{eq: dg-dk}. Then $a_h(\rho_h, 1) = 0$ since $\nabla_h 1=0$ and $\jump{1} = 0$. Moreover, $\dgrad 1 = 0$ so the quadratic variation of the noise is $\abrack{\mathcal{W}(\rho_h^+(t), 1)} = 0$. Since $\mathcal{W}(\rho_{0, h}, 1) = 0$, it follows that $\mathcal{W}(\rho_h^+(t), 1) = 0$ for any $t > 0$ since a martingale with zero quadratic variation is equal to its initial value.
\end{proof}

\begin{proposition}\label{prop: local-conservation}
    The numerical scheme \eqref{eq: dg-dk} is locally conservative. That is, the rate of change of mass on any element is driven only by deterministic and stochastic fluxes across its boundary. In mathematical terms, for any $K \in \Th$ and any $t > 0$ we have
    $$
    \dd \int_K \rho_h(t, \bx) \dd \bx = -\int_{\partial K} \widehat{\boldsymbol{J}}_{\mathrm{det}} \cdot \bn_K \dd S \dd t -\int_{\partial K} \dd \widehat{\boldsymbol{J}}_{\mathrm{stoch}} \cdot \bn_K \dd S.
    $$
    The deterministic flux is given by $\widehat{\boldsymbol{J}}_{\mathrm{det}} = -\frac{1}{2}\avg{\nabla_h\rho_h} + \frac{\eta}{2h_F}\jump{\rho_h}$ and the stochastic flux is given by $\dd \widehat{\boldsymbol{J}}_{\mathrm{stoch}} = N^{-1/2}\sum_j \avg{\boldsymbol{g}_j(\rho_h)}\dd B_j$ where $\boldsymbol{g}_j$ are the functions introduced in Proposition~\ref{prop: construct-diffusion-coefficients} and $B_j$ are independent, standard Brownian motions.
\end{proposition}

\begin{proof}
    Fix an element $K \in \Th$ and choose $v_h = \mathbbm{1}_K$ in \eqref{eq: dg-dk} where $\mathbbm{1}_K$ denotes the indicator function on $K$:
    $$
    \dd (\rho_h, \mathbbm{1}_K) + a_h(\rho_h, \mathbbm{1}_K) \dd t = N^{-1/2}\dd \mathcal{W}(\rho_h^+, \mathbbm{1}_K).
    $$
    Since $\nabla_h \mathbbm{1}_K = 0$, the SIP bilinear form reduces to only boundary integrals
    $$
    a_h(\rho_h, \mathbbm{1}_K) = - \int_{\partial K} \brack{\frac{1}{2} \avg{\nabla_h \rho_h} - \frac{\eta}{2 h_F}\jump{\rho_h}}\cdot \bn_K \dd S.
    $$
    The discrete gradient satisfies
    $$
    \dgrad \mathbbm{1}_K = - \boldsymbol{\mathcal{L}}_h(\jump{\mathbbm{1}_K}) = -\sum_{F \in \Fh}\boldsymbol{\mathcal{L}}_F(\jump{\mathbbm{1}_K}),
    $$
    where we have again used that $\nabla_h \mathbbm{1}_K = 0$. Using Proposition~\ref{prop: existence-of-noise}, it follows that
    \begin{align*}
    \dd \mathcal{W}(\rho_h^+, \mathbbm{1}_K) &= -\sum_j \dd B_j \sum_{F \in \Fh}(\boldsymbol{g}_j(\rho_h), \locallift(\jump{\mathbbm{1}_K})) \\
    &= -\sum_j \dd B_j \sum_{F \in \Fh}\int_F \avg{\boldsymbol{g}_j(\rho_h)} \cdot \jump{\mathbbm{1}_K} \dd S.
    \end{align*}
    For $F \in \mathcal{F}_K$, we have $\jump{\mathbbm{1}_K} = \bn_K$, otherwise it evaluates to zero. It follows that
    \begin{align*}
    \dd \mathcal{W}(\rho_h^+, \mathbbm{1}_K) &= -\sum_j \dd B_j \sum_{F \in \Fk}\int_F \avg{\boldsymbol{g}_j(\rho_h)} \cdot \bn_K \dd S \\
    &= -\sum_j \dd B_j \int_{\partial K} \avg{\boldsymbol{g}_j(\rho_h)} \cdot \bn_K \dd S,
    \end{align*}
    and the claim follows.
\end{proof}

\section{Weak error analysis}\label{sec: main-result}

Before stating and proving the weak error result, we introduce the auxiliary objects and assumptions needed to formulate it.

\subsection{Mean-field limit}\label{subsec: mfl}

We introduce the mean-field limits for the particle system and numerical scheme. These provide the centering for the fluctuations we wish to compare. The Brownian particle system \eqref{eq: empirical-measure} is associated to the following mean-field PDE
\begin{equation}
    \partial_t \overline{\rho} = \frac{1}{2}\Delta \overline{\rho}, \quad \text{ on } \btd \times (0, T],
\end{equation}
with initial condition $\overline{\rho}(0) \coloneqq \overline{\rho}_0$. We also introduce its SIP semi-discretisation: find $\overline{\rho}_h(t, \cdot) \in V_p$ for $t \in (0, T]$ such that
\begin{equation}\label{eq: mfl}
    (\partial_t\overline{\rho}_h(t), v_h) + a_h(\overline{\rho}_h(t), v_h) = 0,
\end{equation}
holds for all $v_h \in V_p$ with initial condition $\rho_{0, h} \in V_p$. Since the drift in \eqref{eq: dg-dk} is linear and the noise is a true martingale (Proposition~\ref{prop: existence-of-noise}), it follows that $\E{\rho_h(t)} = \overline{\rho}_h(t)$ for all $t > 0$. Equality at $t=0$ follows since we start both from the same initial condition $\rho_{0, h}$.

\subsection{Assumptions}

Our first assumption specifies the set of admissible initial data whilst the second provides the scaling regime between $N$ and $h$ required by our analysis.

\begin{assumptionp}{dG1}[Initial conditions]\label{ass: dG1}
    We make the following assumptions on initial conditions:
    \begin{enumerate}
        \item the initial datum for the continuous mean-field limit $\overline{\rho}_0$ has regularity $\overline{\rho}_0 \in W^{p+1, \infty} \cap H^{p+2}$ and its minimum and maximum satisfy
        $$
        \rhomin \coloneqq \min_{\bx \in \btd}\overline{\rho}_0(\bx) > 0, \quad \rhomax \coloneqq \max_{\bx \in \btd}\overline{\rho}_0(\bx) < \infty.
        $$
        \item the initial positions $\{\bB_r(0)\}_{r=1}^N$ are deterministic.
        \item the function $\rho_{0, h} \in V_p$ satisfies the following:
        \begin{enumerate}
            \item it is positive (i.e., $\rho_{0, h} > 0$).
            \item it approximates the empirical measure of the initial configuration $\mu_0^N \coloneqq N^{-1}\sum_{r=1}^N \delta_{\bB_r(0)}$ with accuracy $p$ in the sense that
            \begin{equation}\label{eq: dG1-3b}
                \absv{N^{-1}\sum_{r=1}^N v(\bB_r(0)) - (\rho_{0, h}, \mathcal{I}_h v)} \leq Ch^p \norm{v}_{W^{p+1, \infty}},
            \end{equation}
            holds for all $v \in W^{p+1, \infty}$.
            \item it approximates $\overline{\rho}_0$ with accuracy $p+1$ in the sense that
            \begin{equation}\label{eq: dG1-3c}
            \norm{\rho_{0, h} - \overline{\rho}_0} \leq Ch^{p+1} \norm{\overline{\rho}_0}_{H^{p+2}}.
            \end{equation}
        \end{enumerate}
    \end{enumerate}
\end{assumptionp}

\begin{remark}
    The $O(h^p)$ error in \eqref{eq: dG1-3b} is the minimal order required so that the error from the initial condition does not dominate the weak error estimate in Theorem~\ref{thm: main-theorem}. The $O(h^{p+1})$ error in \eqref{eq: dG1-3c} is required so that the exponent $p+1-d/2$ appearing in Corollary~\ref{cor: ass-corollary} is strictly positive for $p \geq 1$ and $d \in \{1, 2, 3\}$. This is not true for $p-d/2$. In Appendix~\ref{app: construction-of-initial-data}, we give an example of how one can construct initial data satisfying Assumption~\ref{ass: dG1}.
\end{remark}

\begin{remark}
    We expect our results to extend to random initial conditions by appropriate adaptation of the considerations in~\cite{cornalba2026density}.
\end{remark}

\begin{assumptionp}{dG2}[Scaling of $N$ and $h$]\label{ass: dG2}
    Fix $\delta \in (0, \rhomin)$ independent of $N$ and $h$. We assume the scaling
    \begin{equation}
        C(d, \delta, \rhomin, \rhomax)N^{-1/d}\absv{\log N}^{2/d}(1+T) \leq h \leq h_0(\delta),
    \end{equation}
    where
    $$
    h_0(\delta) \coloneqq \min \left \{C(p, d, \mfl_0) \delta^{1/(p+1-d/2)}, 1 \right \},
    $$
    for constants independent of the discretisation parameter $h$.
\end{assumptionp}

\begin{remark}[The role of positivity of $\overline{\rho}_h$]
    Our analysis relies upon strict positivity of the semi-discrete mean-field limit $\overline{\rho}_h(t)$. Specifically, it is required for the proof of Proposition~\ref{prop: exp-decay-negative-part}, in which we show the size of the negative part of the solution to \eqref{eq: dg-dk} is exponentially small. Unfortunately, a SIP-dG discretisation of the heat equation does not obey a discrete maximum principle, which would bound the solution by the extrema of its initial datum. In fact, we do not even have the weaker property that a non-negative initial datum yields a non-negative solution. We must therefore establish strict positivity by other means. We choose to use an upper bound on the mesh size defined via $h_0(\delta)$ in Assumption~\ref{ass: dG2}. The proof that this implies strict positivity is given in Corollary~\ref{cor: ass-corollary}.
\end{remark}

\begin{remark}[The role of $\delta$]
    The parameter $\delta \in (0, \rhomin)$ in Assumption~\ref{ass: dG2} controls how much the semi-discrete mean-field limit can deviate from the minimum of the positive initial datum. Choosing $\delta$ close to the upper limit loosens the mesh restriction but worsens the constants in \eqref{eq: E(N, h)} (and vice versa). If one wishes, $\delta$ could be fixed once and for all (e.g., $\delta = \rhomin/2$) to remove this degree of freedom.
\end{remark}

\begin{corollary}\label{cor: ass-corollary}
    Assume the validity of Assumption~\ref{ass: dG1}. Fix $\delta \in (0, \rhomin)$. Then, for all $p \in \mathbb{N}$ and $d \in \{1,2,3\}$, there exists $h_0(\delta) > 0$ such that, for $h \leq h_0(\delta)$, we have the following bound:
    \begin{equation}
        0 < \rho_{m, \delta} \leq \sdmfl(t, \bx) \leq \rho_{M, \delta} < \infty \quad \forall (t, \bx) \in [0, T] \times \btd,
    \end{equation}
    where we have set $\rho_{m, \delta} \coloneqq \rhomin - \delta$ and $\rho_{M, \delta} \coloneqq \rhomax + \delta$.
    In particular, since $\sdmfl(0) = \rho_{0, h}$, the discrete initial datum satisfies
    \begin{equation}\label{eq: rho0h-two-sided}
        \rho_{m, \delta} \leq \rho_{0, h}(\bx) \leq \rho_{M, \delta} \quad \forall \bx \in \btd.
    \end{equation}
\end{corollary}

\begin{proof}
    For any $\bx \in \btd$ and $t \in [0, T]$ we estimate
    \begin{align*}
        \sdmfl(t, \bx) &= \mfl(t, \bx) - (\mfl(t, \bx) - \sdmfl(t, \bx)), \\
        &\leq \rhomax + \sup_{t \in [0, T]}\norm{\sdmfl(t, \cdot) - \mfl(t, \cdot)}_{L^\infty}, \\
        &\leq \rhomax + \sup_{t \in [0, T]}\norm{\mfl(t, \cdot) - \mathcal{I}_h \mfl(t, \cdot)}_{L^\infty} + \sup_{t \in [0, T]}\norm{\mathcal{I}_h \mfl(t, \cdot) - \sdmfl(t, \cdot)}_{L^\infty},
    \end{align*}
    using the maximum principle for the continuous heat equation. The first term can be settled using a standard interpolation estimate in $L^\infty$~\cite[Theorem 11.13]{ern2021finite} and the maximum principle for the continuous heat equation:
    $$
    \sup_{t \in [0, T]}\norm{\mfl(t, \cdot) - \mathcal{I}_h \mfl(t, \cdot)}_{L^\infty} \leq Ch^{p+1}\norm{\mfl_0}_{W^{p+1, \infty}}.
    $$
    For the second term, we apply an $L^\infty$-$L^2$ inverse inequality~\cite[Lemma 12.1]{ern2021finite}, use the triangle inequality and use another standard interpolation estimate~\cite[Theorem 11.13]{ern2021finite} in $L^2$, coupled with the maximum principle. We find that
    \begin{align*}
        \sup_{t \in [0, T]}\norm{\mathcal{I}_h \mfl(t, \cdot) - \sdmfl(t, \cdot)}_{L^\infty} &\lesssim h^{-\frac d 2} \sup_{t \in [0, T]} \brack{\norm{\mathcal{I}_h \mfl(t, \cdot) - \mfl(t, \cdot)} + \norm{\mfl(t, \cdot) - \sdmfl(t, \cdot)}}, \\
        &\lesssim h^{-\frac d 2} \Big(Ch^{p+1}\norm{\mfl_0}_{H^{p+1}} + \sup_{t \in [0, T]}\norm{\mfl(t, \cdot) - \sdmfl(t, \cdot)}\Big).
    \end{align*}
    By Theorem~\ref{thm: SIP-parabolic-L2}, we have 
    $$
    \sup_{t \in [0, T]}\norm{\mfl(t, \cdot) - \sdmfl(t, \cdot)} \leq Ch^{p+1}\norm{\mfl_0}_{H^{p+2}} + \norm{\mfl_0 - \sdmfl(0)},
    $$
    so it remains to estimate the initial discrepancy. This is done using Assumption~\ref{ass: dG1} giving
    $$
    \sdmfl(t, \bx) \leq \rhomax + C\brack{h^{p+1} + h^{p+1-d/2}}.
    $$
    Since $h \leq 1$ and $d/2 > 0$, we have $h^{p+1} \leq h^{p+1-d/2}$, so the correction is bounded by $2Ch^{p+1-d/2}$. One can then choose $h$ such that
    $
    h \lesssim \delta^{1/(p+1-d/2)},
    $
    making this correction at most $\delta$, which gives the upper bound of the claim. An analogous argument gives the lower bound.
\end{proof}

\begin{remark}[Maximum-norm stability]\label{rem: maxnorm-stability}
    Corollary~\ref{cor: ass-corollary} uses an $L^\infty_t (L^\infty_x)$ estimate for a SIP-dG discretisation of the heat equation which, to the best of our knowledge, does not currently exist in the literature. Instead, we obtain this using an $L^2$-$L^\infty$ inverse estimate coupled with an $L^\infty_t(L^2_x)$ error estimate. This incurs a factor of $h^{-d/2}$, which leads to a more restrictive upper bound on the mesh size, particularly for $d \in \{2, 3\}$. Recently, in~\cite{chen2025maximum}, the authors prove maximum-norm stability of the semi-discrete semigroup associated to a general parabolic problem for an interior-penalty hybridised dG method on polyhedral domains from which $L^\infty_t (L^\infty_x)$ estimates follow. These estimates have order $O(\absv{\log h}^{\alpha_d+1} h^{p+1})$ for some $\alpha_d > 0$ depending on the dimension. They remark that their techniques should also apply to the SIP-dG setting which would improve our result by removing the $h^{-d/2}$ penalty thus weakening the mesh-size restriction in Assumption~\ref{ass: dG2}. The same estimates would serve the adjoint heat flows of Section~\ref{subsec: adjoint-heat-flow} equally well. There we take a different route, avoiding $L^\infty_t(L^\infty_x)$ bounds altogether by means of an $L^\infty$-$L^1$ H\"older inequality (Section~\ref{subsec: proof-outline}); that route has the advantage of incurring no logarithmic factor. We state the consequences of a SIP-dG extension in Conditional Assumption~\ref{ass: special}.

    \end{remark}

    \begin{conditionalassumption}\label{ass: special}
        Assume that the result of~\cite{chen2025maximum} extends to SIP-dG, with appropriate adaptation to periodic domains. That is, we have an $L^\infty_t(L^\infty_x)$ estimate of the form (cf.~\cite[Theorem 1.2]{chen2025maximum})
        $$
        \sup_{t \in [0, T]}\norm{\overline{\rho} - \overline{\rho}_h}_{L^\infty} \lesssim \absv{\log h}^{\alpha_d + 1}h^{p+1}.
        $$
        Then Assumption~\ref{ass: dG2} can be improved so that the mesh restriction $h_0(\delta)$ scales like
        $$
        \absv{\log h_0}^{\alpha_d + 1}h_0^{p+1} \lesssim \delta.
        $$
        That is, \emph{if the extension of~\cite{chen2025maximum} to SIP-dG holds}, the scaling is improved from $p+1-d/2$ to $p+1$ up to a logarithmic factor.
    \end{conditionalassumption}

\subsection{Statement of the weak error estimate}

Our weak error estimate quantifies the error between fluctuations of the underlying particle system \eqref{eq: empirical-measure} and fluctuations of the semi-discrete discontinuous Galerkin scheme proposed for the Dean--Kawasaki equation \eqref{eq: dg-dk}, measured in a suitably weak metric. Specifically, since the former is a measure and the latter is a density they cannot be compared pathwise. Instead, we test both objects against a suitable class of test functions to produce `observables' and then look at the error between arbitrary moments of these quantities. This is stated rigorously in the following result:

\begin{theorem}\label{thm: main-theorem}
Assume the validity of Assumption~\ref{ass: dG1} and Assumption~\ref{ass: dG2}. In particular, by Corollary~\ref{cor: ass-corollary} the mean-field limit $\overline{\rho}_h$ \eqref{eq: mfl} satisfies $0 < \rhomin - \delta \leq \overline{\rho}_h \leq \rhomax + \delta$ for positive $\rhomin, \rhomax$ and some $\delta \in (0, \rhomin)$ on $[0, T]$. Let $\rho_h(t)$ be a (probabilistically weak) solution of the discontinuous Galerkin Dean--Kawasaki approximation given in \eqref{eq: dg-dk} on $[0, T]$. Fix a vector of times $\boldsymbol{T} = (T_1, \ldots, T_M) \in [0, T]^M$, a vector of non-negative integers $\boldsymbol{j} = (j_1, \ldots, j_M)$ with $j \coloneqq \absv{\bj} = \sum_{m=1}^M j_m \geq 1$, and a vector $\boldsymbol{\varphi} = (\varphi_1, \ldots, \varphi_M) \in [W^{p+j, \infty} \cap H^{p+j+1}]^M$. Then the difference of moments between $\rho_h$ and the empirical density $\mu^N$ \eqref{eq: empirical-measure} is given by
    \begin{multline*}
        \left \lvert \E{\prod_{m=1}^M \sbrack{N^{1/2}\brack{\rho_h(T_m) - \E{\rho_h(T_m)}, \mathcal{I}_h \varphi_m}}^{j_m}} \right. \\
        - \left. \E{\prod_{m=1}^M \sbrack{N^{1/2}\abrack{\mu_{T_m}^N - \E{\mu_{T_m}^N}, \varphi_m}}^{j_m}} \right \rvert \leq \mathrm{Err}_{\mathrm{num}} + \mathrm{Err}_{\mathrm{neg}},
    \end{multline*}
    where $\mathrm{Err}_{\mathrm{neg}}$ is the error coming from the negative part of the solution and $\mathrm{Err}_{\mathrm{num}}$ is the numerical error due to the spatial discretisation. In particular, we have
    \begin{align*}
    \mathrm{Err}_{\mathrm{num}} &= h^{p} (C(d, \delta, \rhomin, \rhomax))^{j/2}\brack{\prod_{m=1}^M T_m^{j_m/2}}j^{C_1j+C_2}\\
    &\times \brack{\prod_{m=1}^M \brack{\norm{\varphi_m}_{W^{p+j, \infty}} + \norm{\varphi_m}_{H^{p+j+1}}}^{j_m}}, \\
    \mathrm{Err}_{\mathrm{neg}} &= (C(d, \delta, \rhomin, \rhomax))^{j/2}\brack{\prod_{m=1}^M T_m^{j_m/2}}j^{C_3j+C_4}\brack{\prod_{m=1}^M \norm{\varphi_m}^{j_m}_{H^{p+j}}}\mathcal{E}(N, h),
    \end{align*}
    for constants $C_1, C_2, C_3, C_4 > 0$ independent of $j, N, T, h$ and where
    \begin{equation}\label{eq: E(N, h)}
    \mathcal{E}(N, h) \coloneqq C(d, \delta, \rhomin, \rhomax) \brack{\exp \brack{-\frac{\rho_{m, \delta}N^{1/2}h^{d/2}}{C \rho_{M, \delta}^{1/2}}} + \exp \brack{-ch^{-1}}}.
    \end{equation}
\end{theorem}

We now give some intuition:
\begin{itemize}
    \item $\mathrm{Err}_{\mathrm{num}}$ is the numerical error arising due to the spatial discretisation. Specifically, we constructed our semi-discrete noise to preserve the cross-variation structure of the full noise (see Proposition~\ref{prop: existence-of-noise}) and this structure involves products of gradients (see Section~\ref{subsec: discretisation-of-noise}). Our convergence analysis then relies upon gradient estimates for an appropriate adjoint heat equation and its numerical approximation. Such objects are estimated in a `dG energy norm' whose rate of convergence is limited to $O(h^p)$.
    \item $\mathrm{Err}_{\mathrm{neg}}$ arises since our scheme does not preserve positivity, and it therefore measures the expected maximum size of the negative part of the solution. Crucially, the estimate is exponentially small in the scaling regime $Nh^d \gg 1$ corresponding to there being, on average, many particles per grid cell.
\end{itemize}

\begin{remark}
    We expect our result to extend to arbitrary functionals (rather than just polynomial moments) of the fluctuation observables by appropriate adaptation of the considerations in~\cite{cornalba2023dean, cornalba2026density}. We have chosen to present the arguments in the technically simpler setting of polynomial moments since the focus of this work is on the application of the discontinuous Galerkin method. 
\end{remark} 

\subsection{Outline of the proof}\label{subsec: proof-outline}

The core techniques used to prove Theorem~\ref{thm: main-theorem} are similar in spirit to those presented in~\cite{cornalba2023dean} for finite difference discretisations. We now outline the skeleton of the proof, highlighting where the use of a discontinuous Galerkin scheme requires changes to the analysis. We then state and prove a sequence of results culminating in the proof of Theorem~\ref{thm: main-theorem}. Key proofs are in the main text whilst any proofs which share significant similarities with their counterpart in~\cite{cornalba2023dean} are included in Appendix~\ref{app: proofs}.

\subsubsection*{Cross-variation structure}

The core observation that we exploit is that the noise in the Dean--Kawasaki equation has the following cross-variation structure (see Section~\ref{subsec: discretisation-of-noise}):
$$
\dd \abrack{\int_{\btd}\sqrt{\rho(t)} \bxi \cdot \nabla v_1 \dd \bx, \int_{\btd} \sqrt{\rho(t)} \bxi \cdot \nabla v_2 \dd \bx} = (\rho(t), \nabla v_1 \cdot \nabla v_2) \dd t.
$$
In Proposition~\ref{prop: existence-of-noise}, we then explicitly constructed a semi-discrete noise term $\mathcal{W}$ which preserves this cross-variation structure \eqref{eq: crossvar-target}, up to the negative part of the solution.
Our analysis will then boil down to comparing these two objects for suitably chosen test functions $v_i$ and their semi-discrete counterparts $v_{i, h}$.

\subsubsection*{Choice of test functions}

We choose our test functions as solutions of adjoint (backwards) heat equations (see Section~\ref{subsec: adjoint-heat-flow}) since these compensate the drift in the Dean--Kawasaki equation thereby isolating the noise term. This reduces our analysis to gradient estimates for the continuous and dG semi-discrete adjoint heat equations which we compile in Appendix~\ref{sec: heat-estimates}.

\subsubsection*{Lack of $L^\infty(L^\infty)$ estimates for dG methods}

The analysis in~\cite{cornalba2023dean} relied heavily on $L^\infty_t(L^\infty_x)$ estimates for finite difference discretisations of the adjoint heat equation. These are, to the best of our knowledge, unavailable in the dG setting and we must adjust our proof strategy; see Remark~\ref{rem: maxnorm-stability} for recent progress in this direction and for why we do not rely on it. In a nutshell, when we encounter a term of the form \eqref{eq: crossvar-target} we use an $L^\infty$-$L^1$ H\"older inequality to put an $L^\infty_t(L^\infty_x)$ norm on our semi-discrete Dean--Kawasaki solution $\rho_h(t)$ and an $L^\infty_t(L^1_x)$ norm on the product of gradients $\dgrad \phi_{1, h}^t \cdot \dgrad \phi_{2, h}^t$. The latter is controlled by utilising $L^\infty_t(L^2_x)$ gradient estimates for the semi-discrete adjoint heat equation (see Appendix~\ref{sec: heat-estimates}), and the $L^\infty_t(L^\infty_x)$ norm of $\rho_h(t)$ is dealt with using Proposition~\ref{prop: Linfty-Linfty}. This modification is purely technical rather than conceptual but is required for much of the convergence analysis.

\subsubsection*{Recursion}

Our weak error result essentially compares moments of the particle system to moments of the semi-discrete Dean--Kawasaki model. Informally, denoting moment errors of order $j$ by $\mathcal{M}(j)$, the idea is to set up a recursion of the form
$$
\mathcal{M}(j) \leq \mathcal{M}(j-1) + \mathcal{M}(j-2) + \mathrm{Err}_{\mathrm{num}} + \mathrm{Err}_{\mathrm{neg}},
$$
where the order of moments is linked to the regularity of the test functions used in the analysis. Exhausting and closing off the recursion accumulates the error terms and gives Theorem~\ref{thm: main-theorem}.

\subsection{Backwards (adjoint) heat flows}\label{subsec: adjoint-heat-flow}

We set up some notation for adjoint heat flows since these form a core part of our convergence analysis. In particular, let $\varphi$ be a sufficiently regular final datum and $T>0$ be a final time. Then $\phi^t$ denotes the solution to the following adjoint heat equation:
\begin{equation}\label{eq: adjoint-heat}
    -\partial_t \phi^t = \frac{1}{2}\Delta \phi^t, \quad \text{on}~ \btd \times [0, T),
\end{equation}
with $\phi^T \coloneqq \varphi$. We also consider a symmetric interior penalty semi-discretisation of \eqref{eq: adjoint-heat}. In particular, we define the following problem: find $\phi_h^t(\cdot) \in V_p$ for $t \in [0, T)$ such that
\begin{equation}\label{eq: SIP-adjoint-heat}
    (\partial_t \phi_h^t, v_h) - a_h(\phi_h^t, v_h) = 0,
\end{equation}
holds for all $v_h \in V_p$ with final datum $\phi_h^T \coloneqq \mathcal{I}_h \varphi$ and where $a_h(\cdot, \cdot)$ is the SIP bilinear form defined in Definition~\ref{def: SIP-bilinear-form}. It will also be convenient to introduce the notation $\mathcal{P}^z(\varphi) \coloneqq \phi^{T-z}$ (respectively $\mathcal{P}_h^z(\mathcal{I}_h \varphi) \coloneqq \phi_h^{T-z}$) to emphasise that $\mathcal{P}^z(\varphi)$ (respectively $\mathcal{P}_h^z(\mathcal{I}_h \varphi)$) is the result of evolving the adjoint heat equation (respectively discrete adjoint heat equation) starting from $\varphi$ (respectively $\mathcal{I}_h \varphi$) for a timespan $z$. In Appendix~\ref{sec: heat-estimates}, we collect various error estimates for \eqref{eq: adjoint-heat} and \eqref{eq: SIP-adjoint-heat} that are used in our subsequent convergence analysis.

\subsection{$L^\infty(L^\infty)$ estimates}

We now show that under the $(N, h)$ scaling defined in Assumption~\ref{ass: dG2} we can obtain $L^\infty(L^\infty)$ estimates for the solution $\rho_h$ of \eqref{eq: dg-dk}.

\begin{proposition}\label{prop: Linfty-Linfty}
    Under Assumptions~\ref{ass: dG1} and~\ref{ass: dG2} we have the following estimates:
    \begin{multline}\label{eq: Linfty-Linfty-1}
        \mathbb{P}\brack{\sup_{t \in [0, T]} \norm{\rho_h(t) - \E{\rho_h(t)}}_{L^\infty} \geq B \frac{\rho_{m, \delta}}{4}} \\
        \leq C \exp \brack{- \frac{\rho_{m, \delta}B^{1/2}N^{1/2}h^{d/2}}{C \rho_{M, \delta}^{1/2}}} + C \exp \brack{- cB^{1/4}h^{-1}},
    \end{multline}
    for any $B \geq 1$, where $\rho_{m, \delta}$ and $\rho_{M, \delta}$ are as in Corollary~\ref{cor: ass-corollary}. Moreover, we can deduce the moment bounds
    \begin{align}
        \E{\sup_{t \in [0, T]}\norm{\rho_h(t) - \E{\rho_h(t)}}_{L^\infty}^j}^{1/j} &\leq C(d, \delta, \rhomin, \rhomax)j^4, \label{eq: Linfty-Linfty-4}\\
        \E{\sup_{t \in [0, T]}\norm{\rho_h(t)}_{L^\infty}^j}^{1/j} &\leq C(d, \delta, \rhomin, \rhomax)j^4, \label{eq: Linfty-Linfty-5}
    \end{align}
    for any $j \geq 1$.
\end{proposition}

\begin{proof}[Sketch proof]
    The proof is a modification of~\cite[Proposition 13]{cornalba2023dean}. We sketch the core idea and defer the details of the proof to Appendix~\ref{app: proofs}.
    \begin{itemize}
        \item We first consider the probability of the solution $\rho_h$ deviating too far from the average (mean-field) profile $\E{\rho_h} = \overline{\rho}_h$ at a single space-time point.
        \item This estimate is then extended to a finite set of space-time points using a union bound argument.
        \item To extend the estimate to all points in space, we use the technical Lemma~\ref{lem: norm-equivalence}.
        \item To extend the estimate to all points in time, we further prove that $\rho_h$ cannot deviate too far from $\E{\rho_h}$ on time intervals between the finite set of times for which we already have the estimate.
        \item These estimates are exponentially small in (roughly speaking) the regime $Nh^d \gg 1$ so using Assumption~\ref{ass: dG2} allows us to obtain the uniform-in-$h$ $L^\infty(L^\infty)$ results.
    \end{itemize}
\end{proof}

Using the $L^\infty(L^\infty)$ estimates, it is then possible to prove that the expected size of the negative part of the solution is exponentially small in the scaling regime $Nh^d \gg 1$. The proof is a direct modification of~\cite[Proposition 13]{cornalba2023dean} except we choose to bound the size of the negative part in $L^\infty$ rather than in $L^2$. Firstly, by using $L^\infty$ we recover the $L^2$ setting of~\cite{cornalba2023dean} since $L^\infty \subset L^2$. More importantly, having the estimate in $L^\infty$ means we avoid using an $L^\infty$-$L^2$ inverse inequality in the proof of Lemma~\ref{lem: second-moments} which stops us from picking up an $h^{-d}$ factor that would weaken our result.

\begin{proposition}\label{prop: exp-decay-negative-part}
    In the same setting as Proposition~\ref{prop: Linfty-Linfty}, we have the following estimate
    \begin{multline}
        \E{\sup_{t \in [0, T]}\norm{\rho_h^-(t)}^2_{L^\infty}}^{1/2} \\ 
        \leq C(d, \delta, \rhomin, \rhomax) \brack{\exp \brack{-\frac{\rho_{m, \delta}N^{1/2}h^{d/2}}{C \rho_{M, \delta}^{1/2}}} + \exp \brack{-ch^{-1}}}.
    \end{multline}
\end{proposition}

\begin{proof}
    Observe that
    $$
    \sup_{t \in [0, T]}\norm{\rho_h^-(t)}_{L^\infty}^2 \leq \mathbbm{1}_{\left\{\sup_{t \in [0, T]}\norm{\rho_h(t) - \E{\rho_h}(t)}_{L^\infty} \geq \rho_{m, \delta}\right\}}\sup_{t \in [0, T]} \norm{\rho_h(t) - \E{\rho_h}(t)}_{L^\infty}^2.
    $$
    Taking expectations and using Cauchy--Schwarz gives
    \begin{multline*}
        \E{\sup_{t \in [0, T]}\norm{\rho_h^-(t)}_{L^\infty}^2} \leq \E{\mathbbm{1}_{\left\{\sup_{t \in [0, T]}\norm{\rho_h(t) - \E{\rho_h}(t)}_{L^\infty} \geq \rho_{m, \delta}\right\}}}^{1/2}, \\
        \times \E{\sup_{t \in [0, T]} \norm{\rho_h(t) - \E{\rho_h}(t)}_{L^\infty}^4}^{1/2}.
    \end{multline*}
    We can rewrite this as
    \begin{multline*}
        \E{\sup_{t \in [0, T]}\norm{\rho_h^-(t)}_{L^\infty}^2} \leq \mathbb{P}\brack{{\left\{\sup_{t \in [0, T]}\norm{\rho_h(t) - \E{\rho_h}(t)}_{L^\infty} \geq \rho_{m, \delta}\right\}}}^{1/2}, \\ 
        \times \E{\sup_{t \in [0, T]} \norm{\rho_h(t) - \E{\rho_h}(t)}_{L^\infty}^4}^{1/4} \E{\sup_{t \in [0, T]} \norm{\rho_h(t) - \E{\rho_h}(t)}_{L^\infty}^4}^{1/4}.
    \end{multline*}
    Using \eqref{eq: Linfty-Linfty-1} with $B=4$ and \eqref{eq: Linfty-Linfty-4} with $j=4$ then square-rooting gives the result.
\end{proof}

\subsection{Setting up the recursive structure}

For notational convenience, we detail the proof of Theorem~\ref{thm: main-theorem} in the case of equal final times $T_1 = \cdots = T_M = T$. We introduce some useful notation following~\cite{cornalba2023dean}. For $t \leq T$, we define
$$
\mathcal{T}_N(\varphi, T, t) \coloneqq \abrack{\mu_t^N, \phi^t} - \abrack{\mu_0^N, \phi^0},
$$
and
$$
\mathcal{S}_N(\mathcal{I}_h\varphi, T, t) \coloneqq (\rho_h(t), \phi_h^t) - (\rho_h(0), \phi_h^0).
$$
Here $\phi^t$ (respectively, $\phi_h^t$) solves the adjoint heat equation \eqref{eq: adjoint-heat} (respectively, semi-discrete adjoint heat equation \eqref{eq: SIP-adjoint-heat}) with final datum $\varphi$ (respectively, $\mathcal{I}_h \varphi$). Given a multi-index $\bj = (j_1, \ldots, j_M)$, and a set of test functions $\boldsymbol{\varphi} = (\varphi_1, \ldots, \varphi_M)$, we define
\begin{equation}\label{eq: SN-TN}
\mathcal{T}_N^{\bj}(\boldsymbol{\varphi}, T, t) \coloneqq \prod_{m=1}^M \mathcal{T}_N^{j_m}(\varphi_m, T, t), \quad \mathcal{S}_N^{\bj}(\mathcal{I}_h \boldsymbol{\varphi}, T, t) \coloneqq \prod_{m=1}^M \mathcal{S}_N^{j_m}(\mathcal{I}_h \varphi_m, T, t).
\end{equation}
Finally, define the quantity
$$
\mathcal{D}(\bj, \boldsymbol{\varphi}, T) \coloneqq \absv{\E{\mathcal{S}_N^{\bj}(\mathcal{I}_h \boldsymbol{\varphi}, T, T)} - \E{\mathcal{T}_N^{\bj}(\boldsymbol{\varphi}, T, T)}}.
$$

In Appendix~\ref{sec: recursive}, we compute the It\^o differentials of the quantities $\mathcal{S}_N^{\bj}$ and $\mathcal{T}_N^{\bj}$. We now use these to prove a series of preliminary results that will allow us to show Theorem~\ref{thm: main-theorem}. We first show that first moments have zero expectation.

\begin{lemma}\label{lem: first moments}
    The first moments of the discontinuous Galerkin Dean--Kawasaki equation \eqref{eq: dg-dk} agree with those of the Brownian particle system. That is, 
    $$
    \E{\mathcal{S}_N(\mathcal{I}_h \varphi, T, T)} = \E{\mathcal{T}_N(\varphi, T, T)} = 0.
    $$
\end{lemma}

\begin{proof}
    This follows immediately from Lemma~\ref{lem: big-moments-lemma} since neither $\mathcal{S}_N(\mathcal{I}_h \varphi, T, t)$ nor $\mathcal{T}_N(\varphi, T, t)$ admit drift.
\end{proof}

Next we provide a suitable estimate on second moments. It is here where we can clearly see how the two types of errors in Theorem~\ref{thm: main-theorem} arise.

\begin{lemma}\label{lem: second-moments}
    Assume the validity of Assumptions~\ref{ass: dG1} and~\ref{ass: dG2}. Let $\varphi_1, \varphi_2 \in W^{p+2, \infty} \cap H^{p+3}$, and recall the definition of $\mathcal{E}(N, h)$ in \eqref{eq: E(N, h)}. Then we have the following estimate
    \begin{multline}
        \absv{\E{\mathcal{S}_N(\mathcal{I}_h\varphi_1, T, T)\mathcal{S}_N(\mathcal{I}_h \varphi_2, T, T)} - \E{\mathcal{T}_N(\varphi_1, T, T)\mathcal{T}_N(\varphi_2, T, T)}} \\
        \leq CN^{-1}T\norm{\varphi_1}_{H^{p+2}}\norm{\varphi_2}_{H^{p+2}}\mathcal{E}(N, h) 
        \\ + C(d, \delta, \rhomin, \rhomax)N^{-1}Th^{p}\brack{\norm{\varphi_1}_{H^{p+3}}\norm{\varphi_2}_{H^{p+3}} + \norm{\varphi_1}_{W^{p+2, \infty}}\norm{\varphi_2}_{W^{p+2, \infty}}}.
    \end{multline}
\end{lemma}

\begin{proof}
    Define $r_h^t \coloneqq \dgrad \phi_{1, h}^t \cdot \dgrad \phi_{2, h}^t - \mathcal{I}_h\brack{\nabla \phi_1^t \cdot \nabla \phi_2^t}$. Using Lemma~\ref{lem: big-moments-lemma}, we write
    $$
    \E{\mathcal{S}_N(\mathcal{I}_h\varphi_1, T, T)\mathcal{S}_N(\mathcal{I}_h \varphi_2, T, T)} = N^{-1}\int_0^T \E{(\rho_h^+(t), \dgrad \phi_{1, h}^t \cdot \dgrad \phi_{2, h}^t)} \dd t,
    $$
    where we have used that $\mathcal{S}_N(\mathcal{I}_h \varphi, T, 0) = 0$. Upon writing $\rho_h^+(t) = \rho_h(t) + \rho_h^-(t)$ and suitably adding and subtracting terms, we get
    \begin{multline*}
        \E{\mathcal{S}_N(\mathcal{I}_h\varphi_1, T, T)\mathcal{S}_N(\mathcal{I}_h \varphi_2, T, T)} \\
        = N^{-1}\int_0^T \E{(\rho_h(t), \mathcal{I}_h\brack{\nabla \phi_1^t \cdot \nabla \phi_2^t}) - (\rho_h(0), \mathcal{P}_h^t(\mathcal{I}_h(\nabla \phi_1^t \cdot \nabla \phi_2^t)))} \dd t \\
        + N^{-1}\int_0^T \E{(\rho_h(0), \mathcal{P}_h^t(\mathcal{I}_h(\nabla \phi_1^t \cdot \nabla \phi_2^t)))} \dd t \\
        + N^{-1}\int_0^T \E{(\rho_h^-(t), \dgrad \phi_{1, h}^t \cdot \dgrad \phi_{2, h}^t)} \dd t + N^{-1}\int_0^T \E{(\rho_h(t), r_h^t)} \dd t =: \sum_{i=1}^4 A_i.
    \end{multline*}

    Similarly, using Lemma~\ref{lem: big-moments-lemma} and performing a suitable addition and subtraction gives
    \begin{multline*}
        \E{\mathcal{T}_N(\varphi_1, T, T)\mathcal{T}_N(\varphi_2, T, T)}, \\
        = \frac{1}{N}\int_0^T\E{N^{-1}\sum_{r=1}^N \nabla\phi_1^t(\bB_r(t)) \cdot \nabla \phi_2^t(\bB_r(t)) - \frac{1}{N}\sum_{r=1}^N \mathcal{P}^t(\nabla \phi_1^t \cdot \nabla \phi_2^t)(\bB_r(0))} \dd t, \\
        + \frac{1}{N}\int_0^T \E{N^{-1}\sum_{r=1}^N \mathcal{P}^t(\nabla \phi_1^t \cdot \nabla \phi_2^t)(\bB_r(0))} \dd t =: \sum_{i=1}^2 B_i.
    \end{multline*}

    First observe that $A_1 = B_1$ by Lemma~\ref{lem: first moments}. We then estimate
    \begin{align*}
        \absv{A_3} &\leq N^{-1} \int_0^T \E{\absv{(\rho_h^-(t), \dgrad \phi_{1, h}^t \cdot \dgrad \phi_{2, h}^t)}} \dd t, \\
        &\leq N^{-1}\sup_{t \in [0, T]}\norm{\dgrad \phi_{1, h}^t \cdot \dgrad \phi_{2, h}^t}_{L^1}\int_0^T \E{\norm{\rho_h^-(t)}_{L^\infty}} \dd t, \\
        &\leq N^{-1} T \norm{\dgrad \phi_{1, h}^t}_{L^\infty(0, T; L^2)}\norm{\dgrad \phi_{2, h}^t}_{L^\infty(0, T; L^2)} \E{\sup_{t \in [0, T]}\norm{\rho_h^-(t)}_{L^\infty}^2}^{1/2}, \\
        &\leq CN^{-1}T \norm{\varphi_1}_{H^{p+2}}\norm{\varphi_2}_{H^{p+2}}\mathcal{E}(N, h),
    \end{align*}
    where we have used an $L^\infty$-$L^1$ H\"older inequality, Corollary~\ref{cor: discrete-gradient estimates}, Proposition~\ref{prop: exp-decay-negative-part}, and where $\mathcal{E}(N, h)$ is defined in \eqref{eq: E(N, h)}.
    
    For $A_4$, we estimate
    \begin{align*}
        \absv{A_4} &\leq N^{-1} \int_0^T \E{\absv{(\rho_h(t), r_h^t)}} \dd t, \\
        &\leq T N^{-1}\E{\sup_{t \in [0, T]}\norm{\rho_h(t)}_{L^\infty}^2}^{1/2} \sup_{t \in [0, T]}\norm{r_h^t}_{L^1}.
    \end{align*}
    Using the discrete gradient estimates in Corollary~\ref{cor: discrete-gradient estimates} and a standard interpolation estimate~\cite[Theorem 11.13]{ern2021finite}, we then obtain
    \begin{align*}
        \norm{r_h^t}_{L^1} &\leq \norm{\dgrad \phi_{1, h}^t \cdot \dgrad \phi_{2, h}^t - \nabla \phi_1^t \cdot \nabla \phi_2^t}_{L^1} + \norm{\nabla \phi_1^t \cdot \nabla \phi_2^t - \mathcal{I}_h (\nabla \phi_1^t \cdot \nabla \phi_2^t)}_{L^1}, \\
        &\leq Ch^{p}\norm{\varphi_1}_{H^{p+2}}\norm{\varphi_2}_{H^{p+2}}.
    \end{align*}
    Combining this with Proposition~\ref{prop: Linfty-Linfty} gives
    $$
    \absv{A_4} \leq CN^{-1}Th^p C(d, \delta, \rhomin, \rhomax) \norm{\varphi_1}_{H^{p+2}}\norm{\varphi_2}_{H^{p+2}}.
    $$
    We decompose 
    \begin{multline*}
        A_2 - B_2 \\
        = -\frac{1}{N}\int_0^T\brack{N^{-1}\sum_{r=1}^N \mathcal{P}^t(\nabla \phi_1^t \cdot \nabla \phi_2^t)(\bB_r(0)) - (\rho_h(0), \mathcal{I}_h(\mathcal{P}^t(\nabla \phi_1^t \cdot \nabla \phi_2^t)))} \dd t, \\
        + \frac{1}{N}\int_0^T (\rho_h(0), \mathcal{P}_h^t(\mathcal{I}_h(\nabla \phi_1^t \cdot \nabla \phi_2^t)) - \mathcal{I}_h(\mathcal{P}^t(\nabla \phi_1^t \cdot \nabla \phi_2^t))) \dd t =: C_1 + C_2,
    \end{multline*}
    where we have also used Assumption~\ref{ass: dG1} that the initial particle positions $\{\bB_r(0)\}$ and initial density $\rho_h(0)$ are deterministic. To bound $C_1$, we use Assumption~\ref{ass: dG1} with $v=\mathcal{P}^t(\nabla \phi_1^t \cdot \nabla \phi_2^t)$ and stability of the (adjoint) heat semigroup to get
    $\absv{C_1} \leq CN^{-1}Th^{p}\norm{\varphi_1}_{W^{p+2, \infty}}\norm{\varphi_2}_{W^{p+2, \infty}}$. To estimate $C_2$, we first note that $\absv{\rho_h(0)} \leq \rho_{M, \delta}$ by \eqref{eq: rho0h-two-sided} to pull it out of the inner product. We then add and subtract $\mathcal{P}^t(\nabla \phi_1^t \cdot \nabla \phi_2^t)$ to estimate
    $$
    \norm{\mathcal{P}_h^t(\mathcal{I}_h(\nabla \phi_1^t \cdot \nabla \phi_2^t)) - \mathcal{I}_h(\mathcal{P}^t(\nabla \phi_1^t \cdot \nabla \phi_2^t))} \leq Ch^{p+1}\norm{\varphi_1}_{H^{p+3}}\norm{\varphi_2}_{H^{p+3}},
    $$
    using a standard interpolation estimate~\cite[Theorem 11.13]{ern2021finite} and using the result of Theorem~\ref{thm: SIP-parabolic-L2} applied to the adjoint heat equation with final datum $\nabla \phi_1^t \cdot \nabla \phi_2^t$. Altogether, we have
    $$
    \absv{A_2 - B_2} \leq C(\rho_{M, \delta})N^{-1}Th^{p}\brack{\norm{\varphi_1}_{W^{p+2, \infty}}\norm{\varphi_2}_{W^{p+2, \infty}} + \norm{\varphi_1}_{H^{p+3}}\norm{\varphi_2}_{H^{p+3}}}.
    $$
    We then, up to constants, can combine this with the estimate for $\absv{A_4}$ which is also a numerical error of order $O(h^p)$.
\end{proof}

\begin{proposition}[Recursive formula for higher moments]\label{prop: recursive-formula}
    Assume the validity of Assumptions~\ref{ass: dG1} and~\ref{ass: dG2}. Let $\boldsymbol{\varphi} = (\varphi_1, \ldots, \varphi_M) \in [W^{p+2, \infty} \cap H^{p+3}]^M$, fix a vector $\bj = (j_1, \ldots, j_M)$ such that $\absv{\bj}=j$, recall the definition of $\boldsymbol{j}^{ik}$ for each pair $(i, k) \in \{1, \ldots, M\}^2$ given in Lemma~\ref{lem: big-moments-lemma}, and the definition of $\mathcal{E}(N, h)$ in \eqref{eq: E(N, h)}. Then we have the following recursive formula

\begin{multline*}
    \mathcal{D}(\bj, \boldsymbol{\varphi}, T) \leq N^{-1} \sum_{k, \ell=1}^M \frac{(j_k - \delta_{k\ell})j_\ell}{2} \int_0^T \mathcal{D}\brack{\{\bj^{k \ell}; 1\}, \{\boldsymbol{\phi}^t; \nabla \phi_k^t \cdot \nabla \phi_\ell^t\}, t} \dd t \\
    + N^{-1}\rho_{M, \delta} \sum_{k, \ell=1}^M \frac{(j_k - \delta_{k\ell})j_\ell}{2} \norm{\varphi_k}_{H^{p+2}}\norm{\varphi_\ell}_{H^{p+2}} \int_0^T \mathcal{D}\brack{\bj^{k\ell}, \boldsymbol{\phi}^t, t} \dd t \\
    + \brack{CN^{-1}T C(d, \delta, \rhomin, \rhomax)}^{j/2}(2j)^{3(j-2)} \mathcal{E}(N, h) \sum_{k, \ell=1}^M \frac{(j_k - \delta_{k\ell})j_\ell}{2}\prod_{m=1}^M \norm{\varphi_m}_{H^{p+2}}^{j_m} \\
    + h^p (C N^{-1}T C(d, \delta, \rhomin, \rhomax))^{j/2}(2j)^{3(j-2)} \sum_{k, \ell=1}^M \frac{(j_k - \delta_{k\ell})j_\ell}{2} \\
    \times\prod_{m=1}^M \brack{\norm{\varphi_m}_{W^{p+2, \infty}}+\norm{\varphi_m}_{H^{p+3}}}^{j_m} \\
    =: A^{j-1}_{\mathrm{rec}} + A^{j-2}_{\mathrm{rec}} + \mathrm{Err}_{\mathrm{neg}} + \mathrm{Err}_{\mathrm{num}}.
\end{multline*}
\end{proposition}

\subsection{Proof of Theorem~\ref{thm: main-theorem}}

We now have all the machinery in place to prove Theorem~\ref{thm: main-theorem}. The remaining argument is purely algebraic and combinatorial in nature, and is independent of the choice of spatial discretisation. More precisely, Proposition~\ref{prop: recursive-formula} has the same recursive structure as that used in the proof of~\cite[Theorem 3]{cornalba2023dean}, and the essence of the proof is to exhaust the recursion. Doing so uses only this structure and the regularity of the test functions generated along it, and does not use any finite-difference specific properties. As a result, the combinatorial argument of~\cite[Theorem 3]{cornalba2023dean} applies directly to our setting. We therefore only sketch the core idea for completeness.

\begin{proof}[Proof of Theorem~\ref{thm: main-theorem}]
    First note that for $j=1$ and $j=2$ the result follows from Lemma~\ref{lem: first moments} and Lemma~\ref{lem: second-moments} respectively. We now fix $j \geq 3$. Proposition~\ref{prop: recursive-formula} gives
    $$
    \mathcal{D}(\boldsymbol{j}, \boldsymbol{\varphi}, T) \leq A^{j-1}_{\mathrm{rec}} + A^{j-2}_{\mathrm{rec}} + \mathrm{Err}_{\mathrm{neg}} + \mathrm{Err}_{\mathrm{num}}.
    $$
    We interpret this in the following way:
    \begin{enumerate}
        \item Each moment of order $j$ produces a residual associated with the negative part of the solution denoted $\mathrm{Err}_{\mathrm{neg}}$ and a residual associated to the numerical order of the scheme denoted by $\mathrm{Err}_{\mathrm{num}}$.
        \item Each moment of order $j$ is recursively linked to a set of moments of order $j-1$ and a set of moments of order $j-2$ denoted $A_{\mathrm{rec}}^{j-1}$ and $A_{\mathrm{rec}}^{j-2}$ respectively.
        \item Exhausting the recursion amounts to a bound of the following form
        $$
        \mathcal{D}(\boldsymbol{j}, \boldsymbol{\varphi}, T) \leq \sum_{K=0}^{j-2}\mathcal{R}_K,
        $$
        where $\mathcal{R}_K$ is the sum of all residuals associated with the moments explored after exactly $K$ steps.
        \item Each step of the recursion introduces a test function of the form $\nabla \phi_k^t \cdot \nabla \phi_\ell^t$ which requires one derivative. The minimal regularity required is $W^{p+2, \infty} \cap H^{p+3}$ and the recursion has depth $j-2$ which gives the regularity $W^{p+j, \infty} \cap H^{p+j+1}$ featured in Theorem~\ref{thm: main-theorem}.
        \item The result follows upon suitable control of $\mathcal{R}_K$. 
    \end{enumerate}
    Although we have different numerical and negative part residuals to~\cite{cornalba2023dean}, they are of the same generic form. We can therefore, up to different constants, use the same abstract bound for $\sum_{K=0}^{j-2}\mathcal{R}_K$ as in the proof of~\cite[Theorem 3]{cornalba2023dean}. This amounts to a combinatorial argument that explores all possible routes of exhausting the recursion. In particular, each type of residual picks up a factor of the form $j^{Cj +c}$ for constants $C, c> 0$.
\end{proof}

\begin{remark}
    To extend to unequal final times, each adjoint heat flow is evolved over its own time interval $[0, T_m]$. The associated observables $\mathcal{S}_N$ and $\mathcal{T}_N$ are considered as stopped processes over the common time horizon $[0, T]$. The cross-variation for a pair of indices $(k, \ell)$ is then supported on $T_k \wedge T_\ell$. The structure of the proof is otherwise unchanged up to appropriate insertion of indicator functions, and changing integral upper limits to be of the form $T_k \wedge T_\ell$.
\end{remark}

\begin{remark}[Limiters]\label{remark_limiters}
We have proved that our dG scheme only produces a negative solution with exponentially small probability. Even when it does go negative, the solution may still carry non-negative mass on every mesh element, and whenever this is so it is straightforward to post-process and restore non-negativity whilst preserving mass elementwise. For instance, this can be done by using a standard Zhang--Shu type limiter~\cite{zhang2011maximum}: for $v_h \in V_p$ having non-negative element average on element $K$, denote by $ 
\abrack{v_h}_K \coloneqq \absv{K}^{-1} \int_K v_h \dd \bx
$, and $m_K(v_h) \coloneqq \min_{\bx \in \overline{K}} v_h(\bx)$ and let 
$$
\theta_K \coloneqq 
\begin{cases}
    1, \quad &m_K(v_h) \geq 0, \\
    \dfrac{\abrack{v_h}_K}{\abrack{v_h}_K - m_K(v_h)},\quad &m_K(v_h) < 0.
\end{cases}
$$
The non-negative spatial reconstruction map $P_+$ on $K$ is then defined via$$
P_+(v_h) |_K \coloneqq \abrack{v_h}_K + \theta_K(v_h - \abrack{v_h}_K).
$$
We stress that this is a purely post-processing procedure, and that $P_+(v_h)$ is \textit{not} fed back into the semi-discrete evolution. Furthermore, at the current stage, we are not able to improve the estimate for the probability of applicability of the post-processing procedure on all elements beyond the scaling $1-\mathcal{E}(N,h)$ (which is the same rate as that of the global positivity estimates of the scheme, and which directly stems from \eqref{eq: Linfty-Linfty-1}). 
\end{remark}

\section{Numerical experiments}\label{sec: numerics}

\subsection{Matrix formulation of \eqref{eq: dg-dk}}

Let $\{e_i\}_{i=1}^{n_V}$ be a basis for $V_p$, where $n_V \coloneqq \mathrm{dim}(V_p)$, and let $\{\boldsymbol{f}_i\}_{i=1}^{n_G}$ be a basis for $\boldsymbol{V}_{p-1}^d$. Define the mass matrices ${\boldsymbol{M}}_{ij} \coloneqq (e_j, e_i)$ and $(\boldsymbol{M}_\mathrm{G})_{ij} \coloneqq (\boldsymbol{f}_j, \boldsymbol{f}_i)$ on $V_p$ and $\boldsymbol{V}_{p-1}^d$ respectively, the stiffness matrix $\boldsymbol{A}_{ij} \coloneqq a_h(e_j, e_i)$, and the matrix associated to the discrete gradient $\boldsymbol{G}_{ij} \coloneqq (\dgrad e_j, \boldsymbol{f}_i)$. We expand
$$
\rho_h(t) = \sum_{j=1}^{n_V} U_j(t) e_j,
$$
and substitute this into \eqref{eq: dg-dk} which, upon choosing $v_h = e_i$, gives
$$
\sum_{j=1}^{n_V}\dd U_j(t)(e_j, e_i) + \sum_{j=1}^{n_V}U_j(t) a_h(e_j, e_i) \dd t = N^{-1/2}\dd \mathcal{W}(\rho_h^+(t), e_i),
$$
for $i=1, \ldots, n_V$. We next insert the definition of $\boldsymbol{G}$ into the noise construction given in Proposition~\ref{prop: existence-of-noise} to get
$$
\dd \mathcal{W}(\rho_h^+(t), e_j) = \sum_{r=1}^{n_G} \brack{\sum_{k=1}^{n_G} \boldsymbol{G}_{kj}\brack{\boldsymbol{M}_{\mathrm{G}}^{-1}\boldsymbol{R}(\rho_h(t))}_{kr}} \dd B_r,
$$
where we recall that $\boldsymbol{R}(\rho_h)$ is the matrix square root of the weighted mass matrix $\boldsymbol{P}(\rho_h)$ defined in \eqref{eq: Pt-definition} and $\boldsymbol{B}(t) \coloneqq [B_1(t), \ldots, B_{n_G}(t)]^{\mathrm{T}}$ is a vector of one-dimensional Brownian motions. Let $\dd \boldsymbol{W}(t) \coloneqq [\dd \mathcal{W}(\rho_h^+(t), e_1), \ldots, \dd \mathcal{W}(\rho_h^+(t), e_{n_V})]^{\mathrm{T}}$ and the solution vector $\boldsymbol{U}(t) \coloneqq [U_1(t), \ldots, U_{n_V}(t)]^{\mathrm{T}}$. Writing $\boldsymbol{R}(\boldsymbol{U})$ for $\boldsymbol{R}(\rho_h)$ expressed in terms of the coefficient vector gives
$$
\boldsymbol{M}\dd \boldsymbol{U}(t) + \boldsymbol{A}\boldsymbol{U}(t) \dd t = N^{-1/2}\boldsymbol{G}^{\mathrm{T}}\boldsymbol{M}_{\mathrm{G}}^{-1}\boldsymbol{R}(\boldsymbol{U}(t)) \dd \bB(t).
$$
Given a sequence of times $0 = t_0 < t_1 < \cdots < t_{k-1} < t_k = T$ for some $k \in \mathbb{N}$ and writing $\boldsymbol{U}^n \approx \boldsymbol{U}(t_n)$, we can discretise with semi-implicit Euler to get
$$
(\boldsymbol{M} + \Delta t \boldsymbol{A})\boldsymbol{U}^{n+1} = \boldsymbol{M}\boldsymbol{U}^n + N^{-1/2}\sqrt{\Delta t} \boldsymbol{G}^{\mathrm{T}}\boldsymbol{M}_{\mathrm{G}}^{-1}\boldsymbol{R}(\boldsymbol{U}^n) \boldsymbol{Z},
$$
where $\boldsymbol{Z} \in \mathbb{R}^{n_G}$ is a vector of $\mathcal{N}(0, 1)$ normal random variables.

\begin{remark}[Structure of the matrices]
    The mass matrices $\boldsymbol{M}$ and $\boldsymbol{M}_{\mathrm{G}}$ are block-diagonal so the application of themselves and their inverses is an efficient element-local operation. Since $\boldsymbol{R}$ is obtained from a factorisation of the block-diagonal matrix $\boldsymbol{P}$, it too is block diagonal and so it also benefits from efficient application. Note that only the identity $\boldsymbol{X}\boldsymbol{X}^{\mathrm{T}} = \boldsymbol{P}$ enters the cross-variation, so in practice any such factorisation may be used in place of the symmetric positive semi-definite square root of Proposition~\ref{prop: construct-diffusion-coefficients}; we use a Cholesky factorisation. The stiffness matrix $\boldsymbol{A}$ and discrete gradient matrix $\boldsymbol{G}$ have a local sparsity pattern due to inter-element coupling arising from the fluxes.
\end{remark}

\subsection{Implementation details}

We use the Python package Firedrake~\cite{FiredrakeUserManual} to implement our numerical method. The code to reproduce the figures is available at \url{https://github.com/kamran-arora/Dean-Kawasaki-DG}.

\subsection{Error with smooth observables}

To numerically verify Theorem~\ref{thm: main-theorem}, we set $d=1$ and work on the periodic domain $[0, 2\pi]$. We set $\rho_0(x) = (1+0.35\sin(2x+0.4)+0.25\cos(3x-0.7))/(2\pi)$. We let $N=500,000$ and generate $\rho_{0, h}(x)$ and $\{\bB_i(0)\}_{i=1}^N$ so that they satisfy Assumption~\ref{ass: dG1}. We choose $\Delta t=2.5 \times 10^{-4}$ and $T=0.1$ for which we numerically observe that the spatial error we are interested in is dominant. For the smooth observable $\varphi(x) = \sin(2x+1.1)+0.8\cos(3x-0.4)$, we compute the weak error for second moments for $p=1,2$. More precisely, we use Monte Carlo with $500,000$ sample paths to approximate
$$
\absv{\E{\brack{N^{1/2}(\rho_h(T) - \E{\rho_h(T)}, \mathcal{I}_h \varphi)}^2} - \E{\brack{N^{1/2}\abrack{\mu_T^N - \E{\mu_T^N}, \varphi}}^2}}.
$$
In order to reduce the uncertainty in our Monte Carlo estimator, we recall that $\E{\rho_h} = \overline{\rho}_h$ so that we can instead solve a heat equation using the same numerical method. Moreover, the particle moments can be computed in closed-form and then approximated using a Fourier-based method. Specifically, for $N$ independent Brownian motions $\{\bB_i(t)\}$ with deterministic initial conditions and a smooth test $\varphi$, it holds that
$$
\E{\varphi(\bB_i(t))} = \brack{e^{t \Delta/2} \varphi}(\bB_i(0)).
$$
Therefore, the particle moment can be expressed via
$$
\E{\brack{N^{\frac 1 2}\abrack{\mu_T^N - \E{\mu_T^N}, \varphi}}^2} = \frac{1}{N}\sum_{i=1}^N \brack{\brack{e^{\frac{T \Delta}{2}}\varphi^2}(\bB_i(0)) - \brack{e^{\frac{T\Delta}{2}}\varphi}(\bB_i(0))^2}
$$
where the heat semigroup can then be approximated using Fourier coefficients. The results are shown in Figure~\ref{fig: weak-error} for different values of $h=2\pi/n_x$ where $n_x \in \{6, 8, 10, 12, 16, 20, 24, 28, 32\}$. For $p=1$, we see $O(h^2)$ suggesting a numerical super-convergence effect is in play. For $p=2$, we see an initial period of $O(h^2)$ convergence supporting our theoretical results. The later datapoints have hit the Monte Carlo floor so no conclusion about the order of convergence can be made in this regime. It is also worth noting that although we have the same convergence rate for both $p=1$ and $p=2$ due to an apparent numerical super-convergence effect, the $p=2$ solution appears to lead to a smaller overall error compared with $p=1$. We investigate this further in the following section.

\begin{remark}
    Resolving the convergence rate in the fine mesh regime for $p=2$ requires significantly more samples and/or the use of a variance reduction technique. The former is prohibitively expensive and the latter we defer to future investigations (see the discussion of multilevel Monte Carlo in Section~\ref{sec: future}).
\end{remark}

\begin{figure}
    \centering
    \includegraphics[width=0.8\linewidth]{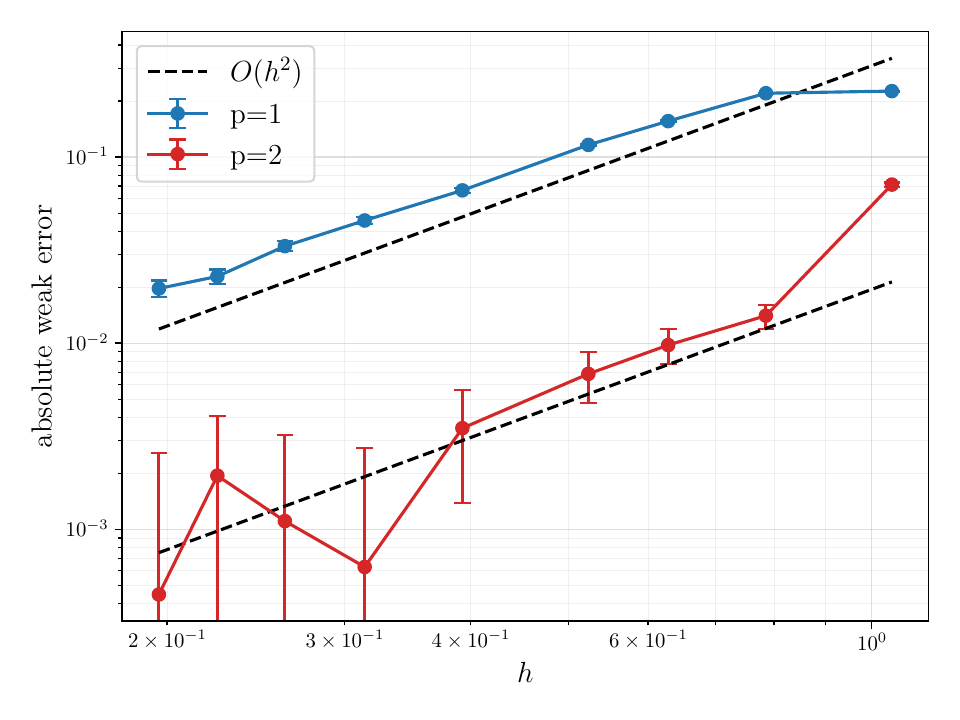}
    \caption{Absolute weak error for second moments using $p=1$ and $p=2$ in $d=1$ for a smooth observable. The fixed parameters are $N=500,000$, $\Delta t = 2.5 \times 10^{-4}$, $T=0.1$, and $500,000$ realisations while the mesh resolution $h$ is varied. The error bars display a $95\%$ confidence interval.}
    \label{fig: weak-error}
\end{figure}

\subsection{Error and cost comparison with discontinuous observables}

Our theoretical result requires sufficient regularity of the test function $\varphi$; however, in practice we would like to be able to use discontinuous observables which fall outside the scope of our theory. For example, taking $\varphi$ to be an indicator function allows one to extract the mass of the solution in a given region. For the next experiments, we choose $\varphi(x) = \mathbbm{1}_{[\pi/2, 3\pi/2]}$ and otherwise use the same experimental setup as before. In Figure~\ref{fig: indicator}\subref{fig: indicator-h}, we again plot the weak error for second moments as a function of the spatial resolution. We observe an order $O(h^{3/4})$ convergence rate for $p=1$ and an order $O(h)$ convergence rate for $p=2$. We cannot link these observations to any theoretical result, but it is not unreasonable to expect that a low-regularity test function leads to lower spatial-convergence rates. Moreover, we again observe that $p=2$ gives a smaller overall error than $p=1$. To investigate whether $p=2$ is truly advantageous, we must fix a computational budget and compare errors against degrees of freedom. Our experiments as described above took, on average, $1.25 \times$ longer to run for $p=2$ than for $p=1$. Therefore, to make the comparison fair we rescale our uncertainties for $p=2$ by $\sqrt{1.25}$. We then plot the experimentally computed errors with manually adjusted confidence intervals against degrees of freedom in Figure~\ref{fig: indicator}\subref{fig: indicator-dofs}. We see that using $p=2$ leads to a smaller overall error at comparable degrees of freedom with a negligible increase in uncertainty, even after accounting for the increased computational cost. Thus, in this particular scenario, the increased cost of a high-order solution is justified by the increased accuracy.

\begin{figure}
    \centering
    \begin{subfigure}[t]{0.48\textwidth}
        \centering
        \includegraphics[width=\textwidth]{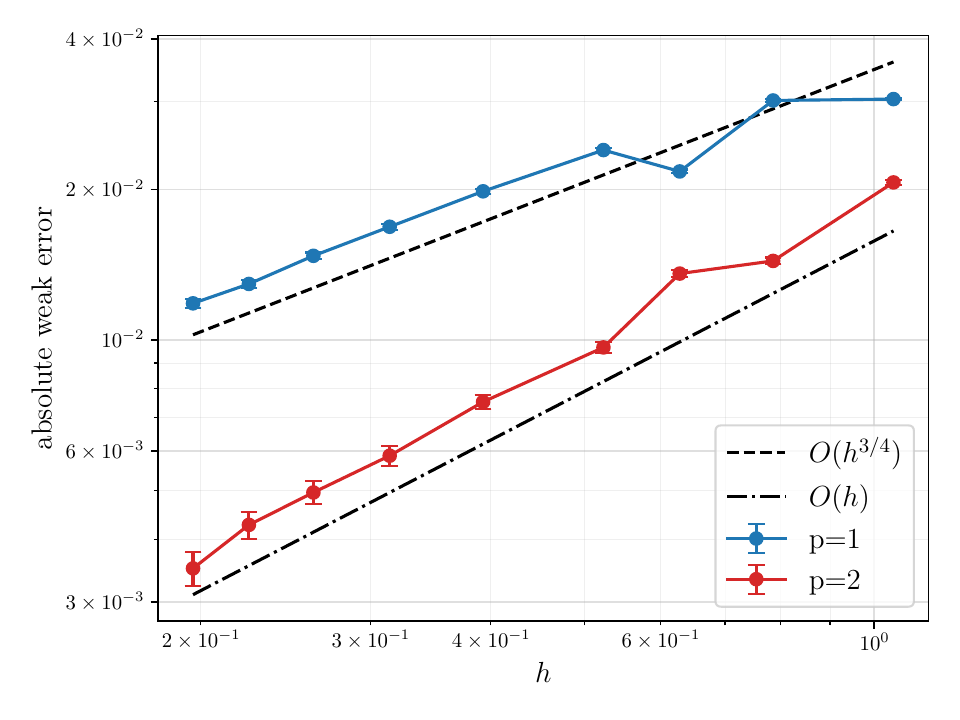}
        \caption{Weak error against \(h\).}
        \label{fig: indicator-h}
    \end{subfigure}
    \hfill
    \begin{subfigure}[t]{0.48\textwidth}
        \centering
        \includegraphics[width=\textwidth]{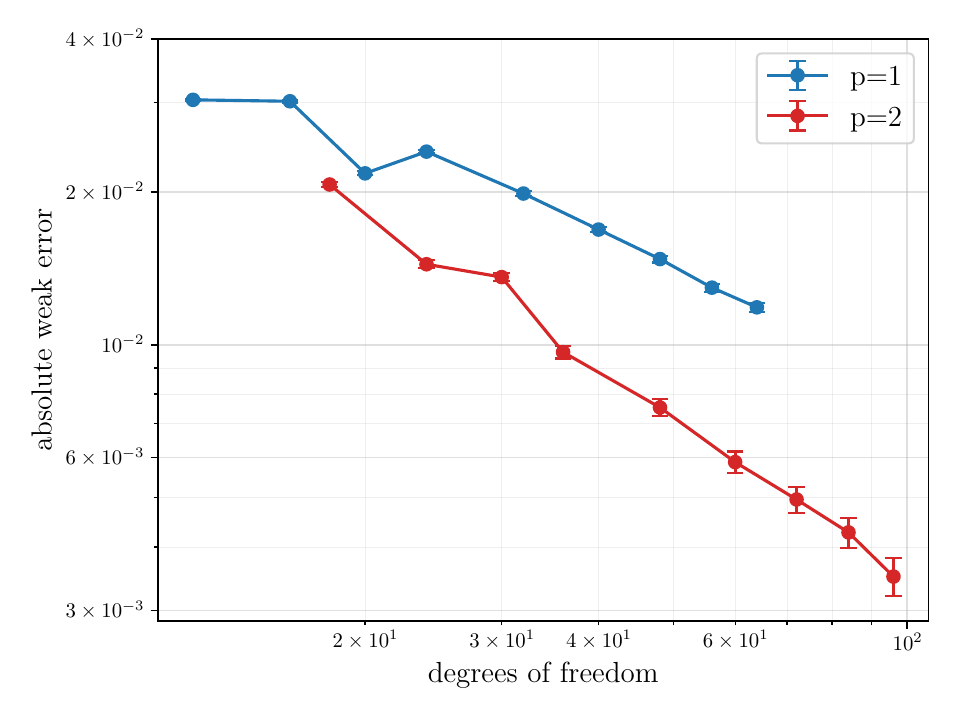}
        \caption{Weak error against DOFs.}
        \label{fig: indicator-dofs}
    \end{subfigure}

    \caption{Absolute weak error for second moments using $p=1$ and $p=2$ in $d=1$ for an indicator function observable. The fixed parameters are $N=500,000$, $\Delta t = 2.5 \times 10^{-4}$, $T=0.1$, and $500,000$ realisations. (A): the mesh resolution $h$ is varied; (B): the degrees of freedom are varied. In both, the error bars display a $95\%$ confidence interval.}
    \label{fig: indicator}
\end{figure}

\subsection{Convection-dominated settings}\label{sec: convection}

To illustrate the advantages of discontinuous Galerkin methods in convection-dominated regimes, we consider a system of $N$ independent Brownian motions moving under the influence of a smooth external potential. That is, we consider the following system of SDEs
$$
\dd \bX_i(t) = -\nabla V(\bX_i(t)) \dd t + \sqrt{2 \sigma}\dd \bB_i(t),
$$
describing the evolution of the particle positions $\bX_i(t) \in \btd$ where $\{\bB_i\}_{i=1}^N$ are independent, standard Brownian motions. The associated Dean--Kawasaki SPDE describing the evolution of the particle density is given by
$$
\partial_t \rho = \sigma \Delta \rho + \nabla \cdot \brack{\rho \nabla V} + N^{-1/2}\nabla \cdot (\sqrt{2 \sigma \rho}\bxi),
$$
where $\bxi$ is a space-time white noise. When the strength of the external potential is large relative to the diffusion, we enter a convection-dominated regime. In this setting, continuous Galerkin schemes can develop spurious oscillations whereas discontinuous Galerkin discretisations can be stabilised using an upwind numerical flux.

Consider the convective term $\nabla \cdot (\boldsymbol{b}u)$. Given two elements $K_\ell$ and $K_r$ that share a facet $F$, we define the upwind value of $u$ with respect to the velocity field $\boldsymbol{b}$ by
$$
u^{\mathrm{up}} \coloneqq u|_{K_{\mathrm{up}}},
$$
where $K_{\mathrm{up}} \in \{K_\ell, K_r\}$ is the element for which $\boldsymbol{b} \cdot \bn_{K_{\mathrm{up}}} \geq 0$. The upwind dG discretisation of the convective term $\nabla \cdot (\rho \nabla V)$ is then given by the following bilinear form
$$
b_h(\rho_h, v_h) \coloneqq -\sum_{K \in \Th} \int_K \rho_h \nabla V \cdot \nabla v_h \dd \bx + \sum_{F \in \Fh}\int_F \rho_h^{\mathrm{up}} \nabla V \cdot \jump{v_h} \dd S.
$$
To derive this, we can multiply $\nabla \cdot (\rho \nabla V)$ by a test function $v$, integrate over a single element, integrate by parts then sum over all elements, to get
$$
\sum_{K \in \Th} \int_K \nabla \cdot (\rho \nabla V)v \dd \bx = -\sum_{K \in \Th} \int_K \rho \nabla V \cdot \nabla v \dd \bx + \sum_{K \in \Th}\int_{\partial K} \rho (\nabla V\cdot \bn_K) v \dd S,
$$
where $\bn_K$ denotes the outward unit normal on $K$. We then project onto our finite-dimensional dG space and seek $\rho_h \in V_p$ such that
$$
\sum_{K \in \Th} \int_K \nabla \cdot (\rho_h \nabla V)v_h \dd \bx = -\sum_{K \in \Th} \int_K \rho_h \nabla V \cdot \nabla v_h \dd \bx + \sum_{K \in \Th}\int_{\partial K} \widehat{\rho}_h (\nabla V \cdot \bn_K) v_h \dd S,
$$
holds for all $v_h \in V_p$ and where $\widehat{\rho}_h$ denotes the numerical flux to be chosen. We choose $\widehat{\rho}_h \coloneqq \rho_h^{\mathrm{up}}$ and, upon rewriting the boundary term as a sum over facets, obtain the bilinear form $b_h(\rho_h, v_h)$.

\begin{figure}
    \centering
    \begin{subfigure}{0.48\textwidth}
        \centering
        \includegraphics[width=\linewidth]{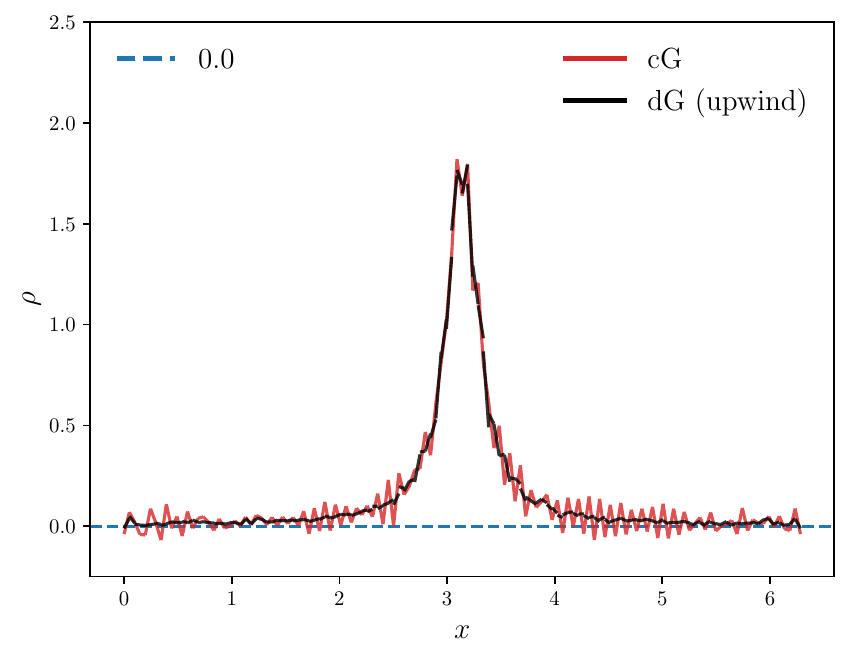}
        \caption{$p=1$.}
        \label{fig: ext_pot_p1}
    \end{subfigure}\hfill
    \begin{subfigure}{0.48\textwidth}
        \centering
        \includegraphics[width=\linewidth]{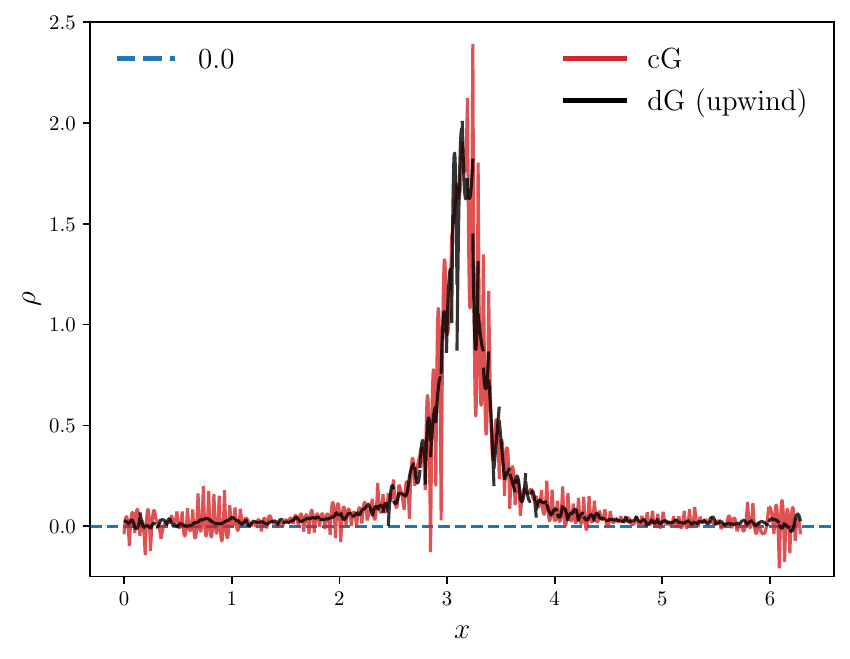}
        \caption{$p=2$.}
        \label{fig: ext_pot_p2}
    \end{subfigure}\hfill
    \caption{Solution to a convection-dominated Dean--Kawasaki SPDE using continuous Galerkin and upwind discontinuous Galerkin. (A): $p=1$; (B): $p=2$.}
    \label{fig: ext_pot}
\end{figure}

As a simple illustrative example, we fix $d=1$, work on $[0, 2 \pi]$ and consider $V(x) = 24 \cos (x)$. This has a minimum at $x=\pi$, so the density will concentrate in the centre of the domain. We set $N=250$ and use a uniform initial condition. The simulation is run for $400$ timesteps on a uniform mesh with $128$ elements. We choose the diffusion parameter $\sigma=0.001$ to ensure we are in a convection-dominated regime. The results for $p=1$ and $p=2$ are shown in Figure~\ref{fig: ext_pot}\subref{fig: ext_pot_p1} and Figure~\ref{fig: ext_pot}\subref{fig: ext_pot_p2} respectively. As expected, the cG solution exhibits large spurious oscillations whereas the dG solution remains stable. Not only are the oscillations unphysical, but they also drive the solution below zero, which is problematic for the Dean--Kawasaki equation given that $\rho$ represents a non-negative particle density.

Moreover, our theoretical analysis for the independent Brownian motion case shows that negativity of the numerical solution is exponentially rare (Proposition~\ref{prop: exp-decay-negative-part}) in an appropriate scaling regime that, roughly speaking, is of the form $Nh^d \gg 1$. This analysis does not cover the present setting, but these experiments suggest a similar result will hold. In particular, the upwind stabilisation appears to suppress the transport-induced oscillations, and thus we only see negativity near the endpoints of the domain where the density is small enough to violate our scaling assumption. In contrast, the cG scheme introduces deterministic undershoots that are unrelated to the natural stochastic fluctuations of the system and which are much larger than the small negative excursions predicted in the purely diffusive setting.

\subsection{Weakly interacting particle systems}

We now replace the external potential of Section~\ref{sec: convection} by an interaction. Let $W$ be a smooth interaction potential and consider
$$
\dd \bX_i(t) = -\frac{1}{N}\sum_{j=1}^N \nabla W(\bX_i(t) - \bX_j(t)) \dd t + \sqrt{2\sigma} \dd \bB_i(t).
$$
The associated Dean--Kawasaki SPDE is
$$
\partial_t \rho = \sigma \Delta \rho + \nabla \cdot \brack{\rho \brack{\nabla W * \rho}} + N^{-1/2}\nabla \cdot (\sqrt{2 \sigma \rho} \bxi).
$$
Writing the convective term as $\nabla \cdot (\boldsymbol{b}[\rho]\rho)$ with the state-dependent velocity field $\boldsymbol{b}[\rho] \coloneqq \nabla W * \rho$, the upwind discretisation applies verbatim with $\nabla V$ replaced by $\boldsymbol{b}[\rho_h]$,
$$
b_h(\rho_h, v_h) \coloneqq -\sum_{K \in \Th} \int_K \rho_h \boldsymbol{b}[\rho_h] \cdot \nabla v_h \dd \bx + \sum_{F \in \Fh}\int_F \rho_h^{\mathrm{up}}\boldsymbol{b}[\rho_h] \cdot\jump{v_h} \dd S.
$$
The upwind value does not appear inside $\boldsymbol{b}$, since $W$ is smooth and hence so is $\nabla W * \rho_h$. The one new difficulty is the convolution, which must be evaluated at each timestep and is non-local; its cost is nevertheless set by the mesh rather than by $N$.

We again take $d=1$ on $[0, 2\pi]$, now with $W(x) = \lambda (1-\cos(x))$ for some $\lambda > 0$. This gives $W'(x) = \lambda \sin(x)$, so the non-local velocity is
$$
\boldsymbol{b}[\rho_h] = W' * \rho_h = \lambda \int_{\btd}\sin (x-y)\rho_h(y) \dd y.
$$
In particular, using a simple trigonometric identity, we have
$$
\boldsymbol{b}[\rho_h] = \lambda \sin(x) \int_{\btd} \cos(y) \rho_h(y) \dd y - \lambda \cos(x) \int_{\btd} \sin(y) \rho_h(y) \dd y,
$$
so we can avoid the evaluation of the convolution in this simple case. The choice of a sinusoidal interaction is closely related to the study of synchronisation (see, for example, the Kuramoto model~\cite{kuramoto2005self}). For small $\lambda \ll \sigma$, we expect the system to converge to a uniform density and, for $\lambda \gg \sigma$, we expect a cluster to form. In the latter case, if we are in a convection-dominated regime, we expect our upwind dG scheme to suppress spurious oscillations that may be present in the cG solution.

\begin{figure}
    \centering
    \includegraphics[width=0.8\linewidth]{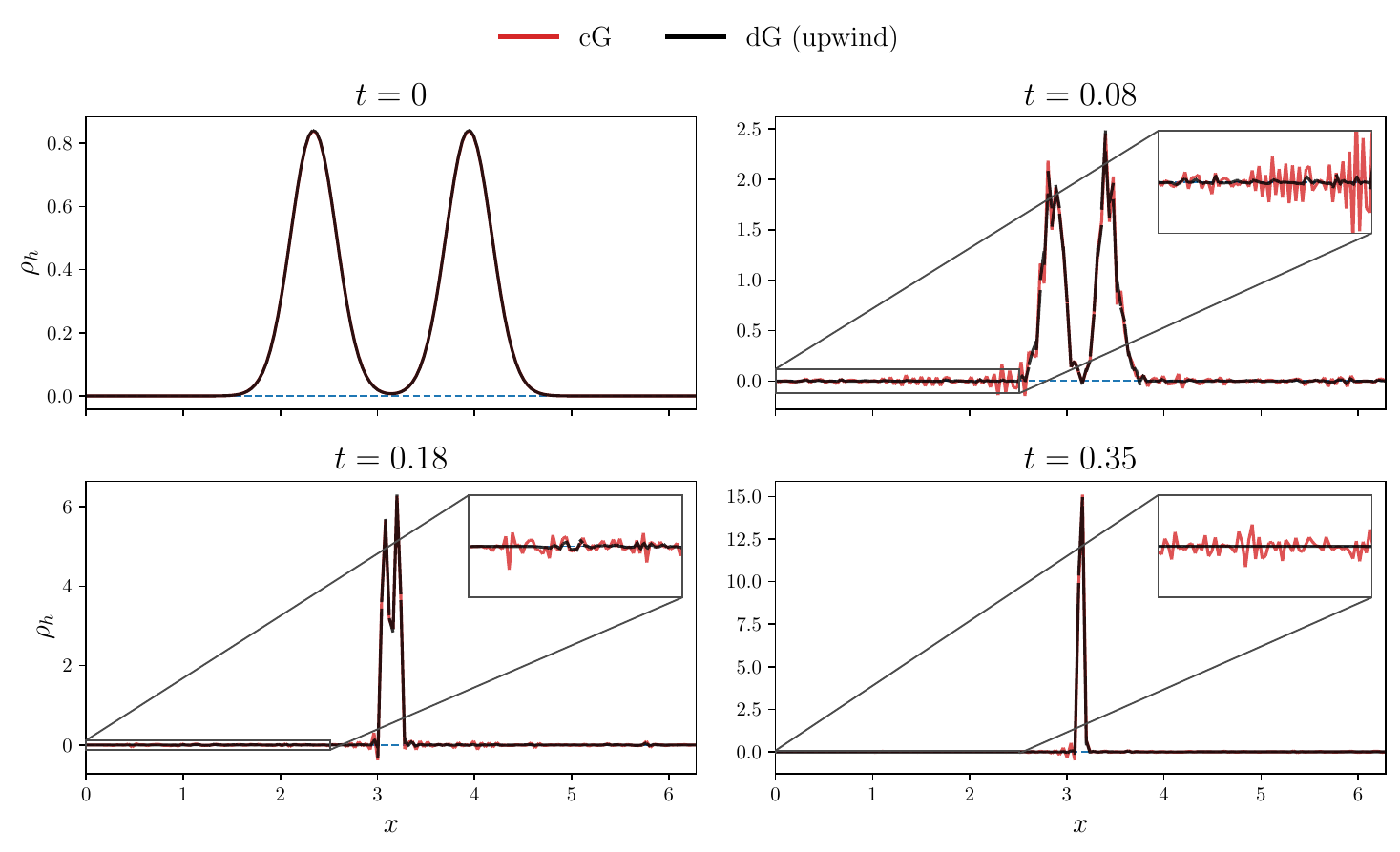}
    \caption{Solution to a Dean--Kawasaki SPDE associated to a system of weakly interacting particles using continuous Galerkin and upwind discontinuous Galerkin.}
    \label{fig: interaction-pot}
\end{figure}

For the experiment, we choose $p=1$, $\sigma = 0.01$, $\lambda = 16.0$, $N=500$ and run the simulation for 1,750 timesteps on a grid with $155$ cells. The results are shown for four different times in Figure~\ref{fig: interaction-pot} and are consistent with our hypotheses. That is, both schemes demonstrate the expected macroscopic clustering behaviour, but the cG method experiences spurious oscillations. 

\subsection{Reflecting Brownian motions}

In many applications, particles are confined to a bounded domain and reflected whenever they reach the boundary. Let $\mathcal{D} \subset \mathbb{R}^d$ be a bounded subset with $C^2$ boundary $\partial \mathcal{D}$, and let $\bn(\bx)$ denote its outward unit normal at $\bx \in \partial \mathcal{D}$. Define the set $A_\ve \coloneqq \{\bx \in \mathcal{D} : d(\partial \mathcal{D}, \bx) < \ve\}$. Then, given a system of $N$ particles we can define the boundary local time of the $j$th particle by
$$
L_j(t) \coloneqq \lim_{\ve \downarrow 0} \frac{\sigma}{\ve}\int_0^t \mathbbm{1}_{A_\ve}\brack{\bX_j(s)} \dd s.
$$
The dynamics for a system of $N$ independent reflecting Brownian motions (RBMs) moving under the influence of a smooth external potential $V$ can then be described by the following Skorokhod equation
\begin{align*}
\dd \bX_j(t)& = -\nabla V(\bX_j(t)) \dd t + \sqrt{2\sigma}\dd \bB_j(t) - \bn(\bX_j(t)) \dd L_j(t), \\
 \dd L_j(t)& = \sigma \delta_{\partial \mathcal{D}}(\bX_j(t)) \dd t,
\end{align*}
where $\delta_{\partial \mathcal{D}}$ is a Dirac measure on $\partial \mathcal{D}$. Intuitively, when a particle $\bX_j$ hits $\partial \mathcal{D}$, it is kicked back in the direction of the inward normal.

In~\cite{bressloff2024generalized}, the author derives a Dean--Kawasaki equation for RBMs on the half-line $[0, \infty)$ with the reflection occurring at the origin. We can generalise this derivation to our multi-dimensional setting to get the following (formal) Dean--Kawasaki equation
\begin{equation}\label{eq: dk-RBM}
    \begin{cases}
        \partial_t \rho = \sigma \Delta \rho + \nabla \cdot (\rho \nabla V) + N^{-1/2} \nabla \cdot (\sqrt{2\sigma \rho} \bxi), \quad &\text{in } \mathcal{D}, \\
        (\sigma \nabla \rho + \rho \nabla V + N^{-1/2}\sqrt{2 \sigma \rho} \bxi) \cdot \bn = 0, \quad &\text{on } \partial\mathcal{D}.
    \end{cases}
\end{equation}
That is, we obtain the standard Dean--Kawasaki equation inside the domain and, at the boundary, we have a zero flux boundary condition. 

\begin{figure}
    \centering
    \includegraphics[width=0.85\linewidth]{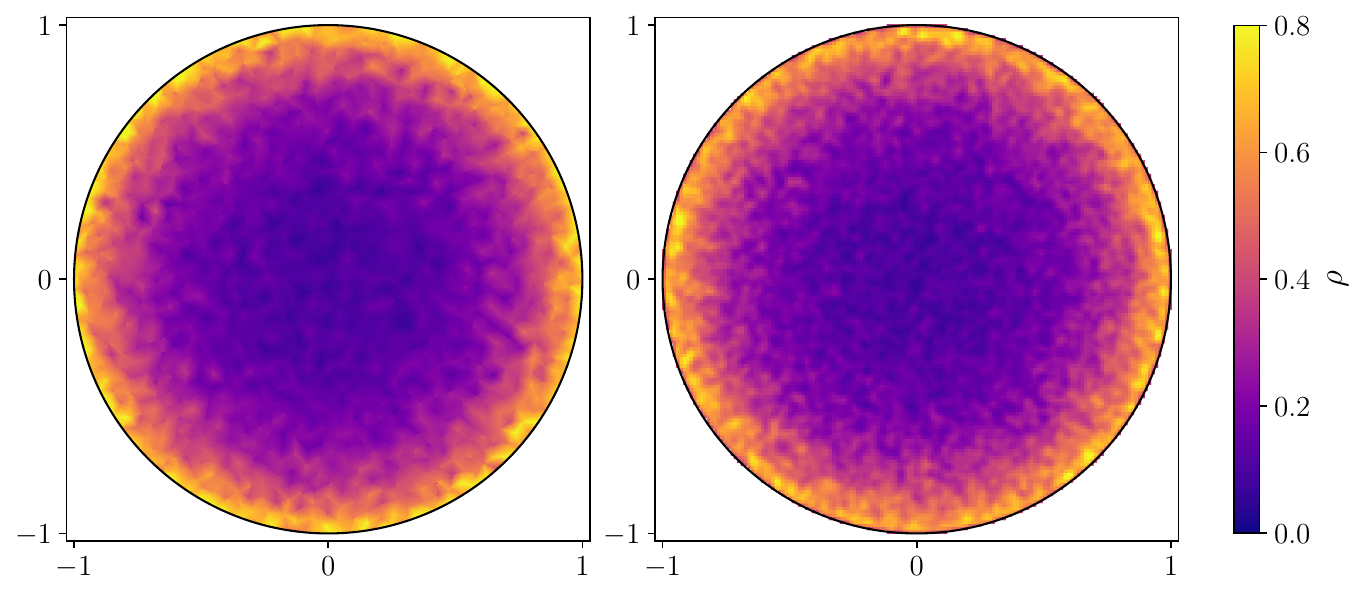}
    \caption{Dean--Kawasaki for RBMs on a disk with an outward radial potential. Left: the upwind dG numerical solution; Right: density of particle positions sampled from the equilibrium density.}
    \label{fig: RBM-comparison}
\end{figure}

\begin{remark}[Interpretation as a martingale problem]
    The derived Dean--Kawasaki equation \eqref{eq: dk-RBM} is purely formal. To make it rigorous, one should interpret the equation via its martingale formulation. In~\cite[Theorem 4.7]{schiavo2024massive}, the author proves that the only solution is the empirical measure of the particle system; i.e., we have a triviality result analogous to that of~\cite{konarovskyi2019dean} for the standard Dean--Kawasaki equation \eqref{eq: Dean--Kawasaki}.
\end{remark}

\begin{remark}[Implementation]
Zero-flux Neumann boundary conditions are simple to impose in both the dG and cG settings. In the dG formulation, the boundary term $\sum_{K \in \Th} \int_{\partial K} \cdot \dd S$ can be split into sums over internal and boundary facets. The internal facets are treated in the usual fashion using numerical fluxes whilst the boundary facet contribution is exactly the normal flux of the current through the boundary. The zero-flux condition enforces that this is zero hence these boundary-facet terms vanish. The cG setting is similar.
\end{remark}

For the numerical experiment, we choose $d=2$ and choose $V(\bx) = -\frac{1}{2}\alpha \absv{\bx}^2$ so that $-\nabla V(\bx) = \alpha \bx$. To avoid direct simulation of the particle system, we instead draw particle samples from the invariant measure of the system. This is the following Gibbs measure
$$
\rho_\infty(\bx) = \frac{1}{\mathcal{Z}}\exp \brack{-\frac{V(\bx)}{\sigma}}, \quad \mathcal{Z} \coloneqq \int_{\mathcal{D}} \exp \brack{-\frac{V(\bx)}{\sigma}} \dd \bx.
$$
For our specific choice of $V(\bx)$ on a disk of radius $R>0$, this evaluates to
$$
\rho_\infty(r) = \frac{\alpha}{2 \pi \sigma \brack{\exp \brack{\frac{\alpha R^2}{2 \sigma}} - 1}}\exp \brack{\frac{\alpha r^2}{2 \sigma}}, \quad 0 \leq r \leq R,
$$
where $r = \absv{\bx}$.

In Figure~\ref{fig: RBM-comparison}, we plot a single realisation of the Dean--Kawasaki solution using $N=100,000$ particles that we ran for $5,000$ timesteps and compare this to the density of $N$ particles sampled from $\rho_\infty$. We see that the Dean--Kawasaki equation is able to replicate the behaviour of the particle system. In particular, in both plots, we see the formation of an annular boundary layer with similar density values caused by the competition between the external potential driving particles outwards and the boundary reflecting them inwards.

The ability to simulate Dean--Kawasaki equations with zero-flux Neumann boundary conditions is important for studying particle systems with reflection at the boundary. In particular, the numerical solution of reflecting SDEs is non-trivial, requiring specialised numerical schemes that carefully treat the solution near and at the boundary to ensure accurate results. Such operations become even more difficult on irregular domains. In contrast, adding the boundary condition to the Dean--Kawasaki equation is a trivial modification yielding a computationally efficient method for studying reflecting particle systems on irregular domains. This significantly broadens the applicability of Dean--Kawasaki equations which previously, to the best of our knowledge, have predominantly been simulated on the torus or rectangular domains. 

\section{Extensions and future work}\label{sec: future}

\textbf{Analysis.} The analysis in this paper was presented for independent Brownian motions, and we presented numerical experiments beyond the scope of our theory that looked at external and interaction potentials, and zero-flux Neumann boundary conditions. One possible extension is to extend the analysis to these settings. The case of interaction potentials has been considered in~\cite{cornalba2026density} for finite-difference discretisations. Another direction to investigate is the analysis for full space-time discretisations. 

\textbf{Positivity preservation.} Solutions to the Dean--Kawasaki equation are empirical measures of a particle system so are non-negative by construction. Discretisations should thus preserve positivity, and this remains a core open question for our method. One avenue is the local post-processing of Remark~\ref{remark_limiters}, and two questions there seem worth pursuing: whether the probability that the post-processing applies on \emph{every} element can be improved beyond $1 - \mathcal{E}(N, h)$, perhaps by limiting only on a subset of elements; and whether limiting can be applied at every timestep of a fully discrete scheme, which would also alleviate the mass loss caused by the positive part $\rho_h^+$ in the noise. Both require an analysis at the fully discrete level, and we leave them for future research.

\textbf{Fluctuation dissipation.} Besides the cross-variation of the noise and positivity, another key structural feature of the Dean--Kawasaki equation (equivalently, of the underlying particle system) that we wish to preserve under numerical discretisation is the fluctuation-dissipation relationship. Roughly speaking, this enforces that the fluctuations injected by the stochastic term are balanced by the dissipation generated by the diffusion operator. A first step towards this goal is to preserve a linearised version of the fluctuation-dissipation relation derived by linearising the Dean--Kawasaki equation around the steady state of the particle system (see, for example, the discussion in~\cite[Section 3.1.2]{bell2026surface}). For a SIP discontinuous Galerkin discretisation, the penalty parameter $\eta$ breaks this structure since it introduces additional dissipation with no corresponding stochasticity. However, following~\cite[Remark 38.19]{ern2021finite2}, if the discrete gradient is constructed in the $(p+1)$-th order polynomial space, then on simplicial meshes the dG energy norm of any $v_h \in V_p$ is controlled by $\|{\dgrad^{p+1} v_h}\|$; consequently, replacing $a_h$ by the penalty-free bilinear form
$$
\tilde{a}_h(v_h, w_h) \coloneqq \frac{1}{2}\brack{\dgrad^{p+1} v_h, \dgrad^{p+1} w_h}
$$
yields a stable and optimally convergent discretisation. Such a form is built from precisely the object that carries the cross-variation of our noise \eqref{eq: crossvar-target}, so linearised fluctuation-dissipation balance would become structural rather than imposed. The price is a higher-dimensional gradient space, and therefore increased computational cost, together with a loss of exact consistency since the liftings are discrete objects.

\textbf{Multilevel Monte Carlo.} A common method of variance reduction is the Multilevel Monte Carlo method~\cite{giles2015multilevel}. This was studied for finite difference discretisations of the Dean--Kawasaki equation in~\cite{cornalba2025multilevel} and allows one to approximate observables of the underlying particle system at reduced computational cost. The method hinges on an appropriate coupling of the noise between coarse and fine grid/mesh resolutions. It would be interesting to investigate the design of couplings that are compatible with our high order noise discretisation.

\medskip

\textbf{AI statement.} ChatGPT and Claude were used for proofreading, assistance with the numerical experiments, and exploration of proof strategies. All content was written by the authors who take full responsibility for the contents of this paper.

\medskip

\textbf{Acknowledgments.} KA is supported by a scholarship from the EPSRC Centre for Doctoral Training in Statistical Applied Mathematics at Bath (SAMBa), under the project EP/S022945/1. This research made use of the Nimbus High Performance Computing Service at the University of Bath. We would like to thank Tristan Pryer for numerous insightful discussions and useful advice about discontinuous Galerkin methods and for his suggestion to use the discrete gradient. 

\bibliographystyle{amsplaindoi}
\bibliography{ref.bib}

\appendix

\section{Proof of weak existence and moment bounds}\label{app: well-posed}

We give a short proof of Theorem~\ref{thm: weak-existence}.

\begin{proof}[Proof of Theorem~\ref{thm: weak-existence}]
    Our variational problem is a finite-dimensional SDE. In particular, let $\{e_i\}_{i=1}^{n_V}$ be a basis for $V_p$, where $n_V \coloneqq \mathrm{dim}(V_p)$. By linearity, we can replace $v_h$ with any of the basis functions, and we can expand $\rho_h(t) = \sum_{i=1}^{n_V} X_h^i(t) e_i$ for some random, time-dependent coefficients $X_h^i(t)$. Substituting into the variational form gives
    $$
    \dd X_h(t) = -M^{-1}A X_h(t) \dd t + N^{-1/2}M^{-1}\Sigma(X_h(t)) \dd B(t),
    $$
    where $X_h(t) = (X_h^1(t), \ldots, X_h^{n_V}(t))$, $M_{k\ell} = (e_\ell, e_k)$ is the mass matrix on $V_p$, $A_{k\ell} = a_h(e_\ell, e_k)$ is the stiffness matrix associated to the SIP bilinear form. For a coefficient vector $x \in \mathbb{R}^{n_V}$, write $X \coloneqq \sum_{i=1}^{n_V}x_i e_i \in V_p$ for the associated function and let $\Sigma(x) \in \mathbb{R}^{n_V \times n_G}$ be the matrix with entries $\Sigma(x)_{k \ell} = (\boldsymbol{g}_\ell(X), \dgrad e_k)$. Then
    $$
    \norm{(-M^{-1}A) x} \leq \norm{M^{-1}A}_{\mathrm{op}}\norm{x} \lesssim_h \norm{x},
    $$
    since the basis is fixed. We also have
    $$
    \norm{\Sigma(x)}_F^2 = \sum_{k=1}^{n_V}\sum_{\ell=1}^{n_G}(\boldsymbol{g}_\ell(X), \dgrad e_k)^2 = \sum_{k=1}^{n_V}(X^+, \absv{\dgrad e_k}^2),
    $$
    using Proposition~\ref{prop: construct-diffusion-coefficients}. Again, since the basis is fixed, it follows that
    $$
    \norm{\Sigma(x)}_F^2 \lesssim_h \norm{X}_{L^1} \lesssim_h \norm{x}.
    $$
    It follows that
    $$
    \norm{(-M^{-1}A) x}^2 + N^{-1}\norm{M^{-1}\Sigma(x)}_F^2 \lesssim_{N, h} 1 + \norm{x}^2.
    $$
    That is, we have a linear growth condition. Coupled with the continuity of the drift and diffusion, we can invoke standard results for finite-dimensional SDEs~\cite[Theorems 2.3 and 2.4]{ikeda2014stochastic} to assert the existence of a non-explosive, probabilistically weak solution for $t \in [0, T]$.

    Proving the moment bound is standard. Fix $q \geq 2$ and write the SDE in the form $\dd X_h(t) = b(X_h(t)) \dd t + \sigma(X_h(t)) \dd B(t)$ where the previous arguments imply that we have the growth condition
    $$
    \absv{b(x)}^q + \norm{\sigma(x)}_F^q \leq C_{N, h, q}(1+\absv{x}^q).
    $$
    We then define the stopping time $ \tau_R \coloneqq \inf \{t \geq 0~:~\absv{X_h(t)} \geq R\}$ for any $R \geq 0$. We can write the stopped SDE as
    $$
    X_h(t \wedge \tau_R) = X_h(0) + \int_0^t \mathbbm{1}_{\{s \leq \tau_R\}}b(X_h(s)) \dd s + \int_0^t \mathbbm{1}_{\{s \leq \tau_R\}} \sigma(X_h(s)) \dd B(s).
    $$
    From this formulation, it follows that
    \begin{multline*}
    \E{\sup_{r \in [0, t]}\norm{X_h(r \wedge \tau_R)}^q} \lesssim_{q} \E{\norm{X_h(0)}^q} + \E{\sup_{r \in [0, t]}\norm{\int_0^r \mathbbm{1}_{\{s \leq \tau_R\}}b(X_h(s))\dd s}^q} \\ 
    + \E{\sup_{r \in [0, t]}\norm{\int_0^r \mathbbm{1}_{\{s \leq \tau_R\}} \sigma(X_h(s)) \dd B(s)}^q}.
    \end{multline*}
    Using the growth condition, H\"older's inequality and the BDG inequality, it follows that
    $$
    \E{\sup_{r \in [0, t]}\norm{X_h(r \wedge \tau_R)}^q} \lesssim_{q, T, N, h} \brack{1+\E{\norm{X_h(0)}^q}} + \int_0^t \E{\sup_{u \in [0, s]}\norm{X_h(u \wedge \tau_R)}^q} \dd s.
    $$
    Applying Gr\"onwall's inequality gives
    $$
    \E{\sup_{t \in [0, T]}\norm{X_h(t \wedge \tau_R)}^q} \lesssim_{q, T, N, h} \brack{1 + \E{\norm{X_h(0)}^q}},
    $$
    where the estimate is independent of $R$. Since $X_h$ is non-explosive, $\tau_R \rightarrow T$ as $R \rightarrow \infty$. The result follows upon taking limits and using Fatou's lemma.
\end{proof}

\section{Recursive relations for moments}\label{sec: recursive}

We compute It\^o differentials for the quantities in \eqref{eq: SN-TN} and prove some related estimates that we require for the convergence analysis.

\begin{lemma}\label{lem: big-moments-lemma}
    Fix $M \in \mathbb{N}$, a multi-index $\bj = (j_1, \ldots, j_M)$ and a set of test functions $\boldsymbol{\varphi} = (\varphi_1, \ldots, \varphi_M) \in [C^2]^M$. For any $(k, \ell) \in \{1, \ldots, M\}^2$, denote by $\bj^{k \ell}$ the vector $\bj$ with both $j_k$ and $j_\ell$ decreased by one unit (if $k=\ell$, then $j_k$ is understood to be reduced by two units). For any $k \in \{1, \ldots, M\}$, denote by $\bj^k$ the vector $\bj$ with $j_k$ decreased by one unit. Then, for $\mathcal{S}_N^{\bj}$ and $\mathcal{T}_N^{\bj}$ as defined in \eqref{eq: SN-TN}, we have the following It\^o differentials:
    \begin{multline}\label{eq: ito-moments-dkdg}
        \dd \mathcal{S}_N^{\bj}(\mathcal{I}_h \boldsymbol{\varphi}, T, t) = -N^{-1/2}\sum_{m=1}^M j_m \mathcal{S}_N^{\bj^m}(\mathcal{I}_h \boldsymbol{\varphi}, T, t) \dd \mathcal{W}(\rho_h^+(t), \phi_{m, h}^t) \\
        + N^{-1}\sum_{k, \ell=1}^M \frac{(j_k - \delta_{k \ell})j_\ell}{2}\mathcal{S}_N^{\bj^{k\ell}}(\mathcal{I}_h \boldsymbol{\varphi}, T, t)(\rho_h^+(t), \dgrad \phi_{k, h}^t \cdot \dgrad \phi_{\ell, h}^t) \dd t,
    \end{multline}
    and
    \begin{multline}\label{eq: ito-moments-particles}
        \dd \mathcal{T}_N^{\bj}(\boldsymbol{\varphi}, T, t) = -\sum_{m=1}^M j_m \mathcal{T}_N^{\bj^m}(\boldsymbol{\varphi}, T, t)\brack{N^{-1}\sum_{r=1}^N \nabla \phi_m^t(\bB_r(t)) \cdot \dd \bB_r(t)} \\
        + N^{-1}\sum_{k, \ell=1}^M \frac{(j_k - \delta_{k \ell})j_\ell}{2}\mathcal{T}_N^{\bj^{k\ell}}(\boldsymbol{\varphi}, T, t)\brack{N^{-1}\sum_{r=1}^N \nabla\phi_k^t(\bB_r(t)) \cdot \nabla \phi_\ell^t(\bB_r(t))} \dd t.
    \end{multline}
\end{lemma}

\begin{proof}
    We first prove \eqref{eq: ito-moments-dkdg}. Consider the case $M=\absv{\bj} = 1$. We compute the It\^o differential of $(\rho_h(t), \phi_h^t) - (\rho_h(0), \phi_h^0)$, which gives
    $$
    \dd \brack{(\rho_h(t), \phi_h^t) - (\rho_h(0), \phi_h^0)} = (\dd \rho_h(t), \phi_h^t) + (\rho_h(t), \partial_t \phi_h^t) \dd t.
    $$
    For the first term, since the test function is deterministic, we can pull the It\^o differential out of the $L^2$ inner product. This is then the numerical scheme \eqref{eq: dg-dk} with $v_h = \phi_h^t$. Similarly, the second term is the numerical scheme \eqref{eq: SIP-adjoint-heat} with $v_h = \rho_h(t)$. Altogether, we get
    \begin{align*}  
      \dd \brack{(\rho_h(t), \phi_h^t) - (\rho_h(0), \phi_h^0)} &= -a_h(\rho_h(t), \phi_h^t) \dd t + N^{-1/2}\dd \mathcal{W}(\rho_h^+(t), \phi_h^t)\\&\qquad + a_h(\phi_h^t, \rho_h(t)) \dd t.
    \end{align*}
    Using the symmetry of $a_h(\cdot, \cdot)$, we deduce that
    $$
    \dd \brack{(\rho_h(t), \phi_h^t) - (\rho_h(0), \phi_h^0)} = N^{-1/2}\dd \mathcal{W}(\rho_h^+(t), \phi_h^t).
    $$
    For the remainder of the proof, we only sketch the key idea referring the reader to~\cite[Lemma 15]{cornalba2023dean} for the complete details. Specifically, for $M=1$ and $\absv{\bj} > 1$, one applies the It\^o formula to the composition of the function $z \mapsto z^j$ with $(\rho_h(t), \phi_h^t)$. An induction argument then extends the result to arbitrary indices $M$. Analogous arguments give \eqref{eq: ito-moments-particles}.
\end{proof}

\begin{corollary}\label{cor: sup-S-one-j}
    For any $j \geq 2$, we have
    \begin{equation}
        \max_{t \in [0, T]}\E{\absv{\mathcal{S}_N(\mathcal{I}_h \varphi, T, t)}^j} \leq \brack{2N^{-1}T C(d, \delta, \rhomin, \rhomax)\max_{t \in [0, T]}\norm{\dgrad \phi_h^t}^2}^{j/2}j^{3j}.
    \end{equation}
\end{corollary}

\begin{proof}
    The proof is a direct modification of~\cite[Lemma 16]{cornalba2023dean}. We sketch the details to illustrate our different application of H\"older's inequality to account for the lack of $L^\infty(L^\infty)$ estimates for solutions of \eqref{eq: SIP-adjoint-heat}. First apply the It\^o formula to the composition of the map $z \mapsto \absv{z}^j$ and $(\rho_h(t), \phi_h^t)$. Upon taking the expected value one obtains
    $$
    \dd \E{\absv{\mathcal{S}_N(\mathcal{I}_h \varphi, T, t)}^j} = N^{-1} \frac{j(j-1)}{2}\E{\absv{\mathcal{S}_N(\mathcal{I}_h \varphi, T, t)}^{j-2}(\rho_h^+(t), \dgrad \phi_h^t \cdot \dgrad \phi_h^t)} \dd t.
    $$
    Using H\"older's inequality in expectation gives
    \begin{multline*}
    \dd \E{\absv{\mathcal{S}_N(\mathcal{I}_h \varphi, T, t)}^j} \leq N^{-1} \frac{j(j-1)}{2} \E{\absv{\mathcal{S}_N(\mathcal{I}_h \varphi, T, t)}^{j-1}}^{\frac{j-2}{j-1}} \\ 
    \times \E{\absv{(\rho_h^+(t), \dgrad \phi_h^t \cdot \dgrad \phi_h^t)}^{j-1}}^{\frac{1}{j-1}}.
    \end{multline*}
    For the final term, we use $L^\infty$-$L^1$ H\"older in space and apply the $L^\infty(L^\infty)$ estimate in Proposition~\ref{prop: Linfty-Linfty} with exponent $j-1$, to get 
    $$
    \dd \E{\absv{\mathcal{S}_N(\mathcal{I}_h \varphi, T, t)}^j} \leq \frac{1}{N}j^6 \norm{\dgrad \phi_h^t}^2 C(d, \delta, \rhomin, \rhomax) \E{\absv{\mathcal{S}_N(\mathcal{I}_h \varphi, T, t)}^{j-1}}^{\frac{j-2}{j-1}} \dd t.
    $$
    Taking the supremum in time and then performing induction over $j$ gives the result.
\end{proof}

\begin{corollary}\label{cor: sup-moments}
    Given $\bj = (j_1, \ldots, j_M)$ with $\absv{\bj}=j$ and $\boldsymbol{\varphi} = (\varphi_1, \ldots, \varphi_M) \in [H^{p+2}]^M$, we have
    \begin{align}
        \max_{t \in [0, T]}\E{\absv{\mathcal{T}_N^{\bj}(\boldsymbol{\varphi}, T, t)}} &\leq \brack{N^{-1}T}^{j/2}j^j \brack{\prod_{m=1}^M \norm{\nabla \varphi_m}_{L^\infty}^{j_m}}, \label{eq: sup-moments-T}\\
        \max_{t \in [0, T]}\E{\absv{\mathcal{S}_N^{\bj}(\mathcal{I}_h\boldsymbol{\varphi}, T, t)}} &\leq \brack{2N^{-1}T C(d, \delta, \rhomin, \rhomax)}^{j/2}j^{3j}\brack{\prod_{m=1}^M \norm{\varphi_m}_{H^{p+2}}^{j_m}}. \label{eq: sup-moments-S}
    \end{align}
\end{corollary}

\begin{proof}
    The proof of \eqref{eq: sup-moments-S} follows from a multi-factor H\"older inequality, Corollary~\ref{cor: sup-S-one-j} and Corollary~\ref{cor: discrete-gradient estimates}. The proof of \eqref{eq: sup-moments-T} can be found in~\cite[Corollary 17]{cornalba2023dean}.
\end{proof}

\section{Heat estimates}\label{sec: heat-estimates}

We prove various error estimates for the adjoint heat flow \eqref{eq: adjoint-heat} that are crucial for our convergence analysis. Our first observation is that it is sufficient to study the regular heat equation by performing the time reversal $T-t \leftrightarrow t$. That is, we denote by $u^t$ the solution to
\begin{equation}\label{eq: regular-heat}
\partial_t u^t = \frac{1}{2}\Delta u^t, \quad \text{on}~ \btd \times (0, T],
\end{equation}
with initial condition $u_0 \coloneqq \varphi$. All the estimates we prove for $u^t$ then immediately hold for $\phi^t$. Similarly, in the semi-discrete setting, we seek $u_h^t(\cdot) \in V_p$ for $t \in (0, T]$ such that
\begin{equation}\label{eq: SIP-regular-heat}
(\partial_t u_h^t, v_h) + a_h(u_h^t, v_h) = 0,
\end{equation}
holds for all $v_h \in V_p$ where $u_h^0 = \mathcal{I}_h \varphi$ and $a_h$ is the SIP bilinear form.

\subsection{Standard estimates}

We recall standard maximum principle and energy estimates for the heat equation used frequently in the subsequent analysis.

\begin{lemma}\label{lem: max-principle}
    Let $u^t$ be the solution to \eqref{eq: regular-heat} with initial datum $\varphi \in H^{p+2}$. Then, there exists a constant $C > 0$ such that
    \begin{equation}
        \sup_{t \in [0, T]}\norm{u^t}_{L^\infty} + \sup_{t \in [0, T]}\norm{\nabla u^t}_{L^\infty} \leq C \norm{\varphi}_{H^{p+2}},
    \end{equation}
    and
    \begin{equation}\label{heat-estimate}
        \norm{u^t}_{L^\infty(0, T; H^{p+1})} + \norm{\partial_tu^t}_{L^2(0, T; H^{p+1})} \leq C \norm{\varphi}_{H^{p+2}}.
    \end{equation}
\end{lemma}

\subsection{Elliptic estimates}

\begin{definition}\label{def: ritz-projection}
    Let $a_h(\cdot, \cdot)$ be the SIP bilinear form (cf. Definition~\ref{def: SIP-bilinear-form}). We define the Ritz projection $\ritz u \in V_p$ of a function $u \in H^{1+r}$ for $r>1/2$ via
    $$
    a_h(\ritz u, v_h) = a_h(u, v_h) \quad \forall v_h \in V_p,
    $$
    subject to the constraint
    $$
    \int_{\btd} (\ritz u - u) \dd \bx = 0.
    $$
    The constraint ensures that the Ritz projection is unique since $a_h(\cdot, \cdot)$ does not distinguish constants.
\end{definition}

We place the following semi-norm on $V_p$:
$$
\norm{v_h}^2_{V_p} \coloneqq \norm{\nabla_h v_h}_{L^2(\btd; \mathbb{R}^d)}^2 + \absv{v_h}_J^2, \quad \absv{v_h}_J^2 \coloneqq \sum_{F \in \Fh}h_F^{-1} \norm{\jump{v_h}}_{L^2(F; \mathbb{R}^d)}^2.
$$
Since dG methods are non-conforming, we also define the space $V_\sharp \coloneqq V_p + H^{1+r}$ and equip it with the following semi-norm
$$
\norm{w}_{V_{\sharp}}^2 \coloneqq \norm{\nabla_h w}_{L^2(\btd)}^2 + \absv{w}_J^2 + \sum_{K \in \Th}h_K \norm{\bn_K \cdot \nabla w |_K}_{L^2(\partial K)}^2.
$$
Note that $\norm{\cdot}_{V_p}$ and $\norm{\cdot}_{V_\sharp}$ are only semi-norms since they vanish on constants. We also have the estimate
\begin{equation}\label{eq: Vsharp-Vp-equiv}
    \norm{v_h}_{V_\sharp} \leq C\norm{v_h}_{V_p}, \quad \forall v_h \in V_p,
\end{equation}
which follows from a discrete trace inequality (see~\cite[Chapter 38.3]{ern2021finite2}).

\begin{lemma}[Boundedness and coercivity]
    The SIP bilinear form $a_h(\cdot, \cdot)$ is bounded on $V_p \times V_p$. Moreover, the boundedness still holds when we view $a_h(\cdot, \cdot)$ as a map on $V_\sharp \times V_p$. Additionally, for penalty parameter $\eta > 0$ sufficiently large, $a_h(\cdot, \cdot)$ is coercive on $V_p \times V_p$. In mathematical terms, we have
    \begin{align*}
        \absv{a_h(v_h, w_h)} &\leq C\norm{v_h}_{V_p}\norm{w_h}_{V_p}, \\
        \absv{a_h(v, w_h)} &\leq C\norm{v}_{V_\sharp}\norm{w_h}_{V_p}, \\
        a_h(v_h, v_h) &\geq C \norm{v_h}^2_{V_p},
    \end{align*}
    for constants $C > 0$.
\end{lemma}

\begin{theorem}\label{thm: SIP-elliptic}
    There exists $C>0$, independent of $h$ such that for all $w \in H^{p+1}$
    \begin{align}
        \norm{w - \ritz w}_{V_\sharp} \leq Ch^p \norm{w}_{H^{p+1}}, \\
        \norm{w - \ritz w}_{L^2} \leq Ch^{p+1}\norm{w}_{H^{p+1}}.
    \end{align}
\end{theorem}

\begin{proof}
    This is an adaptation of standard results~\cite[Theorem 38.10]{ern2021finite2} and~\cite[Theorem 38.12]{ern2021finite2} to the periodic setting.
\end{proof}

\subsection{Parabolic estimates}

Equipped with estimates in the elliptic setting, we can now prove the estimates we require for SIP discretisations of the heat equation. We have the following error estimate in the dG norm:

\begin{theorem}\label{thm: SIP-parabolic-energy}
    Let $u^t$ solve the continuous heat equation \eqref{eq: regular-heat} and $u_h^t$ solve the associated SIP semi-discretisation \eqref{eq: SIP-regular-heat}. Assume that $\varphi \in H^{p+2}$ and that $\eta > 0$ is sufficiently large so that $a_h(\cdot, \cdot)$ is coercive. Then there exists a constant $C>0$, independent of $h$, such that 
    $$
    \norm{u_h - u}_{L^\infty(0, T; V_\sharp)} \leq C h^p \norm{\varphi}_{H^{p+2}}.
    $$
\end{theorem}

\begin{proof}
    Throughout the proof, we use $\lesssim$ to denote inequality up to an arbitrary constant $C>0$ that depends on the shape regularity of the mesh, dimension, polynomial degree and $\eta>0$ but is independent of $h$.
    
    \textbf{Step 1:}
    Since the SIP bilinear form is consistent, we have
    $$
    (\partial_t(u_h^t - u^t), v_h) + a_h(u_h^t - u^t, v_h) = 0.
    $$
    Adding and subtracting the Ritz projection and using its definition gives
    \begin{equation}\label{eq: parabolic-error-split}
    (\partial_t(u_h^t - \ritz u^t), v_h) + a_h(u_h^t - \ritz u^t, v_h) = (\partial_t(u^t - \ritz u^t), v_h).
    \end{equation}
    We then choose $v_h = \partial_t(u_h^t - \ritz u^t)$ which gives
    $$
    \norm{\partial_t(u_h^t - \ritz u^t)}^2 + a_h(u_h^t - \ritz u^t, \partial_t(u_h^t - \ritz u^t)) = (\partial_t(u^t - \ritz u^t), \partial_t(u_h^t - \ritz u^t)).
    $$
    Using the symmetry of $a_h(\cdot, \cdot)$, and the Cauchy--Schwarz and Young inequalities gives
    \begin{multline*}
    \norm{\partial_t(u_h^t - \ritz u^t)}^2 + \frac{1}{2}\frac{\dd}{\dd t}a_h(u_h^t - \ritz u^t, u_h^t - \ritz u^t) \\
    \leq \frac{1}{2}\norm{\partial_t(u^t - \ritz u^t)}^2 + \frac{1}{2}\norm{\partial_t(u_h^t - \ritz u^t)}^2.
    \end{multline*}
    It then follows that
    $$
    \frac{\dd}{\dd t}a_h(u_h^t - \ritz u^t, u_h^t - \ritz u^t) \leq \norm{\partial_t(u^t - \ritz u^t)}^2.
    $$
    Integrating in time up to time $t \leq T$ gives
    \begin{equation}\label{eq: sip-ritz-integrated}
    a_h(u_h^t - \ritz u^t, u_h^t - \ritz u^t) \leq a_h(u_h^0 - \ritz u^0, u_h^0 - \ritz u^0) + \int_0^t \norm{\partial_t(u^r - \ritz u^r)}^2 \dd r.
    \end{equation}
    For the first term, viewing $a_h(\cdot, \cdot)$ as a map on $V_\sharp \times V_p$ and recalling that $u_h^0 = \mathcal{I}_h u^0$, we can add and subtract $u^0$ to get
    $$
    a_h(u_h^0 - \ritz u^0, u_h^0 - \ritz u^0) = a_h(\mathcal{I}_h u^0 - u^0, u_h^0 - \ritz u^0) + \underbrace{a_h(u^0 - \ritz u^0, u_h^0 - \ritz u^0)}_{=0},
    $$
    where the latter term vanishes by definition of the Ritz projection. We now use coercivity of $a_h(\cdot, \cdot)$ on $V_p \times V_p$ and boundedness on $V_\sharp \times V_p$ to deduce that
    $$
    \norm{u_h^0 - \ritz u^0}_{V_p}^2 \lesssim a_h(\mathcal{I}_h u^0 - u^0, u_h^0 - \ritz u^0) \lesssim \norm{\mathcal{I}_h u^0 - u^0}_{V_\sharp} \norm{u_h^0 - \ritz u^0}_{V_p},
    $$
    from which it follows that
    $$
    \norm{u_h^0 - \ritz u^0}_{V_p} \lesssim \norm{\mathcal{I}_h u^0 - u^0}_{V_\sharp}.
    $$
    We therefore obtain the estimate
    \begin{equation}\label{eq: initial-sip-ritz}
    a_h(u_h^0 - \ritz u^0, u_h^0 - \ritz u^0) \lesssim \norm{\mathcal{I}_h u^0 - u^0}_{V_\sharp}^2.
    \end{equation}
    Equipped with the estimate \eqref{eq: initial-sip-ritz}, we then use coercivity in \eqref{eq: sip-ritz-integrated} to get 
    $$
    \norm{u_h^t - \ritz u^t}_{V_p}^2 \lesssim \norm{\mathcal{I}_h u^0 - u^0}_{V_\sharp}^2 + \int_0^t \norm{\partial_t(u^r - \ritz u^r)}^2 \dd r.
    $$
    Extending the integral up to the final time, and using that $\sqrt{a^2+b^2} \leq a+b$ for $a, b \geq 0$ gives
    \begin{equation}\label{eq: parabolic-step1-final}
    \norm{u_h^t - \ritz u^t}_{V_p} \lesssim \norm{\mathcal{I}_h u^0 - u^0}_{V_\sharp} + \norm{\partial_t(u - \ritz u)}_{L^2(0, T; L^2(\btd))}.
    \end{equation}

    \textbf{Step 2:} Adding and subtracting the Ritz projection and using \eqref{eq: Vsharp-Vp-equiv} gives
    $$
    \norm{u_h^t - u^t}_{V_\sharp} \lesssim \norm{u_h^t - \ritz u^t}_{V_p} + \norm{\ritz u^t - u^t}_{V_\sharp}.
    $$
    Coupled with \eqref{eq: parabolic-step1-final} we get
    \begin{equation}\label{eq: parabolic-step2-final}
    \norm{u_h^t - u^t}_{V_\sharp} \lesssim \norm{\mathcal{I}_h u^0 - u^0}_{V_\sharp} + \norm{\partial_t(u - \ritz u)}_{L^2(0, T; L^2(\btd))} + \norm{\ritz u^t - u^t}_{V_\sharp}.
    \end{equation}

    \textbf{Step 3:}
    We have the estimate 
    $$
    \norm{\mathcal{I}_h u^0 - u^0}_{V_\sharp} \lesssim h^p\norm{\varphi}_{H^{p+1}}.
    $$
    This follows from recalling the definition of the $V_\sharp$ semi-norm, using a standard interpolation estimate~\cite[Theorem 11.13]{ern2021finite} and a multiplicative trace inequality~\cite[Lemma 12.15]{ern2021finite}.

    \textbf{Step 4:} It remains to bound the latter two terms in \eqref{eq: parabolic-step2-final}. To deal with $\norm{\ritz u^t - u^t}_{V_\sharp}$, we use Theorem~\ref{thm: SIP-elliptic} and standard properties of the heat equation (cf.~\eqref{heat-estimate}) to write
    $$
    \norm{\ritz u^t - u^t}_{V_\sharp} \lesssim h^p\norm{u^t}_{H^{p+1}} \lesssim h^p \norm{u^0}_{H^{p+2}}.
    $$
    For $\norm{\partial_t(u - \ritz u)}_{L^2(0, T; L^2(\btd))}^2$, we write
    \begin{align*}
        \norm{\partial_t(u - \ritz u)}_{L^2(0, T; L^2(\btd))}^2 &= \int_0^T \norm{\partial_t u^t - \ritz (\partial_t u^t)}_{L^2}^2 \dd t, \\
        &\lesssim h^{2p+2} \int_0^T \norm{\partial_t u^t}_{H^{p+1}}^2 \dd t, \\
        &\lesssim h^{2p+2}\norm{u^0}_{H^{p+2}}^2,
    \end{align*}
    where we have used Theorem~\ref{thm: SIP-elliptic} and standard energy estimates for the heat equation. Altogether, we have the estimate
    $$
    \norm{u_h^t - u^t}_{V_\sharp} \lesssim h^p \norm{\varphi}_{H^{p+2}}.
    $$
    Taking the supremum in time gives the result.
\end{proof}

\begin{theorem}\label{thm: SIP-parabolic-L2}
    Under the same assumptions as Theorem~\ref{thm: SIP-parabolic-energy} there exists a constant $C_T>0$, depending on $T$ but independent of $h$, such that
    $$
    \norm{u_h - u}_{L^\infty(0, T; L^2)} \leq C_T h^{p+1} \norm{\varphi}_{H^{p+2}}.
    $$
\end{theorem}

\begin{proof}
    The proof is similar to that of Theorem~\ref{thm: SIP-parabolic-energy}. We use the same splitting \eqref{eq: parabolic-error-split} but choose $v_h = u_h^t - \ritz u^t$, use coercivity and finally Young's inequality to get
    $$
    \frac{1}{2}\frac{\dd}{\dd t} \norm{u_h^t - \ritz u^t}^2 \lesssim \norm{\partial_t(u^t - \ritz u^t)}^2 + \norm{u_h^t - \ritz u^t}^2.
    $$
    Gr\"onwall's inequality gives
    $$
    \norm{u_h^t - \ritz u^t}^2 \lesssim_T \norm{u_h^0 - \ritz u^0}^2 + \int_0^T \norm{\partial_t(u^t - \ritz u^t)}^2 \dd t.
    $$
    It is straightforward to estimate
    $$
    \norm{u_h^0 - \ritz u^0} \lesssim h^{p+1}\norm{\varphi}_{H^{p+2}},
    $$
    by adding and subtracting $u^0$, using a standard interpolation estimate~\cite[Theorem 11.13]{ern2021finite}, and Theorem~\ref{thm: SIP-elliptic}. For the term inside the integral, we use Theorem~\ref{thm: SIP-elliptic} and standard estimates for the heat equation (cf.~\eqref{heat-estimate}) to get
    $$
    \int_0^T\norm{\partial_t(u^t - \ritz u^t)}^2 \dd t \lesssim h^{2p+2}\norm{\varphi}^2_{H^{p+2}}.
    $$
    Altogether, we have that
    $$
    \norm{u_h^t - \ritz u^t} \lesssim_T h^{p+1}\norm{\varphi}_{H^{p+2}}.
    $$
    We then note the decomposition
    $$
    \norm{u_h^t - u^t} \leq \norm{u_h^t - \ritz u^t} + \norm{u^t - \ritz u^t}.
    $$
    We have already estimated the first term, and it is straightforward to estimate the second using Theorem~\ref{thm: SIP-elliptic}. The result follows.
    
\end{proof}

\subsection{Gradient estimates}

We first recall the following lemma providing estimates on the lifting operators defined in \eqref{eq: local-lift} and \eqref{eq: global-lift}.

\begin{lemma}[{\cite[Lemma 38.15]{ern2021finite2}}]\label{lem: ern-lift-lemma}
    For any $F \in \Fh$ let $\mathcal{T}_F \coloneqq \{K \in \Th : F \subset \partial K\}$ denote the pair of elements sharing the facet $F$, and define $D_F \coloneqq \mathrm{int}\brack{\bigcup_{K \in \mathcal{T}_F}K}$. Then for any $\ell \geq 0$, $\boldsymbol{\mathcal{L}}_F^\ell$ is supported on $D_F$ and the following estimates hold:
    \begin{align}
        \norm{\boldsymbol{\mathcal{L}}_F^\ell(\bv)}_{L^2(D_F; \mathbb{R}^d)} &\leq Ch_F^{-1/2}\norm{\bv}_{L^2(F; \mathbb{R}^d)}, \quad \forall \bv \in L^2(F; \mathbb{R}^d),~ \forall F \in \Fh, \\
        \norm{\boldsymbol{\mathcal{L}}_h^\ell(\jump{v})}_{L^2(\btd; \mathbb{R}^d)} &\leq C \absv{v}_J, \quad \forall v \in \mathcal{H}^1(\Th),
    \end{align}
    where the constants $C>0$ are independent of $h$.
\end{lemma}

\begin{corollary}\label{cor: coercivity-type-bound-discrete-gradient}
    For any $v_h \in V_p$, we have the following estimate
    \begin{equation}
        \norm{\dgrad v_h}^2 \leq C a_h(v_h, v_h),
    \end{equation}
    provided $\eta > 0$ is chosen large enough to ensure coercivity of $a_h(\cdot, \cdot)$.
\end{corollary}

\begin{proof}
    By definition of the discrete gradient and using Lemma~\ref{lem: ern-lift-lemma}, we estimate
    \begin{align*}
        \norm{\dgrad v_h} &\leq \norm{\nabla_h v_h} + \norm{\boldsymbol{\mathcal{L}}_h(\jump{v_h})}, \\
        &\leq \norm{\nabla_h v_h} + C \absv{v_h}_J, \\
        &\leq C \norm{v_h}_{V_p}.
    \end{align*}
    The result follows by coercivity of $a_h(\cdot, \cdot)$ with respect to $\norm{\cdot}_{V_p}$.
\end{proof}

\begin{corollary}\label{cor: discrete-gradient estimates}
    For $i =1,2$, let $u_i^t$ solve the continuous heat equation \eqref{eq: regular-heat} with initial datum $\varphi_i \in H^{p+2}$. Let $u_{i, h}^t$ solve the associated SIP semi-discretisation \eqref{eq: SIP-regular-heat} with initial datum $\mathcal{I}_h \varphi_i$. Then under the assumptions of Theorem~\ref{thm: SIP-parabolic-energy}, there exists $C>0$, independent of $h$, such that the following estimates hold:
    \begin{align}
        \norm{\dgrad u_{i, h} - \nabla u_i}_{L^\infty(0, T; L^2)} &\leq Ch^p \norm{\varphi_i}_{H^{p+2}}, \label{eq: dgrad1}\\
        \norm{\dgrad u_{i, h}}_{L^\infty(0, T; L^2)} &\leq C \norm{\varphi_i}_{H^{p+2}}, \label{eq: dgrad2} \\
        \norm{\dgrad u_{1, h} \cdot \dgrad u_{2, h} - \nabla u_1 \cdot \nabla u_2}_{L^\infty(0, T; L^1)} &\leq Ch^p \norm{\varphi_1}_{H^{p+2}}\norm{\varphi_2}_{H^{p+2}}. \label{eq: dgrad3}
    \end{align}
\end{corollary}

\begin{proof}
    Since $u_i^t$ is smooth, we have $\jump{u_i^t} = 0$ on every facet hence $\nabla u_i^t = \dgrad u_i^t$. Therefore,
    \begin{align*}
    \norm{\dgrad u_{i, h}^t - \nabla u_i^t}_{L^2} &= \norm{\dgrad u_{i, h}^t - \dgrad u_i^t}_{L^2}, \\
    &\leq \norm{\nabla_h u_{i, h}^t - \nabla u_i^t}_{L^2} + \norm{\boldsymbol{\mathcal{L}}_h(\jump{u_{i, h}^t - u_i^t})}, \\
    &\leq \norm{\nabla_h u_{i, h}^t - \nabla u_i^t}_{L^2} + C\absv{u_{i, h}^t - u_i^t}_J, \\
    &\leq Ch^{p} \norm{\varphi_i}_{H^{p+2}},
    \end{align*}
    where we have used Theorem~\ref{thm: SIP-parabolic-energy} in the last inequality.
    Taking the supremum in time proves \eqref{eq: dgrad1}. For the second estimate, observe that
    \begin{align*}
        \norm{\dgrad u_{i, h}^t} &\leq \norm{\dgrad u_{i, h}^t - \nabla u_i^t} + \norm{\nabla u_i^t} \\
        &\leq Ch^p \norm{\varphi_i}_{H^{p+2}} + C \norm{\nabla u_i^t}_{L^\infty},
    \end{align*}
    where we have used \eqref{eq: dgrad1} and the simple embeddings $L^\infty \hookrightarrow L^2$ and $H^{p+2} \hookrightarrow H^{p+1}$. Taking the supremum in time and applying standard maximum principle estimates for the heat equation gives \eqref{eq: dgrad2}. For the third estimate, we write
    \begin{align*}
        &\norm{\dgrad u_{1, h}^t \cdot \dgrad u_{2, h}^t - \nabla u_1^t \cdot \nabla u_2^t}_{L^1} \\
        &\quad\quad\quad\leq \norm{\dgrad u_{1, h}^t \cdot \dgrad u_{2, h}^t - \dgrad u_{1, h}^t \cdot \nabla u_2^t}_{L^1} + \norm{\dgrad u_{1, h}^t \cdot \nabla u_2^t - \nabla u_1^t \cdot \nabla u_2^t}_{L^1}, \\
        &\quad\quad\quad\leq \norm{\dgrad u_{1, h}^t} \norm{\dgrad u_{2, h}^t - \nabla u_2^t} + \norm{\nabla u_2^t} \norm{\dgrad u_{1, h}^t - \nabla u_1^t}, \\
        &\quad\quad\quad\leq Ch^p\norm{\varphi_1}_{H^{p+2}}\norm{\varphi_2}_{H^{p+2}},
    \end{align*}
    where we have used H\"older's inequality, \eqref{eq: dgrad1}, \eqref{eq: dgrad2} and standard maximum principle estimates for the heat equation.
\end{proof}

\section{Proofs for the convergence analysis}\label{app: proofs}

We provide proofs for the various results stated in Section~\ref{sec: main-result}.

\subsection{$L^\infty(L^\infty)$ estimates}

We first require the following three technical lemmas. Lemma~\ref{lem: norm-equivalence} is used to prove Step 3 of Proposition~\ref{prop: Linfty-Linfty} whilst Lemmas~\ref{lem: basis-mass-bounds} and~\ref{lem: sip-bound-technical} are used to prove Step 4 of Proposition~\ref{prop: Linfty-Linfty}.

\begin{lemma}\label{lem: norm-equivalence}
    Fix an element $K \in \Th$. Then there exists a finite set of points $\{\by_i\}_{i=1}^{n_p} \subset \overline{K}$ and a constant $C_{\mathrm{eq}} = C_{\mathrm{eq}}(p, d) > 0$ such that for every $f \in \scalarpolys$
    $$
    \norm{f}_{L^\infty(K)} \leq C_{\mathrm{eq}} \max_{1 \leq i \leq n_p} \absv{f(\by_i)},
    $$
    where $n_p = \mathrm{dim}(\scalarpolys)$.
\end{lemma}

\begin{proof}
    See~\cite[Equation (1)]{bos2018fekete} where we can take $C_{\mathrm{eq}} = \binom{p+d}{d}$.
\end{proof}

\begin{lemma}\label{lem: basis-mass-bounds}
    Let $\Th$ be an affine, shape-regular, quasi-uniform, simplicial mesh. For each element $K \in \Th$, let $\{e_j^K\}$ be the associated nodal basis functions. Then we have the following estimates
    $$
    \norm{e_i^K}_{L^2(K)} \leq Ch^{d/2}, \quad \absv{(M_K^{-1})_{ij}} \leq Ch^{-d},
    $$
    where $(M_K)_{ij} \coloneqq (e_i^K, e_j^K)$ is the mass matrix on element $K$.
\end{lemma}

\begin{proof}
    Let $\widehat{K}$ be the reference element with fixed basis $\{\widehat{e}_i\}$. Since $\Th$ is affine, there exists an affine, invertible map $T_K: \widehat{K} \rightarrow K$ such that $T_K(\widehat{\boldsymbol{x}}) = \boldsymbol{J}_K \widehat{\boldsymbol{x}} + \boldsymbol{b}_K$ for some Jacobian matrix $\boldsymbol{J}_K$ and vector $\boldsymbol{b}_K$. The nodal basis functions are defined via $e^K_i(\bx) = \widehat{e}_i(T_K^{-1}(\bx))$ (see~\cite[Equation 9.5]{ern2021finite}). By a change of variables, and noting that $\boldsymbol{J}_K$ is constant over $\widehat{K}$~\cite[Section 8.1]{ern2021finite}, it follows that
    $$
    \norm{e_i^K}_{L^2(K)} = \absv{\mathrm{det}\boldsymbol{J}_K}^{1/2}\norm{\widehat{e}_i}_{L^2(\widehat{K})}.
    $$
    Using~\cite[Lemma 11.1]{ern2021finite} and the quasi-uniformity of the mesh to estimate $\absv{\mathrm{det}\boldsymbol{J}_K}$, it follows that
    $$
    \norm{e_i^K}_{L^2(K)} \leq Ch^{d/2}\norm{\widehat{e}_i}_{L^2(\widehat{K})}.
    $$
    Finally, since the reference basis functions are fixed, they are uniformly bounded independently of $h$ so their $L^2(\widehat{K})$ norm can be absorbed into the constant giving the first claim. Next, let $\widehat{M}$ denote the mass matrix associated with the reference basis. That is, $\widehat{M}_{ij}\coloneqq (\widehat{e}_i,\widehat{e}_j)_{L^2(\widehat K)}$. 
    Again, using a change of variables it follows that
    $$
    (M_K)_{ij} = \absv{\mathrm{det} \boldsymbol{J}_K} \widehat{M}_{ij}.
    $$
    The reference mass matrix is invertible hence
    $$
    (M_K^{-1})_{ij} = \absv{\mathrm{det} \boldsymbol{J}_K}^{-1} (\widehat{M}^{-1})_{ij}.
    $$
    Again, using~\cite[Lemma 11.1]{ern2021finite} to estimate $\absv{\mathrm{det}\boldsymbol{J}_K}$, it follows that
    $$
    \absv{(M_K^{-1})_{ij}} \leq Ch^{-d} \absv{(\widehat{M}^{-1})_{ij}},
    $$
    and then the result follows since $\widehat{M}^{-1}$ is fixed independent of $h$ so can be absorbed into the constant.
\end{proof}

\begin{lemma}\label{lem: sip-bound-technical}
    Let $a_h(\cdot, \cdot)$ be the SIP bilinear form. For each element $K$, we consider a set of nodal points $\{\bx_i^K\}$ and the associated nodal basis functions $\{e_j^K\}$ such that $e_j^{K'}(\bx_i^K) = \delta_{ij}\delta_{KK'}$. We then have the following estimate
    \begin{equation}
        \sum_{K \in \Th}\sum_{\ell=1}^{n_p} \absv{a_h(v_h, e_\ell^K)} \leq Ch^{-2} \sum_{K \in \Th}\sum_{\ell=1}^{n_p} \absv{(v_h, e_\ell^K)},
    \end{equation}
    for any $v_h \in V_p$.
\end{lemma}

\begin{proof}
    Fix an element $K \in \Th$ and some $\ell \in \{1, \ldots, n_p\}$. Using the Cauchy--Schwarz inequality, an $L^2$-$L^2$ discrete trace inequality~\cite[Lemma 12.8]{ern2021finite}, an $H^1$-$L^2$ inverse inequality~\cite[Lemma 12.1]{ern2021finite}, and the quasi-uniformity of the mesh it follows that
    \begin{align}\label{estimate_a_h}
    \absv{a_h(v_h, e_\ell^K)} \leq Ch^{-2} \norm{e_\ell^K}_{L^2(K)}\sum_{K' \in \mathcal{N}(K)}\norm{v_h}_{L^2(K')},
    \end{align}
    where $\mathcal{N}(K) \coloneqq \{K'\in \Th: \partial K \cap \partial K' \neq \emptyset\}$ is the set consisting of $K$ and its neighbouring elements. Using Lemma~\ref{lem: basis-mass-bounds} we know that 
    $$
    \norm{e_\ell^K}_{L^2(K)} \leq Ch^{d/2},
    $$
    for all $K \in \Th$ and $\ell \in \{1, \ldots, n_p\}$. Now, since $v_h \in V_p$, its restriction to element $K$ admits the following expansion:
    $$
    v_h|_K = \sum_{\ell=1}^{n_p} c_\ell e_\ell^K,
    $$
    for some coefficients $c_\ell$. Taking norms and using the bound on the basis functions implies that
    $$
    \norm{v_h}_{L^2(K)} \leq Ch^{d/2} \sum_{\ell=1}^{n_p} \absv{c_\ell}.
    $$
    Additionally, we can test the expansion of $v_h|_K$ with each basis vector on the element to get that
    $$
    (v_h, e_j^K)_{L^2(K)} = \sum_{\ell=1}^{n_p} c_\ell (e_\ell^K, e_j^K),
    $$
    which upon defining the element-local mass matrix $(M_K)_{j\ell} \coloneqq (e_\ell^K, e_j^K)_{L^2(K)}$ can be rewritten as
    $$
    (v_h, e_j^K)_{L^2(K)} = \sum_{\ell=1}^{n_p} (M_K)_{j \ell} c_\ell.
    $$
    Inverting this relationship gives
    $$
    c_j = \sum_{\ell=1}^{n_p} (M_K^{-1})_{j\ell} (v_h, e_\ell^K)_{L^2(K)}.
    $$
    Using Lemma~\ref{lem: basis-mass-bounds} we know that
    $$
    \absv{(M_K^{-1})_{j\ell}} \leq Ch^{-d}.
    $$
    It follows that
    $$
    \absv{c_\ell} \leq Ch^{-d}\sum_{j=1}^{n_p}\absv{(v_h, e_j^K)_{L^2(K)}},
    $$
    so that
    $$
    \norm{v_h}_{L^2(K)} \leq Ch^{-d/2}n_p\sum_{j=1}^{n_p}\absv{(v_h, e_j^K)_{L^2(K)}}.
    $$
    Plugging this into \eqref{estimate_a_h} gives
    $$
    \absv{a_h(v_h, e_\ell^K)} \leq Ch^{-2} n_p \sum_{K' \in \mathcal{N}(K)}\sum_{j=1}^{n_p}\absv{(v_h, e_j^{K'})_{L^2(K')}},
    $$
    where the $h^{-d/2}$ is cancelled by the $h^{d/2}$ estimate on $\norm{e_\ell^K}_{L^2(K)}$ (Lemma~\ref{lem: basis-mass-bounds}). Summing over each $\ell \in \{1, \ldots, n_p\}$ and $K \in \Th$ and absorbing constants gives the result.
\end{proof}

We can now prove Proposition~\ref{prop: Linfty-Linfty}. Steps 1 and 3 require significant changes from~\cite[Proposition 13]{cornalba2023dean} so we spell them out in full. The remaining steps do not depend on the choice of spatial discretisation and follow the same argument as in~\cite{cornalba2023dean} with only minor changes to the constants. We therefore only state the modifications required in the present setting.

\begin{proof}[Proof of Proposition~\ref{prop: Linfty-Linfty}]
       \textbf{Step 1: energy estimates for test functions.} Fix a point $\bx_0$ in the interior of some element $K$. Define the final datum $\varphi_h^{\bx_0}(\cdot) \in V_p$ by
    $$
    (\varphi_h^{\bx_0}, v_h) = v_h(\bx_0), \quad \forall v_h \in V_p.
    $$
    Choosing $v_h = \varphi_h^{\bx_0}$ gives
    $$
    \norm{\varphi_h^{\bx_0}}^2 = \varphi_h^{\bx_0}(\bx_0) \leq \norm{\varphi_h^{\bx_0}}_{L^\infty(K)} \leq Ch_K^{-d/2}\norm{\varphi_h^{\bx_0}}_{L^2(K)} \leq Ch^{-d/2}\norm{\varphi_h^{\bx_0}},
    $$
    where we have used an $L^\infty$-$L^2$ inverse inequality~\cite[Lemma 12.1]{ern2021finite} and the quasi-uniformity of the mesh. It follows that
    $$
    \norm{\varphi_h^{\bx_0}} \leq Ch^{-d/2}.
    $$
    We now evolve $\varphi_h^{\bx_0}$ by the SIP adjoint heat flow \eqref{eq: SIP-adjoint-heat}. That is, we set $\phi_h(T, \cdot) = \varphi_h^{\bx_0}$ and consider
    $$
    (\partial_t \phi_h^t, v_h) = a_h(\phi_h^t, v_h).
    $$
    Choosing $v_h = \phi_h^t$ and integrating in time gives
    $$
    \norm{\varphi_h^{\bx_0}}^2 - \norm{\phi_h^0}^2 = 2\int_0^T a_h(\phi_h^t, \phi_h^t) \dd t.
    $$
    By Corollary~\ref{cor: coercivity-type-bound-discrete-gradient}
    $$
    \norm{\varphi_h^{\bx_0}}^2 \geq C \int_0^T \norm{\dgrad \phi_h^t}^2 \dd t.
    $$
    Combining this with our upper bound for $\norm{\varphi_h^{\bx_0}}$ gives
    $$
    \int_0^T \norm{\dgrad \phi_h^t}^2 \dd t \leq Ch^{-d},
    $$
    where $C$ depends on the shape regularity of the mesh and on $\eta$ but is independent of $h$. Now suppose the point $\bx_0$ falls on a facet $F \in \Fk$ for some element $K$. We then define the final datum $\varphi_h^{\bx_0}$ via
    $$
    (\varphi_h^{\bx_0}, v_h) = v_h|_K(\bx_0),
    $$
    where the right-hand side corresponds to evaluating the trace of $v_h$ taken from the interior of the element $K$ at the point $\bx_0$. Using a discrete trace inequality~\cite[Lemma 12.8]{ern2021finite}, we estimate
    $$
    \varphi_h^{\bx_0}(\bx_0) \leq \norm{\varphi_h^{\bx_0}}_{L^\infty(F)} \leq Ch_K^{-d/2} \norm{\varphi_h^{\bx_0}}_{L^2(K)},
    $$
    and then the energy estimate follows via the same argument as before.

    \textbf{Step 2: exponentially decaying bounds for a chosen point $\bx_0$.} 
    Define the stopping time
    $$
    T_s \coloneqq \inf \cbrack{t > 0~:~ \sup_{\bx \in \btd}\absv{\rho_h - \E{\rho_h}}(t, \bx) \geq B \frac{\rho_{m, \delta}}{2}},
    $$
    for some arbitrary but fixed $B \geq 1$. Upon replacing $B \rhomin/8$ with $B \rho_{m, \delta}/(8 C_{\mathrm{eq}})$, the following estimate is then identical to that of~\cite[Proposition 13, Step 2]{cornalba2023dean}
    \begin{equation}\label{eq: step-2-final-result}
        \mathbb{P}\brack{T \leq T_s \text{ and } \absv{\rho_h - \E{\rho_h}}(T, \bx_0) \geq B \frac{\rho_{m, \delta}}{8C_{\mathrm{eq}}}} \leq 2 \exp \brack{- \frac{\rho_{m, \delta} B^{1/2}N^{1/2}h^{d/2}}{C \rho_{M, \delta}^{1/2}}},
    \end{equation}
    where $C_{\mathrm{eq}}$ is the norm-equivalence constant in Lemma~\ref{lem: norm-equivalence}. Note that in the right-hand side we have absorbed it into the generic constant $C$.

    \textbf{Step 3: extending the previous estimate to finitely many time points in $[0, T \wedge T_s]$.} We now extend the previous estimate a.e. in space and to all points of the form $ih^\beta \leq T \wedge T_s$ for $i \in \mathbb{N}$. Consider the event $A$ defined by
    $$
    A = \cbrack{\norm{\rho_h(ih^\beta) - \E{\rho_h(ih^\beta)}}_{L^\infty} \geq B\frac{\rho_{m, \delta}}{8} \text{ for some } i \in \mathbb{N} \text{ with } ih^\beta \leq T \wedge T_s}.
    $$
    It follows that 
    $$
    A \subseteq \bigcup_{i \in \mathbb{N}}\bigcup_{K \in \Th}\bigcup_{j = 1}^{n_p} \cbrack{ih^\beta \leq T_s \text{ and } \absv{\rho_h - \E{\rho_h}}(ih^\beta, \by_j^K) \geq B \frac{\rho_{m, \delta}}{8 C_{\mathrm{eq}}}},
    $$
    where we have used Lemma~\ref{lem: norm-equivalence} to estimate the $L^\infty(K)$ norm using the finite set of points $\{\by_j^K\}_{j=1}^{n_p}$ up to a constant $C_{\mathrm{eq}}$. Using union bounds, we estimate
    $$
    \mathbb{P}(A) \leq \sum_{\substack{i \in \mathbb{N} \\ ih^\beta \leq T}}\sum_{K \in \Th}\sum_{j=1}^{n_p} \mathbb{P}\brack{ih^\beta \leq T_s \text{ and } \absv{\rho_h - \E{\rho_h}}(ih^\beta, \by_j^K) \geq B \frac{\rho_{m, \delta}}{8 C_{\mathrm{eq}}}}.
    $$
    We then use the estimate \eqref{eq: step-2-final-result} alongside the observation that there are of order $h^{-d}$ elements in our mesh to get
    \begin{multline}\label{eq: Linfty-Linfty-2}
        \mathbb{P}\brack{\norm{\rho_h(ih^\beta) - \E{\rho_h(ih^\beta)}}_{L^\infty} \geq B \frac{\rho_{m, \delta}}{8} \text{ for some } i \in \mathbb{N} \text{ with } ih^\beta \leq T \wedge T_s} \\
        \leq CTh^{-\beta - d} \exp \brack{- \frac{\rho_{m, \delta}B^{1/2}N^{1/2}h^{d/2}}{C \rho_{M, \delta}^{1/2}}}.
    \end{multline}

    \textbf{Step 4: extending the estimate to all times in $[0, T]$.} This is identical to~\cite[Proposition 13, Step 4]{cornalba2023dean}: upon replacing $\rhomin$ with $\rho_{m, \delta}$ and $\rhomax$ with $\rho_{M, \delta}$, one obtains
    \begin{equation}\label{eq: Linfty-Linfty-3}
        \mathbb{P}\brack{ih^\beta \leq T_s, \sup_{t \in [ih^\beta, (i+1)h^\beta]} \norm{\rho_h(t) - \rho_h(ih^\beta)}_{L^\infty} \geq B \frac{\rho_{m, \delta}}{10}} \leq C \exp \brack{-B^{1/4}h^{-\beta/8}},
    \end{equation}
    for $\beta \geq 6d+8$ and $h$ sufficiently small. We highlight that in this step one requires an $h^{-2}$ estimate on the discrete Laplacian. This is precisely where Lemma~\ref{lem: sip-bound-technical} is used. In the finite-difference setting an analogous estimate is used only with the nodal basis functions here replaced by the canonical basis functions associated with the grid.

    \textbf{Step 5: obtaining \eqref{eq: Linfty-Linfty-1}.} This is identical to~\cite[Proposition 13, Step 5]{cornalba2023dean}: combine \eqref{eq: Linfty-Linfty-2} and \eqref{eq: Linfty-Linfty-3} to get
    \begin{multline*}
    \mathbb{P}\brack{\sup_{t \in [0, T]}\norm{\rho_h(t) - \E{\rho_h(t)}}_{L^\infty} \geq B \frac{\rho_{m, \delta}}{4}} \\ \leq CTh^{-\beta - d} \exp \brack{- \frac{\rho_{m, \delta}B^{1/2}N^{1/2}h^{d/2}}{C \rho_{M, \delta}^{1/2}}} + C \exp \brack{-cB^{1/4}h^{-1}}.
    \end{multline*}
    Upon choosing $h \geq C(d, \delta, \rhomin, \rhomax)N^{-1/d}\absv{\log N}^{2/d}(1+T)$, this implies \eqref{eq: Linfty-Linfty-1}.

    \textbf{Step 6:} 
    This is identical to~\cite[Proposition 13, Step 6]{cornalba2023dean}: upon replacing $\rhomin$ with $\rho_{m, \delta}$ and $\rhomax$ with $\rho_{M, \delta}$ one obtains \eqref{eq: Linfty-Linfty-4} and \eqref{eq: Linfty-Linfty-5}.
\end{proof}

\subsection{Recursive formula for higher moments}

We provide the proof of Proposition~\ref{prop: recursive-formula}. It shares many similarities with~\cite[Proposition 11]{cornalba2023dean} except we frequently use an $L^\infty$-$L^1$ H\"older inequality for the reasons explained in Section~\ref{subsec: proof-outline}.

\begin{proof}[Proof of Proposition~\ref{prop: recursive-formula}]
    Using Lemma~\ref{lem: big-moments-lemma} gives
    \begin{multline*}
    \dd \E{\mathcal{S}_N^{\bj}(\mathcal{I}_h \boldsymbol{\varphi}, T, t)} \\
    = \frac{1}{N}\E{\sum_{k, \ell=1}^M \frac{(j_k - \delta_{k\ell})j_\ell}{2}\mathcal{S}_N^{\bj^{k\ell}}(\mathcal{I}_h \boldsymbol{\varphi}, T, t)(\rho_h^+(t), \dgrad \phi_{k, h}^t \cdot \dgrad \phi_{\ell, h}^t)} \dd t.
    \end{multline*}
    Define
    $$
    r^t_{k, \ell, h} \coloneqq \dgrad\phi_{k, h}^t \cdot \dgrad \phi_{\ell, h}^t - \mathcal{I}_h(\nabla \phi_k^t \cdot \nabla \phi_\ell^t).
    $$
    Upon a similar addition and subtraction of terms as in Lemma~\ref{lem: second-moments}, we obtain
    \begin{multline*}
        \dd \E{\mathcal{S}_N^{\bj}(\mathcal{I}_h \bvphi, T, t)} 
        = N^{-1}\mathbb{E} \left [\sum_{k, \ell=1}^M \frac{(j_k - \delta_{k\ell})j_\ell}{2} \mathcal{S}_N^{\bj^{k\ell}}(\mathcal{I}_h \boldsymbol{\varphi}, T, t) \right . \\
        \left . \times \brack{(\rho_h(t), \mathcal{I}_h(\nabla \phi_k^t \cdot \nabla \phi_\ell^t)) - (\rho_h(0), \mathcal{P}_h^t(\mathcal{I}_h(\nabla \phi_k^t \cdot \nabla \phi_\ell^t)))} \right ] \dd t \\
        + N^{-1} \E{\sum_{k, \ell=1}^M \frac{(j_k - \delta_{k\ell})j_\ell}{2} \mathcal{S}_N^{\bj^{k\ell}}(\mathcal{I}_h \boldsymbol{\varphi}, T, t) (\rho_h(0), \mathcal{P}_h^t(\mathcal{I}_h(\nabla \phi_k^t \cdot \nabla \phi_\ell^t)) ) } \dd t \\
        + N^{-1} \E{ \sum_{k, \ell=1}^M \frac{(j_k - \delta_{k\ell})j_\ell}{2} \mathcal{S}_N^{\bj^{k\ell}}(\mathcal{I}_h \boldsymbol{\varphi}, T, t) (\rho_h^-(t), \dgrad \phi_{k, h}^t \cdot \dgrad \phi_{\ell, h}^t) } \dd t \\
        + N^{-1} \E{\sum_{k, \ell=1}^M \frac{(j_k - \delta_{k\ell})j_\ell}{2} \mathcal{S}_N^{\bj^{k\ell}}(\mathcal{I}_h \boldsymbol{\varphi}, T, t) (\rho_h(t), r_{k, \ell, h}^t) } \dd t =: \sum_{i=1}^4 A_i \dd t.
    \end{multline*}
    Again, using Lemma~\ref{lem: big-moments-lemma} and mimicking the addition and subtraction performed in Lemma~\ref{lem: second-moments} we obtain
    \begin{multline*}
        \dd \E{\mathcal{T}_N^{\bj}(\bvphi, T, t)} = N^{-1} \mathbb{E} \left [\sum_{k, \ell=1}^M \frac{(j_k - \delta_{k\ell})j_\ell}{2} \mathcal{T}_N^{\bj^{k\ell}}(\boldsymbol{\varphi}, T, t) \right . \\
        \left . \times \brack{\frac{1}{N}\sum_{r=1}^N \nabla \phi_k^t(\bB_r(t)) \cdot \nabla \phi_\ell^t(\bB_r(t)) - \frac{1}{N}\sum_{r=1}^N \mathcal{P}^t(\nabla \phi_k^t \cdot \nabla \phi_\ell^t)(\bB_r(0)) } \right ] \dd t \\
        + N^{-1} \E{\sum_{k, \ell=1}^M \frac{(j_k - \delta_{k\ell})j_\ell}{2} \mathcal{T}_N^{\bj^{k\ell}}(\boldsymbol{\varphi}, T, t) \brack{\frac{1}{N}\sum_{r=1}^N \mathcal{P}^t(\nabla \phi_k^t \cdot \nabla \phi_\ell^t)(\bB_r(0))}} \dd t \\
        =: \sum_{i=1}^2 B_i \dd t.
    \end{multline*}
    First notice that we can estimate 
    \begin{multline*}
    \absv{A_1 - B_1} \dd t \leq N^{-1} \sum_{k, \ell=1}^M \frac{(j_k - \delta_{k\ell})j_\ell}{2}\left \lvert \E{\mathcal{S}_N^{\bj^{k\ell}}(\mathcal{I}_h \boldsymbol{\varphi}, T, t)\mathcal{S}_N(\mathcal{I}_h(\nabla \phi_k^t \cdot \nabla \phi_\ell^t), T, t)} \right . \\
    \left .- \E{\mathcal{T}_N^{\bj^{k\ell}}(\boldsymbol{\varphi}, T, t)\mathcal{T}_N(\nabla \phi_k^t \cdot \nabla \phi_\ell^t, T, t)}\right \rvert \dd t,
    \end{multline*}
    where for each pair $(k, \ell)$ the exponent vector $\bj$ is decreased by two units to $\bj^{k \ell}$, but we pick up an additional test function $\nabla \phi_k^t \cdot \nabla \phi_\ell^t$. In particular, $\absv{A_1 - B_1}$ can be estimated in terms of moments of order $j-1$ as follows
    $$
    \int_0^T \absv{A_1 - B_1} \dd t \leq N^{-1} \sum_{k, \ell=1}^M \frac{(j_k - \delta_{k\ell})j_\ell}{2} \int_0^T \mathcal{D}\brack{\{\bj^{k \ell}; 1\}, \{\boldsymbol{\phi}^t; \nabla \phi_k^t \cdot \nabla \phi_\ell^t\}, t} \dd t.
    $$
    To bound $A_3$ we integrate in time, use H\"older's inequality in space, Cauchy--Schwarz in expectation, Corollary~\ref{cor: sup-moments}, Proposition~\ref{prop: Linfty-Linfty}, and Corollary~\ref{cor: discrete-gradient estimates} to get
    \begin{align*}
        &\int_0^T \absv{A_3} \dd t \leq CN^{-1} \int_0^T \sum_{k, \ell=1}^M \frac{(j_k - \delta_{k\ell})j_\ell}{2} \norm{\dgrad \phi_{k, h}^t} \norm{\dgrad \phi_{\ell, h}^t} \\
        &\quad\quad\times \E{\absv{\mathcal{S}_N^{\bj^{k\ell}}(\mathcal{I}_h \boldsymbol{\varphi}, T, t)}^2}^{1/2}\E{\norm{\rho_h^-(t)}_{L^\infty}^2}^{1/2} \dd t \\
        &\leq CN^{-1} T \mathcal{E}(N, h) \sum_{k, \ell=1}^M \frac{(j_k - \delta_{k\ell})j_\ell}{2} \sup_{t \in [0, T]}\norm{\dgrad \phi_{k, h}^t} \sup_{t \in [0, T]}\norm{\dgrad \phi_{\ell, h}^t} \\ 
        &\quad\quad\times (2N^{-1}TC(d, \delta, \rhomin, \rhomax))^{j/2-1}(2j-4)^{3(j-2)} \prod_{m=1}^M \norm{\varphi_m}_{H^{p+2}}^{j_m - \delta_{km} - \delta_{\ell m}} \\
        &\leq \brack{CN^{-1}T C(d, \delta, \rhomin, \rhomax)}^{\frac j 2}(2j)^{3(j-2)} \mathcal{E}(N, h) \sum_{k, \ell=1}^M \frac{(j_k - \delta_{k\ell})j_\ell}{2}\prod_{m=1}^M \norm{\varphi_m}_{H^{p+2}}^{j_m}.
    \end{align*}
    Similarly, using the H\"older and Cauchy--Schwarz inequalities, Corollary~\ref{cor: sup-moments} and Proposition~\ref{prop: Linfty-Linfty} we can estimate $A_4$ as follows
    \begin{align*}
        &\int_0^T \absv{A_4} \dd t \leq N^{-1} \sum_{k, \ell=1}^M \frac{(j_k - \delta_{k\ell})j_\ell}{2} \int_0^T \E{\absv{\mathcal{S}_N^{\bj^{k\ell}}(\mathcal{I}_h \boldsymbol{\varphi}, T, t)} \norm{\rho_h(t)}_{L^\infty}\norm{r^t_{k, \ell, h}}_{L^1}} \dd t \\
        &\leq N^{-1} \sum_{k, \ell=1}^M \frac{(j_k - \delta_{k\ell})j_\ell}{2} \sup_{t \in [0, T]} \norm{r^t_{k, \ell, h}}_{L^1} \\
        &\quad\quad\times\int_0^T \E{\absv{\mathcal{S}_N^{\bj^{k\ell}}(\mathcal{I}_h \boldsymbol{\varphi}, T, t)}^2}^{1/2} \E{\norm{\rho_h(t)}_{L^\infty}^2}^{1/2} \dd t \\
        &\leq CN^{-1} T C(d, \delta, \rhomin, \rhomax) h^p \sum_{k, \ell=1}^M \frac{(j_k - \delta_{k\ell})j_\ell}{2}\norm{\varphi_k}_{H^{p+2}}\norm{\varphi_\ell}_{H^{p+2}} \\
        &\quad\quad\times (2N^{-1}TC(d, \delta, \rhomin, \rhomax))^{j/2-1}(2j-4)^{3(j-2)} \prod_{m=1}^M \norm{\varphi_m}_{H^{p+2}}^{j_m - \delta_{km} - \delta_{\ell m}} \\
        &\leq h^p (C N^{-1}T C(d, \delta, \rhomin, \rhomax))^{j/2}(2j)^{3(j-2)} \sum_{k, \ell=1}^M \frac{(j_k - \delta_{k\ell})j_\ell}{2}\prod_{m=1}^M \norm{\varphi_m}_{H^{p+2}}^{j_m},
    \end{align*}
    where the estimate on $\lVert r_{k, \ell, h}^t\rVert_{L^1}$ is analogous to the bound on $\lVert r_h^t\rVert_{L^1}$ derived in the proof of Lemma~\ref{lem: second-moments}. Next, we rewrite the difference $A_2 - B_2$ as
    \begin{align*}
    (A_2 - B_2)& \dd t = N^{-1} \sum_{k, \ell=1}^M \frac{(j_k - \delta_{k\ell})j_\ell}{2} \brack{\E{\mathcal{S}_N^{\bj^{k\ell}}(\mathcal{I}_h \boldsymbol{\varphi}, T, t)} - \E{\mathcal{T}_N^{\bj^{k\ell}}( \boldsymbol{\varphi}, T, t)}} \\
    &\quad\quad\times (\rho_h(0), \mathcal{P}_h^t(\mathcal{I}_h(\nabla \phi_k^t \cdot \nabla \phi_\ell^t))) \dd t \\
    &-N^{-1}\sum_{k, \ell=1}^M \frac{(j_k - \delta_{k\ell})j_\ell}{2} \E{\mathcal{T}_N^{\bj^{k\ell}}( \boldsymbol{\varphi}, T, t)} \\
    &\quad\quad\times \brack{\frac{1}{N} \sum_{r=1}^N \mathcal{P}^t(\nabla \phi_k^t \cdot \nabla \phi_\ell^t)(\bB_r(0)) - (\rho_h(0), \mathcal{P}_h^t(\mathcal{I}_h(\nabla \phi_k^t \cdot \nabla\phi_\ell^t))) } \dd t \\
    &=: \Theta_1 + \Theta_2,
    \end{align*}
    where we have added and subtracted a suitable term in $\mathcal{T}_N(\boldsymbol{\varphi}, T, t)$ and used Assumption~\ref{ass: dG1} that the initial particle positions $\{\bB_r(0)\}$ and initial density $\rho_h(0)$ are deterministic. To deal with $\Theta_1$, we can use estimates of order $j-2$ since for each pair $(k, \ell)$, the exponent vector is decreased by two units to $\bj^{k\ell}$. Moreover, using $\absv{\rho_h(0)} \leq \rho_{M, \delta}$ from \eqref{eq: rho0h-two-sided}, we can estimate
    $$
    \absv{(\rho_h(0), \mathcal{P}_h^t(\mathcal{I}_h(\nabla \phi_k^t \cdot \nabla \phi_\ell^t)))} \leq \rho_{M, \delta} \norm{\varphi_k}_{H^{p+2}}\norm{\varphi_\ell}_{H^{p+2}}.
    $$
    Put together, we get
    $$
    \int_0^T \absv{\Theta_1} \dd t \leq N^{-1}\rho_{M, \delta} \sum_{k, \ell=1}^M \frac{(j_k - \delta_{k\ell})j_\ell}{2} \norm{\varphi_k}_{H^{p+2}}\norm{\varphi_\ell}_{H^{p+2}} \int_0^T \mathcal{D}\brack{\bj^{k\ell}, \boldsymbol{\phi}^t, t} \dd t.
    $$
    To handle $\Theta_2$, we first apply Corollary~\ref{cor: sup-moments} to get
    \begin{multline*}
    \int_0^T \absv{\Theta_2} \dd t \leq N^{-1} \sum_{k, \ell=1}^M \frac{(j_k - \delta_{k\ell})j_\ell}{2} (N^{-1}T)^{(j-2)/2}j^{j-2} \prod_{m=1}^M \norm{\nabla \varphi_m}_{L^\infty}^{j_m - \delta_{km} - \delta_{\ell m}} \\
    \times \int_0^T \absv{\frac{1}{N} \sum_{r=1}^N \mathcal{P}^t(\nabla \phi_k^t \cdot \nabla \phi_\ell^t)(\bB_r(0)) - (\rho_h(0), \mathcal{P}_h^t(\mathcal{I}_h(\nabla \phi_k^t \cdot \nabla \phi_\ell^t))) } \dd t.
    \end{multline*}
    We then add and subtract $(\rho_h(0), \mathcal{P}^t(\nabla \phi_k^t \cdot \nabla \phi_\ell^t))$ in the final bracket term and apply the triangle inequality to split it as
    \begin{multline*}
        \absv{\frac{1}{N} \sum_{r=1}^N \mathcal{P}^t(\nabla \phi_k^t \cdot \nabla \phi_\ell^t)(\bB_r(0)) - (\rho_h(0), \mathcal{P}_h^t(\mathcal{I}_h(\nabla \phi_k^t \cdot \nabla \phi_\ell^t))) } \\
        \leq \absv{\frac{1}{N} \sum_{r=1}^N \mathcal{P}^t(\nabla \phi_k^t \cdot \nabla \phi_\ell^t)(\bB_r(0)) - (\rho_h(0), \mathcal{P}^t(\nabla \phi_k^t \cdot \nabla \phi_\ell^t))} \\
        + \absv{(\rho_h(0), \mathcal{P}^t(\nabla \phi_k^t \cdot \nabla \phi_\ell^t)) - (\rho_h(0), \mathcal{P}_h^t(\mathcal{I}_h(\nabla \phi_k^t \cdot \nabla\phi_\ell^t)))}.
    \end{multline*}
    The first term is settled by Assumption~\ref{ass: dG1} and the stability of the (adjoint) heat semigroup. To estimate the second term we use Theorem~\ref{thm: SIP-parabolic-L2}. Altogether,
    \begin{multline*}
    \int_0^T \absv{\Theta_2} \dd t \\
    \leq h^{p} (CN^{-1}T)^{j/2}j^{j-2} \sum_{k, \ell=1}^M \frac{(j_k - \delta_{k\ell})j_\ell}{2} \prod_{m=1}^M \brack{\norm{\varphi_m}_{W^{p+2, \infty}}+\norm{\varphi_m}_{H^{p+3}}}^{j_m}.
    \end{multline*}
    The estimates for $\absv{\Theta_2}$ and $\absv{A_4}$ are then combined (up to constants) since they are both numerical errors of order $O(h^p)$.
\end{proof} 

\section{Construction of initial data}\label{app: construction-of-initial-data}

We show one possible construction of initial data $\rho_{0, h}$ and $\{\bB_i(0)\}_{i=1}^N$ satisfying Assumption~\ref{ass: dG1} given $\overline{\rho}_0 \in W^{p+1, \infty} \cap H^{p+2}$ such that $\int_{\btd}\overline{\rho}_0(\bx) \dd \bx = 1$, $\rhomin > 0$, and $\rhomax < \infty$. The construction is valid for $p=1, 2$.

\textbf{Step 1: constructing $\rho_{0, h}$.} Define 
$$
\rho_{0, h} \coloneqq \mathcal{I}_h \overline{\rho}_0 + c_h(\overline{\rho}_0), \quad c_h(\overline{\rho}_0) \coloneqq \frac{1}{\absv{\btd}}\int_{\btd}\brack{\overline{\rho}_0 - \mathcal{I}_h \overline{\rho}_0} \dd \bx.
$$
In particular, $\int \rho_{0, h} \dd \bx = \int \overline{\rho}_0 \dd \bx = 1$ so that the initial condition conserves mass. We next show that $\rho_{0, h}$ is positive and therefore a probability density.

\begin{lemma}
    There exists $h_0 > 0$ and $0 < \rho_* < \rho^* < \infty$ such that for $h \leq h_0$ we have $\rho_* < \rho_{0, h} < \rho^*$.
\end{lemma}

\begin{proof}
    We first note that 
    $$
    \norm{\overline{\rho}_0 - \rho_{0, h}}_{L^\infty} \leq \norm{\overline{\rho}_0 - \mathcal{I}_h \overline{\rho}_0}_{L^\infty} + \absv{c_h} \leq 2 \norm{\overline{\rho}_0 - \mathcal{I}_h \overline{\rho}_0}_{L^\infty} \leq Ch^{p+1}\norm{\overline{\rho}_0}_{W^{p+1, \infty}},
    $$
    using the standard interpolation estimate~\cite[Theorem 11.13]{ern2021finite}. For any $K \in \Th$ and $\bx \in K$, we decompose $\rho_{0, h}(\bx) = \overline{\rho}_0(\bx) - (\overline{\rho}_0(\bx) - \rho_{0, h}(\bx))$. We then estimate
    $$
    \rho_{0, h}(\bx) \geq \min_{\bx \in \btd} \overline{\rho}_0 - \norm{\overline{\rho}_0 - \rho_{0, h}}_{L^\infty(K)} \geq \rhomin - Ch^{p+1}\norm{\overline{\rho}_0}_{W^{p+1, \infty}}.
    $$
    One can then choose $h_0 > 0$ such that $Ch^{p+1}\norm{\overline{\rho}_0}_{W^{p+1, \infty}} \leq \frac{1}{2} \rhomin$. The lower bound then follows with $\rho_* \coloneqq \frac{1}{2} \rhomin$. The argument for $\rho^*$ is analogous.
\end{proof}

Next observe that $\norm{c_h} = \absv{c_h}\absv{\btd}^{1/2} \leq \norm{\overline{\rho}_0 - \mathcal{I}_h \overline{\rho}_0}$ so that another application of a standard interpolation estimate~\cite[Theorem 11.13]{ern2021finite}, this time in $L^2$, gives
$$
\norm{\rho_{0, h} - \overline{\rho}_0}_{L^2} \leq Ch^{p+1}\norm{\overline{\rho}_0}_{H^{p+1}}.
$$
We have now satisfied (1), (3a) and (3c) of Assumption~\ref{ass: dG1}. It remains to suitably place the deterministic particles to satisfy (2) and (3b).

\textbf{Step 2: placing the particles.} We first require the following Lemma:
\begin{lemma}
    Let $\overline{\rho}_0$ be a probability density with $0 < \rhomin \leq \overline{\rho}_0 \leq \rhomax < \infty$. For every $N \in \mathbb{N}$ there exists a partition of $\btd$ into axis-parallel rectangles $\{R_i\}_{i=1}^N$ such that
    $$
    \int_{R_i} \overline{\rho}_0(\bx) \dd \bx = \frac{1}{N}, \quad \mathrm{diam}(R_i) \leq CN^{-1/d},
    $$
    where $C$ depends only on $d$ and $\rhomax/\rhomin$, and in particular is independent of $N$.
\end{lemma}

\begin{proof}
    Write $a \asymp b$ to mean the existence of constants $c, C > 0$ such that $cb \leq a \leq Cb$. Without loss of generality, identify $\btd$ with the periodic cell $[0, 1]^d$. We argue by induction over dimension. When $d=1$, choose points $0=s_0 < s_1 < \cdots < s_N = 1$ such that
    $$
    \int_{s_{i-1}}^{s_i} \overline{\rho}_0(\bx) \dd \bx = \frac{1}{N}, \quad i=1, \ldots, N.
    $$
    Setting $R_i = [s_{i-1}, s_i]$ gives the desired partition since 
    $$
    \mathrm{diam}(R_i) = \absv{R_i} \leq \rhomin^{-1}N^{-1}.
    $$
    Now assume the result holds in dimension $d-1$. Let $L = \lfloor N^{1/d}\rfloor$. Then $L \asymp N^{1/d}$ and we can choose integers $n_1, \ldots, n_L \asymp N^{(d-1)/d}$ such that $\sum_{\ell=1}^L n_\ell = N$. Take the first coordinate and partition it into slabs $S_\ell \coloneqq [s_{\ell-1}, s_\ell] \times [0, 1]^{d-1}$ for $0 = s_0 < s_1 < \ldots < s_L = 1$ such that $\int_{S_\ell} \overline{\rho}_0 \dd \bx = n_\ell / N$. Note that $s_{\ell} - s_{\ell-1} \asymp N^{-1/d}$. We then define
    $$
    \rho_\ell(\bx') \coloneqq \frac{N}{n_\ell}\int_{s_{\ell-1}}^{s_\ell}\overline{\rho}_0(x_1, \bx') \dd x_1,
    $$
    which are probability densities on $[0, 1]^{d-1}$. By the boundedness of $\overline{\rho}_0$, it follows that the densities $\rho_\ell$ are bounded from above and below (strictly away from zero, and independently of $N, \ell$). We therefore apply the inductive hypothesis to deduce the existence of rectangles $\{R_{\ell, j}'\}_{j=1}^{n_\ell}$ partitioning $[0, 1]^{d-1}$ such that
    $$
    \int_{R_{\ell, j}'} \rho_\ell(\bx') \dd \bx' = \frac{1}{n_\ell}, \quad \mathrm{diam}(R_{\ell, j}') \lesssim n_{\ell}^{-1/(d-1)} \lesssim N^{-1/d}.
    $$ Therefore, the rectangles $R_{\ell, j} \coloneqq [s_{\ell-1}, s_\ell] \times R_{\ell, j}'$ give the required partition.
\end{proof}

We now construct the initial Brownian particle placements.

\begin{lemma}\label{lem: appE-2}
    For any $N \in \mathbb{N}$, there exists a deterministic set of Brownian particles $\{\bB_i(0)\}_{i=1}^N$ such that
    $$
    \absv{\frac{1}{N}\sum_{i=1}^N v(\bB_i(0)) - \int_{\btd}v(\bx) \overline{\rho}_0(\bx) \dd \bx} \lesssim N^{-2/d}\norm{v}_{W^{2, \infty}},
    $$
    holds for all $v \in W^{2, \infty}$.
\end{lemma}

\begin{proof}

For each $i = 1, \ldots, N$ we then place the initial Brownian motions according to
$$
\bB_i(0) \coloneqq N\int_{R_i} \bx \overline{\rho}_0(\bx) \dd \bx.
$$
That is, $\bB_i(0)$ is the $\overline{\rho}_0$--weighted barycentre of $R_i$. It follows that
$$
\int_{R_i} (\bx - \bB_i(0))\overline{\rho}_0(\bx) \dd \bx = 0.
$$
Next, we Taylor expand $v \in W^{2, \infty}$ around $\bB_i(0)$ to get
$$
v(\bx) = v(\bB_i(0)) + \nabla v(\bB_i(0)) \cdot (\bx - \bB_i(0)) + r_i(\bx),
$$
where the remainder $r_i(\bx)$ satisfies
$$
\absv{r_i(\bx)} \lesssim \absv{\bx - \bB_i(0)}^2 \norm{v}_{W^{2, \infty}}.  
$$
Multiplying by $\overline{\rho}_0(\bx)$ and integrating over $R_i$ gives the estimate
\begin{align*}
    \absv{\frac{1}{N} v(\bB_i(0)) - \int_{R_i}v(\bx) \overline{\rho}_0(\bx) \dd \bx} &\lesssim \norm{v}_{W^{2, \infty}} \int_{R_i} \absv{\bx - \bB_i(0)}^2\overline{\rho}_0(\bx) \dd \bx, \\
    &\lesssim \mathrm{diam}(R_i)^2 \norm{v}_{W^{2, \infty}} \int_{R_i} \overline{\rho}_0(\bx) \dd \bx, \\
    &\lesssim N^{-1-2/d}\norm{v}_{W^{2, \infty}}.
\end{align*}
Summing over $i = 1, \ldots, N$ gives the result.
\end{proof}

Finally, it remains to estimate the weak error appearing in Assumption~\ref{ass: dG1}. We write
$$
\absv{\frac{1}{N}\sum_{i=1}^N v(\bB_i(0)) - (\rho_{0, h}, \mathcal{I}_h v)} \leq \absv{\frac{1}{N}\sum_{i=1}^N v(\bB_i(0)) - (\overline{\rho}_0, v)} + \absv{(\overline{\rho}_0, v) - (\rho_{0, h}, \mathcal{I}_h v)}.
$$
The first term is settled using Lemma~\ref{lem: appE-2}, noting that $\norm{v}_{W^{2, \infty}} \leq \norm{v}_{W^{p+1, \infty}}$ since $p \geq 1$. For the second term we add and subtract $(\overline{\rho}_0, \mathcal{I}_h v)$. It is then settled using the standard interpolation estimate~\cite[Theorem 11.13]{ern2021finite} and the $L^\infty$ stability of $\mathcal{I}_h$ to give a bound of order $O(h^{p+1})$. Altogether
$$
\absv{\frac{1}{N}\sum_{i=1}^N v(\bB_i(0)) - (\rho_{0, h}, \mathcal{I}_h v)} \lesssim \brack{h^{p+1} + N^{-2/d}}\norm{v}_{W^{p+1, \infty}}.
$$
For $p \in \{1, 2\}$ and the scaling $Nh^d \gtrsim 1$ implied by Assumption~\ref{ass: dG2}, it follows that $h^{p+1} + N^{-2/d} \lesssim h^{p+1} + h^2 \lesssim h^{p}$, and the claim follows.
\begin{remark}[Higher polynomial degrees]\label{rem: higher-order-initial-data}
    The restriction to $p \in \{1,2\}$ is structural. A single particle in a cell
    of $\overline{\rho}_0$-mass $N^{-1}$ reproduces that mass and, via the
    barycentre, the first moment; the rule is exact on $P_1$ and no better,
    giving $O(N^{-2/d})$, which meets the $O(h^p)$ of
    Assumption~\ref{ass: dG1}(3b) only for $p \leq 2$.

    For general $p$ one places $q$ particles per cell so that the rule is exact
    on $P_p$. Since every particle carries mass $N^{-1}$, this requires an
    \emph{equal-weight} cubature rule for $\overline{\rho}_0 \dd \bx$ on the
    cell. Kane~\cite[Theorem 4]{kane2015small} guarantees such a rule for every
    $q$ above an explicit threshold determined by the dimension of the mean-zero
    subspace of $P_p$ and by the extremal ratio
    $\sup f / \absv{\inf f}$ over that subspace; both are bounded in terms of
    $p$, $d$ and $\max_{\btd}\overline{\rho}_0/\min_{\btd}\overline{\rho}_0$
    alone. Applied on a partition into $\lfloor N/q \rfloor$ cells of
    $\overline{\rho}_0$-mass $n_j/N$ with $q \leq n_j \leq 2q-1$, this gives
    $$
    \absv{\frac{1}{N}\sum_{i=1}^N v(\bB_i(0)) - \int_{\btd} v \overline{\rho}_0 \dd \bx}
    \lesssim N^{-(p+1)/d}\norm{v}_{W^{p+1,\infty}}
    \lesssim h^{p+1}\norm{v}_{W^{p+1,\infty}},
    $$
    so Assumption~\ref{ass: dG1}(3b) holds at every degree. We do not pursue
    this: Kane's argument is non-constructive, whereas the barycentres of
    Lemma~\ref{lem: appE-2} are explicit and cover the degrees used in our
    experiments.
\end{remark}

\end{document}